\documentclass{amsart}
\pdfoutput 1
\usepackage[all,hyperref,numberbysection]{bi-discrete}
\usepackage{amsmath}
\usepackage{upgreek}
\usepackage{microtype}
\numberwithin{equation}{section}

\usepackage[backend=biber,maxbibnames=5,maxalphanames=5,style=alphabetic,bibencoding=utf8,giveninits,url=false,isbn=false]{biblatex}
\AtBeginBibliography{\small}
\allowdisplaybreaks

\newcommand{\dxyn}{\frac{dx\, dy}{\abs{x-y}^n}}

\begin{document}

\author[Diening]{Lars Diening}
\address{Fakult\"at f\"ur Mathematik, Universit\"at Bielefeld, Postfach 100131, D-33501 Bielefeld, Germany}
\email{lars.diening@uni-bielefeld.de}
\author[Kim]{Kyeongbae Kim}
\address{Department of Mathematical Sciences, Seoul National University, Seoul 08826, Korea}
\email{kkba6611@snu.ac.kr}
\author[Lee]{Ho-Sik Lee}
\address{Fakult\"at f\"ur Mathematik, Universit\"at Bielefeld, Postfach 100131, D-33501 Bielefeld, Germany}
\email{ho-sik.lee@uni-bielefeld.de}
\author[Nowak]{Simon Nowak}
\address{Fakult\"at f\"ur Mathematik, Universit\"at Bielefeld, Postfach 100131, D-33501 Bielefeld, Germany}
\email{simon.nowak@uni-bielefeld.de}

\makeatletter
\@namedef{subjclassname@2020}{\textup{2020} Mathematics Subject Classification}
\makeatother

\subjclass[2020]{Primary: 35R09, 35B65; Secondary: 47G20, 35J60}

\keywords{nonlocal equations, fractional $p$-Laplacian, regularity}
\thanks{Lars Diening and Simon Nowak gratefully acknowledge funding by the Deutsche Forschungsgemeinschaft (DFG, German Research Foundation) - SFB 1283/2 2021 - 317210226. Kyeongbae Kim thanks for funding of the National Research Foundation of Korea (NRF) through IRTG 2235/NRF-2016K2A9A2A13003815 at Seoul National University. Ho-Sik Lee thanks for funding of the Deutsche Forschungsgemeinschaft through GRK 2235/2 2021 - 282638148.
}

\title{Higher differentiability for the fractional $p$-Laplacian}

\begin{abstract}
	In this work, we study the higher differentiability of solutions to the inhomogeneous fractional $p$-Laplace equation under different regularity assumptions on the data. In the superquadratic case, we extend and sharpen several previous results, while in the subquadratic regime our results constitute completely novel developments even in the homogeneous case. In particular, in the local limit our results are consistent with well-known higher differentiability results for the standard inhomogeneous $p$-Laplace equation.
	All of our main results remain valid in the vectorial context of fractional $p$-Laplace systems.
\end{abstract}

\maketitle

\tableofcontents

%% ------------------------------------------------------------

\section{Introduction}

\subsection{Aim and scope}
This paper is devoted to the higher differentiability of solutions to fractional $p$-Laplace equations of the type
\begin{equation}
	\label{pt : eq.main}
	(-\Delta_p)^s u=f\quad\text{in }\Omega \subset \mathbb{R}^n.
\end{equation}
Here $n \geq 1$ and the \emph{fractional $p$-Laplacian} $(-\Delta_p)^s$ is formally defined by
\begin{equation} \label{spLap}
	(-\Delta_p)^s u(x)=(1-s)\, \mathrm{P}.\mathrm{V}.\int_{\RRn} \frac{|u(x)-u(y)|^{p-2} (u(x)-u(y))}{|x-y|^{n+sp}}\, dy,
\end{equation}
where $s \in (0,1)$ and $p \in (1,\infty)$.

Equation \eqref{pt : eq.main} arises naturally in the calculus of variations as the Euler-Lagrange equation of the inhomogeneous $W^{s,p}$ energy
\begin{equation} \label{energy}
	u \mapsto (1-s) \, \frac{1}{p} \iint_{(\Omega^c \times \Omega^c)^c} \frac{|u(x)-u(y)|^p}{|x-y|^{n+sp}} dx\,dy - \int_{\Omega} fu \, dx.
\end{equation}

%For the sake of simplicity, let us begin our discussion by 
The factor $(1-s)$ appearing in \eqref{spLap} and \eqref{energy} guarantees that the energy \eqref{energy} converges to the standard inhomogeneous $p$-Dirichlet energy
\begin{equation} \label{locEnergy}
	u \mapsto \frac{1}{p} \int_{\Omega} |\nabla u|^p \, dx - \int_{\Omega} fu \, dx
\end{equation}
as $s \to 1$, see for instance \cite{BBM1}. Since the Euler-Lagrange equation of the functional \eqref{locEnergy} is given by the standard inhomogeneous $p$-Laplace equation
\begin{equation} \label{pLap}
	-\Delta_p u:=-\textnormal{div} (|\nabla u|^{p-2} \nabla u)=f \quad\text{in }\Omega \subset \mathbb{R}^n,
\end{equation}
the nonlocal operator \eqref{spLap} can indeed be considered to be a fractional analogue of the classical $p$-Laplacian $-\Delta_p$.

%For the sake of simplicity, let us begin our discussion concerning higher differentiability in the homogeneous case when $f=0$. 

In order to provide some context and motivation, let us review some known higher differentiability results for the local $p$-Laplacian. %For the sake of simplicity, let us begin our discussion in the homogeneous case when $f=0$. 
%For the local $p$-Laplace equation, in this case the following fundamental higher differentiability results are known.
%In the superquadratic case when $p \geq 2$, a fundamental observation of Uhlenbeck (see \cite[Lemma 3.1]{Uhl77}) is that
Indeed, if $u \in W^{1,p}(\Omega)$ is a weak solution to the inhomogeneous $p$-Laplace equation \eqref{pLap}, we have the following implications
\begin{equation} \label{locreg}
	f \in L^{p^\prime}_{\loc}(\Omega) \implies \begin{cases}
		u \in W^{\alpha,p}_{\loc}(\Omega) \text{ } \forall \alpha \in \left (0, p^\prime \right ), \quad &\textnormal{if } p \in (2,\infty), \\
		u \in W^{2,p}_{\loc}(\Omega), \quad & \textnormal{if } p \in (1,2],
	\end{cases}
\end{equation}
see \cite{Simon77,Simon81} for the case $p>2$ and \cite[Remark 1.5]{BrascoSant} and \cite{Thelin} for the case $p \leq 2$. Here as usual $p^\prime:=\frac{p}{p-1}$ denotes the H\"older conjugate of $p$. Moreover, for $p>2$ the above conclusion can be improved if $f$ is sufficiently differentiable in the sense of the following implication
\begin{equation} \label{pg2loc}
	f \in W^{t,p^\prime}_{\loc}(\Omega) \textnormal{ for some } t > \frac{p-2}{p} \implies u \in W^{\alpha,p}_{\loc}(\Omega) \text{ } \forall \alpha \in \left (0, \frac{p+2}{p}\right ). 
\end{equation}
For $f=0$, this regularity result goes back to Uhlenbeck (see \cite[Lemma 3.1]{Uhl77}), while for a general right-hand side $f$ the implication \eqref{pg2loc} was proved in \cite[Theorem 1.1]{BrascoSant}. For various further developments concerning gradient differentiability of solutions to \eqref{pLap}, we refer for instance to \cite{MinCZMD,KuuMin12,AKMARMA,BCDKS,CM1,BCDM,BDW,Donglocal}.

A main goal of the present paper is to provide analogues of the implications \eqref{locreg}-\eqref{pg2loc} in the case of the inhomogeneous fractional $p$-Laplace equation \eqref{pt : eq.main}, which are in particular consistent with the results from the local setting as $s \to 1$. Moreover, in the subquadratic regime when $p \in (1,2)$ our results are already new in the homogeneous case when $f=0$. In fact, in this case for the whole range $s \in (0,1)$ we in particular prove that the gradient of any weak solution to $(-\Delta_p)^s u=0$ exists in $L^p$, which in contrast to classical $p$-harmonic functions is a priori not known for $(s,p)$-harmonic functions and is instead a highly nontrivial observation.

\subsection{Setup}

Before stating our main results, we need to fix our setup more rigorously. First of all, in order to control the growth of solutions at infinity, for $\beta >0$ and $q \in (0,\infty)$ we consider the tail spaces
$$L^q_{\beta}(\mathbb{R}^n):= \left \{u \in L^1_{\loc}(\mathbb{R}^n) \mathrel{\Big|} \int_{\mathbb{R}^n} \frac{|u(y)|^q}{(1+|y|)^{n+\beta}}dy < \infty \right \}$$
introduced in \cite{existence}.
We remark that a function $u \in L^1_{\loc}(\mathbb{R}^n)$ belongs to the space $L^{p-1}_{sp}(\mathbb{R}^n)$ if and only if the \emph{nonlocal tails} of $u$ given by
\begin{equation*}
	\mathrm{Tail}(u;B_{r}(x_{0}))=\left((1-s)r^{sp}\int_{\bbR^{n}\setminus B_{r}(x_{0})}\frac{|u(y)|^{p-1}}{|y-x_{0}|^{n+sp}}\,dy\right)^{\frac{1}{p-1}}
\end{equation*}
are finite for all $r>0$, and $x_0 \in \mathbb{R}^n$.

For notational convenience, we also use the following nonlocal excess functional.
Indeed, for any $p \in (1,\infty)$, $r>0$, $x_0 \in \mathbb{R}^n$ and $u\in  L_{sp}^{p-1}(\bbR^{n})$,
\begin{align*}
	E(u;B_r(x_0))&=\left(\dashint_{B_r(x_0)}|u-(u)_{B_{r}(x_0)}|^{p}\,dx\right)^{\frac{1}{p}}\\
	&\quad+\mathrm{Tail}(u-(u)_{B_r(x_0)};B_r(x_0)).
\end{align*}

We now define weak solutions to \eqref{pt : eq.main}.
\begin{definition}[Weak solutions] \label{def:weaksol}
	Let $\Omega \subset \mathbb{R}^n$ be an open set, $s \in (0,1)$ and $p \in (1,\infty)$. Moreover, assume that $f \in L^{p^\prime}_{\loc}(\Omega)$. We say that $u \in W^{s,p}_{\loc}(\Omega) \cap L^{p-1}_{sp}(\mathbb{R}^n)$ is a weak solution of \eqref{pt : eq.main}, if for any $\psi\in W^{s,p}(\Omega)$ with compact support in $\Omega$,
	\begin{align*}
		&(1-s)\int_{\bbR^{n}}\int_{\bbR^{n}}\frac{|u(x)-u(y)|^{p-2}(u(x)-u(y))}{|x-y|^{n+sp}}(\psi(x)-\psi(y))\,dx\,dy=\int_{\Omega}f\psi\,dx.
	\end{align*}
\end{definition}

\subsection{Main results}
Throughout this section, we fix some arbitrary parameter $s \in (0,1)$.
Our first main result yields sharp higher differentiability in the case when the data belongs to $L^{p^\prime}$ and therefore might not be differentiable.

\begin{theorem}[Non-differentiable data]\label{thm.1}
	Let $u\in W^{s,p}_{\mathrm{loc}}(\Omega)\cap L^{p-1}_{sp} (\bbR^{n})$ be a weak solution to \eqref{pt : eq.main} with $f\in L^{p'}_{\mathrm{loc}}(\Omega)$. Fix $s_0\in(0,s]$ and denote 
	\begin{equation}\label{defn.alpha0}
		\begin{aligned}
			\alpha_0=\begin{cases}
				s_0p'&\quad\text{if }p\geq2,\\
				s_0p'&\quad\text{if }p\leq2 \text{ with }s_0p'\leq 1,\\
				2+(s_0-1)p&\quad\text{if }p\leq2\text{ with }s_0p'>1.
			\end{cases}
		\end{aligned}
	\end{equation}
	Then the following holds.
	
	\textbf{In case of $s_0p'\leq 1$: }For any $\alpha\in(0,\alpha_0)$, we have $u\in W^{\alpha,p}_{\mathrm{loc}}(\Omega)$. Moreover,
	\begin{align*}
		r^{\alpha-\frac{n}{p}}[u]_{W^{\alpha,p}(B_r(x_0))}\leq cE(u;B_{2r}(x_0))+cr^{\frac{sp}{p-1}-\frac{n}{p}}\|f\|^{\frac{1}{p-1}}_{L^{p'}(B_{2r}(x_0))}
	\end{align*}
	holds for some constant $c=c(n,s_0,p,\alpha)$, provided that $B_{2r}(x_0)\Subset \Omega$. 
	
	\textbf{In case of $s_0p'>1$: }For any $\alpha\in(1,\alpha_0)$, we have $\nabla u \in W^{\alpha-1,p}_{\mathrm{loc}}(\Omega)$. Moreover,
	\begin{align*}
		r^{\alpha-\frac{n}{p}}[\nabla u]_{W^{\alpha-1,p}(B_r(x_0))}\leq cE(u;B_{2r}(x_0))+cr^{\frac{sp}{p-1}-\frac{n}{p}}\|f\|^{\frac{1}{p-1}}_{L^{p'}(B_{2r}(x_0))}
	\end{align*}
	holds for some constant $c=c(n,s_0,p,\alpha)$, provided that $B_{2r}(x_0)\Subset \Omega$.
\end{theorem}

\begin{remark}[Stability] \normalfont
	Due to the presence of the additional parameter $s_0$, the constants in Theorem \ref{thm.1} and also in all of our further main results below do not blow up in the local limit $s \to 1^-$ for $\alpha$ fixed.
\end{remark}

\begin{remark}[Sharpness, non-differentiable data]\label{rmk.sharp} \normalfont
For any $\epsilon>0$, we construct a weak solution $u\notin W^{sp'+\epsilon,p}_{\mathrm{loc}}$ to \eqref{pt : eq.main} with $f\in L^{p'}_{\mathrm{loc}}$. By Appendix \ref{app.power}, we observe 
	\begin{align*}
		u(x)=|x|^{sp'+\frac{\epsilon}{2}-\frac{n}{p}}\in \left(W^{s,p}_{\mathrm{loc}}(\bbR^n)\cap L^{p-1}_{sp}(\bbR^n)\right)\setminus W^{sp'+\epsilon,p}(B_1).
	\end{align*}
	By \eqref{eq11} in Appendix \ref{app.power}, we see that $u$ is a weak solution to 
	\begin{align*}
 		(-\Delta_p)^su=f \quad\text{in }B_1,
	\end{align*}
	where
	\begin{align*}
		f(x)=c|x|^{\frac{\epsilon(p-1)}{2}-\frac{n(p-1)}{p}}
	\end{align*}
	for some constant $c=c(n,s,p)$, which implies $f\in L^{p'}_{\mathrm{loc}}(\bbR^n)$. In this sense, our estimate in Theorem \ref{thm.1} is sharp for $s_0=s$ in the case when $p\geq2$ and also in the case when $p\leq2 \text{ with }sp'\leq 1$, since in both cases we have $\alpha_0=sp'$.
\end{remark}

By using bootstrap arguments, we are able to improve differentiability of the solution when the right-hand side is differentiable. Before stating the results, for convenience of notation for any $f\in W^{t,p'}(B_R)$ with $t\in(0,1)$ let us write
\begin{align}\label{not.norm}
    \|f\|_{\widetilde{W}^{t,p'}(B_{R})} \coloneqq [f]_{{W}^{t,p'}(B_{R})}+R^{-t}\|f\|_{L^p(B_{R})}.
\end{align}
We start with the superquadratic case when $p\geq 2$. 
\begin{theorem}[Differentiable data, superquadratic case]\label{thm.regd}
	Let $p\geq2$ and $u\in W^{s,p}_{\mathrm{loc}}(\Omega)\cap L^{p-1}_{sp} (\bbR^{n})$ be a weak solution to \eqref{pt : eq.main} with $f\in W^{t,p'}_{\mathrm{loc}}(\Omega)$, where $t\in\left(0,\min\left\{\frac{s_0p}{p-2},1\right\}\right)$ for some constant $s_0\in(0,s]$.
	We denote 
	\begin{equation}\label{defn.alpha1}
		\begin{aligned}
			\alpha_1=\begin{cases}
				\frac{t+s_0p}{p-1}&\quad\text{if }\frac{t+s_0p}{p-1}\leq 1,\\
				s_0+\frac{1+t}{p}&\quad\text{if }\frac{t+s_0p}{p-1}>1.
			\end{cases}
		\end{aligned}
	\end{equation} Then the following holds.

	\textbf{In case of $\frac{t+s_0p}{p-1}\leq 1$: } For any $\alpha\in(0,\alpha_1)$, we have
	$u\in W^{\alpha,p}_{\mathrm{loc}}(\Omega)$ with the estimate 
	\begin{equation}\label{est.sec.thm7}
		r^{\alpha-\frac{n}p}[u]_{W^{\alpha,p}(B_r(x_0))}\leq cE(u;B_{2r}(x_0))+cr^{\frac{sp+t}{p-1}-\frac{n}p}\|f\|^{\frac{1}{p-1}}_{\widetilde{W}^{t,p'}(B_{2r}(x_0))}
	\end{equation}
	for some constant $c=c(n,s_0,p,t,\alpha)$, whenever $B_{2r}(x_0)\Subset \Omega$.

	\textbf{In case of $\frac{t+s_0p}{p-1}> 1$: } For any $\alpha\in(1,\alpha_1)$, we have $\nabla u\in W^{\alpha-1,p}_{\mathrm{loc}}(\Omega)$ with the estimate 
	\begin{equation}\label{est.secg.thm7}
		r^{\alpha-\frac{n}p}[\nabla u]_{W^{\alpha-1,p}(B_r(x_0))}\leq cE(u;B_{2r}(x_0))+cr^{\frac{sp+t}{p-1}-\frac{n}p}\|f\|^{\frac{1}{p-1}}_{\widetilde{W}^{t,p'}(B_{2r}(x_0))}
	\end{equation}
	for some constant $c=c(n,s_0,p,t,\alpha)$, whenever $B_{2r}(x_0)\Subset \Omega$.
\end{theorem}

%\begin{corollary}\label{cor.deg}
%	Let $p\geq2$. Let $u\in W^{s,p}_{\mathrm{loc}}(\Omega)\cap L^{p-1}_{sp} (\bbR^{n})$ be a weak solution to \eqref{pt : eq.main} with $f=0$.
%	Let us fix $s_0\in(0,s]$. Then we have $u\in W^{\alpha,p}_{\mathrm{loc}}(\Omega)$ for any $\alpha\in\left(0,\min\left\{\frac{s_0p}{p-2},s_0+\frac{2}{p}\right\}\right)$. In particular, we have 
%	\begin{align*}
%		r^{\alpha-\frac{n}{p}}[u]_{W^{\alpha,p}(B_r(x_0))}\leq cE(u;B_{2r}(x_0))
%	\end{align*}
%	if $\frac{s_0p}{p-2}\leq 1$ and
%	\begin{align*}
%		r^{\alpha-\frac{n}{p}}[\nabla u]_{W^{\alpha-1,p}(B_r(x_0))}\leq cE(u;B_{2r}(x_0))
%	\end{align*}
%	if $\frac{s_0p}{p-2}> 1$, where $c=c(n,s_0,p,\alpha)$ provided that $B_{2r}(x_0)\Subset \Omega$. 
%\end{corollary}
\begin{remark} \normalfont
	In the homogeneous case when $f=0$, Theorem \ref{thm.regd} yields the same amount of differentiability that was recently obtained in \cite[Theorem 1.1 and Theorem 1.6]{DuzLiao}.
\end{remark}
Next, we provide a similar result in the case when $p\in(1,2]$.
\begin{theorem}[Differentiable data, subquadratic case] \label{thm.regs}
	Let $p\in(1,2]$ and $u\in W^{s,p}_{\mathrm{loc}}(\Omega)\cap L^{p-1}_{sp}(\bbR^n)$ be a weak solution to \eqref{pt : eq.main} with $f\in W^{t,p'}_{\mathrm{loc}}(\Omega)$, where $t\in(0,p-1)$. Let us fix $s_0\in(0,s]$ and denote 
	\begin{equation}\label{defn.alpha2}
		\begin{aligned}
			\alpha_2=\begin{cases}
				\frac{t+s_0p}{p-1}&\quad\text{if }\frac{t+s_0p}{p-1}\leq 1,\\
				1+\frac{1+t}{2}-\frac{(1-s_0)p}{2}&\quad\text{if }\frac{t+s_0p}{p-1}>1.
			\end{cases}
		\end{aligned}
	\end{equation} Then the following holds.

	\textbf{In case of $\frac{t+s_0p}{p-1}\leq 1$: } For any $\alpha\in(0,\alpha_2)$, we have
	$u\in W^{\alpha,p}_{\mathrm{loc}}(\Omega)$ with the estimate 
	\begin{equation}\label{est.sec.thm9}
		r^{\alpha-\frac{n}p}[u]_{W^{\alpha,p}(B_r(x_0))}\leq cE(u;B_{2r}(x_0))+cr^{\frac{sp+t}{p-1}-\frac{n}p}\|f\|^{\frac{1}{p-1}}_{\widetilde{W}^{t,p'}(B_{2r}(x_0))}
	\end{equation}
	for some constant $c=c(n,s_0,p,t,\alpha)$, whenever $B_{2r}(x_0)\Subset \Omega$.

	\textbf{In case of $\frac{t+s_0p}{p-1}> 1$: } For any $\alpha\in(1,\alpha_2)$, we have $\nabla u\in W^{\alpha-1,p}_{\mathrm{loc}}(\Omega)$ with the estimate 
	\begin{equation}\label{est.secg.thm9}
		r^{\alpha-\frac{n}p}[\nabla u]_{W^{\alpha-1,p}(B_r(x_0))}\leq cE(u;B_{2r}(x_0))+cr^{\frac{sp+t}{p-1}-\frac{n}p}\|f\|^{\frac{1}{p-1}}_{\widetilde{W}^{t,p'}(B_{2r}(x_0))}
	\end{equation}
	for some constant $c=c(n,s_0,p,t,\alpha)$, whenever $B_{2r}(x_0)\Subset \Omega$.
\end{theorem}
\begin{remark}[Sharpness, differentiable data] \normalfont
For any $\epsilon>0$, we are able to construct a weak solution $u\notin W^{\frac{sp+t}{p-1}+\epsilon,p}_{\mathrm{loc}}$ to \eqref{pt : eq.main} with $f\in W^{t,p'}_{\mathrm{loc}}$. As in Remark \ref{rmk.sharp}, we consider the function 
	\begin{equation*}
		u(x)=|x|^{\frac{sp+t}{p-1}+\frac{\epsilon}{2}-\frac{n}p}\in \left(W^{s,p}_{\mathrm{loc}}(\bbR^n)\cap L^{p-1}_{sp}(\bbR^n)\right)\setminus  W^{\frac{sp+t}{p-1}+\epsilon,p}(B_1).
	\end{equation*}
By following the same lines as in the proof of \eqref{eq11}, we obtain
	\begin{align*}
		(-\Delta_p)^su(x)=f(x)=c|x|^{t+\frac{\epsilon(p-1)}{2}-\frac{n(p-1)}{p}}\quad\text{in }B_1	\end{align*}
	for some constant $c=c(n,s,p)$, where $f\in  W^{t,p'}_{\mathrm{loc}}$. This implies that our results given in Theorem \ref{thm.regd} and Theorem \ref{thm.regs} are sharp when $\frac{t+sp}{p-1}\leq1$. 
\end{remark}

Finally, we summarize our findings in the homogeneous subquadratic setting, since our results are new already in this case.

\begin{corollary}[Homogeneous subquadratic case]
	\label{cor.regs}
	Let $p\in(1,2]$ and $u\in W^{s,p}_{\mathrm{loc}}(\Omega)\cap L^{p-1}_{sp} (\bbR^{n})$ be a weak solution to \eqref{pt : eq.main} with $f\equiv 0$. Let us fix $s_0\in(0,s]$. Then for any $\alpha\in\left(1,\max\{1+\frac{s_0p}{2},2+(s_0-1)p\}\right)$, we have $\nabla u\in W^{\alpha,p}_{\mathrm{loc}}(\Omega)$ and
	\begin{align}\label{est.sin.grad}
		r^{1+\alpha-\frac{n}{p}} [\nabla u]_{W^{\alpha,p}(B_r(x_0))}\leq cE(u;B_{2r}(x_0))
	\end{align}
	for some constant $c=c(n,s_0,p,\alpha)$, provided that $B_{2r}(x_0)\Subset \Omega$.
\end{corollary}

\begin{remark}[Existence of weak gradient] \normalfont
	Corollary \ref{cor.regs} implies that if $u$ is a solution of the fractional $p$-Laplace equation $(-\Delta_p)^su=0$ in $\Omega$ with $p\leq2$, then the weak gradient of the solution always exists and belongs to $L^p_{\loc}(\Omega)$, which in contrast to the standard $p$-Laplace equation is a priori not known in the nonlocal case. 
\end{remark}

\begin{remark}[Higher differentiability for fractional $p$-Laplace systems] \label{systems} \normalfont
	Finally, we remark that all of our arguments apply essentially verbatim in the vectorial context of the inhomogeneous fractional $p$-Laplace system. In this vectorial setting local $L^\infty$-estimates have recently been shown in \cite{BehDieNowSch}.
\end{remark}

\subsection{Previous results}
Studying the regularity of weak solutions to nonlocal equations of fractional $p$-Laplacian-type has been a very active area of research in recent years.
The first regularity results in this direction were concerned with local boundedness and H\"older regularity for small exponents in the spirit of De Giorgi-Nash-Moser theory, see e.g.\ \cite{BP,DKP,IanMosSqu16,CozziJFA,DFPJDE,CKWCalcVar,KPT,LiaoPar}, where the whole range $p \in (1,\infty)$ is covered. Moreover, fine zero-order regularity estimates are obtained in e.g.\ \cite{KuuMinSir15,KMSsurvey,KLL,NOS24}.

Concerning higher differentiability in the superquadratic case when $p\geq 2$, in the fundamental paper \cite{BraLin17} Brasco and Lindgren prove the conclusion of our Theorem \ref{thm.regd} in the special case when $t=s$ assuming an additional amount of differentiability of the solution at infinity. Thus, for $t>s$ in Theorem \ref{thm.regd} we improve the differentiability gain obtained in \cite{BraLin17} and additionally remove the mentioned differentiability assumption at infinity, thereby proving higher differentiability under optimal assumptions on the long-range behavior of the solution. Again in the superquadratic case, while we were preparing the present paper, in the interesting recent work \cite{DuzLiao} the authors prove in particular the conclusion of our Theorem \ref{thm.regd} in the homogeneous case when $f=0$ by a slightly different approach. 

Concerning non-differentiable data, a large amount of regularity results are known in the linear case when $p=2$ in the additional presence of coefficients. Indeed, in \cite{KMS1} a slight differentiability gain was obtained for such linear nonlocal equations with bounded measurable coefficients. Moreover, in \cite{CozziSob} the same regularity as in our Theorem \ref{thm.1} for $p=2$ is proved in the additional presence of suitably regular coefficients. Below the gradient level, the assumptions on the coefficients were subsequently weakened in e.g.\ \cite{MSY,FMSYPDEA,MeV,MeI}. Again in the case $p=2$, in \cite{KuuSimYan22} sharp higher differentiability in the presence of measure data was obtained for linear nonlocal equations with H\"older coefficients. Moreover, a similar higher differentiability result for nonlinear nonlocal measure data problems with linear growth was recently obtained by the authors in \cite{DieKimLeeNow24}. In the case when $p \geq 2$, a slight improvement of differentiability in the spirit of \cite{KuuMinSir15} was observed under non-differentiable data in e.g.\ \cite{SchikorraMA,MengScott,BKK23}. In addition, below the gradient level, the conclusion of our Theorem \ref{thm.1} in the superquadratic case was proved by different methods by two of the authors in \cite[Corollary 1.10]{DNCZ}. On the other hand, for $p \neq 2$ our Theorem \ref{thm.1} to the best of our knowledge yields the first higher differentiability results above the gradient level in the presence of non-differentiable data. Furthermore, in the subquadratic case when $p<2$, all of our higher differentiability results seem to be completely new.

In addition, higher H\"older regularity and Calder\'on-Zygmund estimates below the gradient level for equations of fractional $p$-Laplacian-type were for instance studied in \cite{BLS,DNCZ,Kyeongbae1,DuzLiao} for $p \geq 2$ and in \cite{GarLin23} for $p<2$. In contrast, proving H\"older regularity of the gradient of solutions to \eqref{pt : eq.main}  in the spirit of the classical results \cite{NU68,Uhl77} remains an intriguing open question. Nevertheless, we note that in \cite{DieKimLeeNow24} we recently established gradient H\"older regularity as well as gradient potential estimates for nonlinear nonlocal equations with linear growth.

\subsection{Technical approach and novelties}

In the proofs of our main results, as in the previous papers \cite{BraLin17} and \cite{DuzLiao} that deal with the superquadratic case and regular data, we rely on fractional differentiation of the equation in terms of difference quotients in both the superquadratic as well as in the subquadratic case. The main technical reason why the authors in \cite{DuzLiao} are able to improve the higher differentiability result obtained in \cite{BraLin17} for $f=0$ is a novel tail estimate for finite differences, see \cite[Lemma 3.1]{DuzLiao}. On the other hand, in Theorem \ref{thm.regd} we achieve a similar improvement of the result obtained in \cite{BraLin17} by a related, but somewhat different approach. In fact, instead of estimating the tail terms containing finite differences directly, in Section \ref{sec:loc} we develop and apply certain localization arguments that enable us to essentially treat the nonlocal part of the fractional $p$-Laplacian as a right-hand side. One advantage of our approach is that while in the tail estimate \cite[Lemma 3.1]{DuzLiao} solutions are assumed to be locally bounded, our localization approach does not require any additional regularity assumption. In particular, since in the presence of a right-hand side as in Theorem \ref{thm.regd} solutions might in general not be locally bounded, our approach is suitable for obtaining higher differentiability results in the presence of general data, while in \cite{DuzLiao} only the homogeneous case when $f=0$ is treated.

In addition, a similar philosophy of localizing the solution was already successfully implemented in order to obtain optimal regularity results e.g.\ in the context of nonlocal equations in nondivergence form (see e.g.\ \cite{CSCPAM,FRRO}), for parabolic nonlocal equations (see e.g.\ \cite{KW24}) and recently by the authors also for nonlinear nonlocal equations with linear growth (see \cite{DieKimLeeNow24}). For this reason, we believe that the localization arguments for the fractional $p$-Laplacian developed in the present paper are of independent interest and might also be useful beyond obtaining higher differentiability results.

\subsection{Outline}
In Section \ref{sec:prelim}, we gather some preliminary definitions and technical results mostly related to useful embeddings and properties of difference quotients.
In Section \ref{sec:loc}, we then develop and prove our localization arguments mentioned in the previous section.
Section \ref{sec:nddata} is devoted to proving higher differentiability in the presence of non-differentiable data, that is, to the proof of Theorem \ref{thm.1}.
Finally, in Section \ref{sec:ddata} we conclude the paper by proving our remaining main results that are concerned with differentiable data.

\section{Preliminaries} \label{sec:prelim}
We denote general constants $c\geq1$ which may change each line. Brackets are used to clarify dependencies on constants, so that e.g. $c=c(n,s,p)$ means that $c$ depends only on $n,s$ and $p$. First of all, for any fixed $p \in (1,\infty)$ we define the function $\varphi:\mathbb{R} \to \mathbb{R}$ and its derivative by
\begin{equation}
	\label{pt : assmp.phi}
	\phi(t):= \tfrac{1}{p} |t|^p, \quad \phi'(t)= |t|^{p-2}t.
\end{equation}
We then denote
\begin{equation}\label{defn.avftn}
    A(X)\coloneqq |X|^{p-2} X=\phi'(|X|)\frac{X}{|X|},\quad
    V(X)\coloneqq |X|^\frac{p-2}{2}X=\sqrt{\phi'(|X|)|X|}\frac{X}{|X|}
\end{equation}
for any $X\in \bbR$. Note that while of course, in the present scalar setting we have $A(X)=\phi'(X)$, we still use the notation introduced in \eqref{defn.avftn} since it is required to handle the vectorial case mentioned in Remark \ref{systems}. 

We mention some useful elementary inequalities (see e.g.\ \cite{DieEtt08,DFTW}).
\begin{lemma}
\label{el : lem.eleineq}
For any $X,Y\in\bbR$, the following inequalities hold.
\vspace{1mm}
\begin{enumerate}
    \item $|V(X)-V(Y)|^{2}\eqsim (A(X)-A(Y))\cdot (X-Y)$.
	\vspace{1mm}
    \item $ |V(X)-V(Y)|^{2}\eqsim \phi''(|X|+|Y|)|X-Y|^{2}$.
    \vspace{1mm}
    \item $|{A}(X)-{A}(Y)|\eqsim {\phi''}(|X|+|Y|)|X-Y|$.
    \vspace{1mm}
    \item For $p\geq2$, $({A}(X)-{A}(Y))\cdot(X-Y)\geq\frac{1}{c}|X-Y|^{p}$.
\end{enumerate}
All of the (implicit) constants depend only on $p$.
\end{lemma}

As in \cite{DFTW} we define a shifted $N$-function
\begin{align*}
    \phi_{a}'(t)\coloneqq \frac{\phi'(a\vee t)}{a\vee t}t\quad\text{ for any }a,t>0,
\end{align*}
where $a\vee b=\max\{a,b\}$. Then
\begin{align}\label{equi.shif}
    \phi'_{|X|}(|X-Y|)\eqsim |A(X)-A(Y)|
\end{align}
and
\begin{align*}
    \phi'_{a}(t) t\eqsim \phi_{a}(t)
\end{align*}
for any $X,Y\in\bbR$ and $a,t>0$, where the implicit constants depend on $p$.
Hence it follows that
\begin{equation}\label{ineq1.shif}
    |\phi_a(t)|\leq c |a\vee t|^{p-2} |t|^2\leq c|a|^{p-2}|t|^2+c|t|^p\quad\text{if }p\geq2
\end{equation}
and
\begin{equation}\label{ineq2.shif}
    |\phi_a(t)| \leq c |t|^p\quad\text{if }p\leq2
\end{equation}
for some constant $c=c(p)$.

Given a measurable function $g: \bbR^n\to \bbR$ and $h\in\bbR^n$, we denote
\begin{equation*}
    g_h(x) \coloneqq g(x+h),\,\quad \delta_h g(x)\coloneqq g_h(x)-g(x)\quad\text{and}\quad \delta_h^2g\coloneqq \delta_h(\delta_h g).
\end{equation*}
In addition, for a measurable function $G:\bbR^n\times\bbR^n\to\bbR$, we similarly write
\begin{align*}
   G_h(x,y)\coloneqq G(x+h,y+h)\quad\text{and}\quad \delta_h G(x,y)\coloneqq G_h(x,y)-G(x,y).
\end{align*}

\subsection{Embedding results}
Let us first recall the fractional Sobolev spaces. For $\Omega \subset\setR^n$, $s\in(0,1)$ and $p\in[1,\infty)$,  $W^{s,p}(\Omega)$ consists of all functions $g:\Omega\rightarrow\setR^n$ satisfying
\begin{align*}
\|g\|_{W^{s,p}(\Omega)} &\coloneqq \|g\|_{L^p(\Omega)}+[g]_{W^{s,p}(\Omega)}\\
&:=\left(\int_{\Omega}|g|^p\,dx\right)^{\frac{1}{p}}+\left(\int_{\Omega}\int_{\Omega}\dfrac{|g(x)-g(y)|^p}{|x-y|^{n+sp}}\,dxdy\right)^{\frac{1}{p}}<\infty.
\end{align*}
We also define the corresponding local fractional Sobolev spaces as 
\begin{align*}
W^{s,p}_{\loc}(\Omega):=\{g\in L^p_{\loc}(\Omega): g\in W^{s,p}(K)\,\,\text{for any compact set } K\subset \Omega\}.
\end{align*}
Then from \cite[Lemma 4.5]{CozziJFA}, the following embedding result holds between two fractional Sobolev spaces: For $0<s<s'<1$ and $p\in[1,\infty)$,
\begin{align}\label{eq:embed}
[g]_{W^{s,p}(B_r)}\leq r^{(s'-s)p}[g]_{W^{s',p}(B_r)}
\end{align}
with any $g\in W^{s',p}(B_r)$ for some $c=c(n,s,p)$. See \cite{DinPalVal12} for more details on fractional Sobolev spaces.
Next, we mention a robust fractional Poincar\'e inequality.
\begin{lemma}\label{lem.poist}(See \cite[Corollary 2.1]{Pon04}) Let $g\in W^{s,p}(B_R(x_0))$ and fix $s_0\in(0,s]$. Then there is a constant $c=c(n,s_0,p)$ such that 
\begin{equation*}
    \dashint_{B_R(x_0)}|g-(g)_{B_R(x_0)}|^{p}\,dx\leq c(1-s)R^{sp}\dashint_{B_R(x_0)}\int_{B_R(x_0)}\frac{|g(x)-g(y)|^{p}}{|x-y|^{n+sp}}\,dx\,dy.
\end{equation*}
\end{lemma}

The next three lemmas yield several relations between fractional Sobolev spaces and Besov-type spaces which are described by finite differences. In what follows, each cut-off function is always assumed to be nonnegative.
\begin{lemma}\label{lem.firemb}
    Let $p\in[1,\infty)$ and $\sigma\in(0,1)$. If $g\in L^{p}(B_{R+h_0}(x_0))$ satisfies
    \begin{equation*}
        \sup_{0<|h|<h_0}\left\|\frac{\delta_hg}{|h|^{\sigma}}\right\|_{L^{p}(B_{R}(x_0))}< \infty
    \end{equation*}
    for some constant $h_0\in(0,R/4)$, then for every $\alpha\in(0,\sigma)$, $g\in W^{\alpha,p}({B_{R/2}(x_0)})$ holds. Moreover, we have
    \begin{align*}
        [g]^p_{W^{\alpha,p}({B_{R/2}(x_0)})}&\leq \frac{ch_0^{p(\sigma-\alpha)}}{(\sigma-\alpha)^p}\sup_{0<|h|<h_0}\left\|\frac{\delta_h g}{|h|^\sigma}\right\|^p_{L^p(B_R(x_0))}\\
        &\quad+c\left(\frac{h_0^{p(1-\alpha)}}{R^p(\sigma-\alpha)^p}+ \frac{h_0^{-\alpha p}}{\alpha}\right)\|g-k\|^p_{L^{p}(B_R(x_0))}
    \end{align*}
    for any $k\in\bbR$ with some constant $c=c(n,p)$.
\end{lemma}
\begin{proof}
    Take a cut-off function $\psi\in C_{c}^{\infty}(B_{3R/4}(x_0))$ such that $\psi\equiv 1$ on $B_{R/2}(x_0)$ and $|\nabla \psi|\leq 8/R$ to observe 
    \begin{align*}
        \sup_{0<|h|<h_0}\left\|\frac{\delta_h (g\psi)}{|h|^\sigma}\right\|_{L^p(\bbR^n)}\leq c\sup_{0<|h|<h_0}\left\|\frac{\delta_h g}{|h|^\sigma}\right\|_{L^p(B_R(x_0))}+c\frac{h_0^{1-\sigma}}{R}\left\|g\right\|_{L^p(B_R(x_0))}
    \end{align*}
    for some constant $c=c(n,p)$. Therefore by \cite[Proposition 2.7]{BraLin17}, we get
    \begin{align*}
        [g\psi]^p_{W^{\alpha,p}(\bbR^n)}
        &\leq c\frac{h_0^{p(\sigma-\alpha)}}{(\sigma-\alpha)^p}\sup_{0<|h|<h_0}\left\|\frac{\delta_h(g\psi)}{|h|^{\sigma}}\right\|^p_{L^{p}(\bbR^n)}\\
        &\quad+ c\frac{h_0^{-\alpha p}}{\alpha}\|g\psi\|^p_{L^{p}(\bbR^n)}
    \end{align*}
    for some constant $c=c(n,p)$. Combining the above two inequalities yields
    \begin{align*}
        [g\psi]^p_{W^{\alpha,p}(\bbR^n)}&\leq \frac{ch_0^{p(\sigma-\alpha)}}{(\sigma-\alpha)^p}\sup_{0<|h|<h_0}\left\|\frac{\delta_h g}{|h|^\sigma}\right\|^p_{L^p(B_R(x_0))}+c\frac{h_0^{p(1-\alpha)}}{R^p(\sigma-\alpha)^p}\left\|g\right\|^p_{L^p(B_R(x_0))}\\
        &\quad+ c\frac{h_0^{-\alpha p}}{\alpha}\|g\|^p_{L^{p}(B_R(x_0))}
    \end{align*}
    for some constant $c=c(n,p)$. Since we are able to prove the above inequality with $g$ replaced by $g-k$ for any real number $k$, we obtain the desired result by the fact that $\psi\equiv 1$ on $B_{R/2}(x_0)$.
\end{proof}

% \begin{lemma}\label{lem.secquo}(See \cite[Lemma 2.2]{BraLin17}.)
%     Let $p\in[1,\infty)$ and $\sigma\in(0,2)$. If $v\in L^{p}(\bbR^n)$ satisfies
%     \begin{equation*}
%         \sup_{0<|h|<h_0}\left\|\frac{\delta_h^2g}{|h|^{\sigma}}\right\|_{L^{p}(\bbR^{n})}+3h_0^{-\sigma}\|g\|_{L^{p}(\bbR^n)}\leq M_0
%     \end{equation*}
%     for some constants $h_0$ and $M_0$, then $g\in \mathcal{B}^{\sigma,p}_\infty(\bbR^n)$ with the estimate 
%     \begin{equation*}
%         [g]_{\mathcal{B}^{\sigma,p}_\infty(\bbR^n)}\leq M_0.
%     \end{equation*}
% \end{lemma}

\begin{lemma}\label{lem.secmeb}
Let $g\in L^{p}(\bbR^n)$ satisfy
\begin{equation}\label{ass.lem.secemb}
\sup_{0<|h|<h_0}\left\|\frac{\delta_h^2 g}{|h|^{\sigma}}\right\|_{L^{p}(\bbR^n)}\leq M_0
\end{equation}
for some constants $h_0,M_0>0$, where $\sigma\in(0,2)$. Then the following holds.

If $\sigma\in(0,1]$, there exists $c=c(n,p)$ such that for any $\alpha\in(0,\sigma)$, we have
\begin{equation}\label{ineq.seq1}
\begin{aligned}
[g]_{W^{\alpha,p}(\bbR^n)}^p&\leq \frac{ch_0^{p(\sigma-\alpha)}}{(\sigma-\alpha)^p(1-\sigma)^p}M_0^p\\
&\quad+c\left(\frac{h_0^{-\alpha p}+h_0^{p(\sigma-\alpha)}}{(\sigma-\alpha)^p(1-\sigma)^p}+\frac{h_0^{-\alpha p}}{{\alpha}}\right)\|g\|_{L^{p}(\bbR^n)}^p.
\end{aligned}
\end{equation}

If $\sigma\in(1,2)$, there exists $c=c(n,p)$ such that for any $\alpha\in(0,\sigma-1)$, we have
\begin{equation}\label{ineq.seq2}
\begin{aligned}
[\nabla g]_{W^{\alpha,p}(\bbR^n)}^p&\leq \frac{ch_0^{p(\sigma-1-\alpha)}}{(\sigma-1-\alpha)^p(2-\sigma)^p(\sigma-1)^p}\left(M_0^p+h_0^{-\sigma p}\|g\|_{L^{p}(\bbR^n)}^p\right)\\
&\quad+\frac{ch_0^{-\alpha p}}{\alpha}\left(\| g\|_{L^p(\bbR^n)}^p+\frac{1}{(\sigma-1)^p}(M_0^p+h_0^{-\sigma p}\| g\|_{L^p(\bbR^n)}^p)\right).
\end{aligned}
\end{equation} 
\end{lemma}
\begin{proof}
    Let us first prove \eqref{ineq.seq1}. Suppose $\sigma\in(0,1]$ and fix $\alpha\in(0,\sigma)$. Then by \cite[Lemma 2.3]{BraLin17}, we get
    \begin{align*}
        \sup_{0<|h|<h_0}\left\|\frac{\delta_h g}{|h|^{\sigma}}\right\|_{L^{p}(\bbR^n)}^p\leq \frac{c}{\left(1-\sigma\right)^p}\left[M_0^p+\left(h_0^{-{\sigma p}}+1\right)\|g\|_{L^p(\bbR^n)}^p\right]
    \end{align*}
    for some constant $c=c(n,p)$. We next use \cite[Proposition 2.7]{BraLin17} to see that
    \begin{align*}
        [g]_{W^{\alpha,p}(\bbR^n)}^p&\leq c\left(\frac{h_0^{{p(\sigma-\alpha)}}}{(\sigma-\alpha)^p}\sup_{0<|h|<h_0}\left\|\frac{\delta_h g}{|h|^{\sigma}}\right\|_{L^{p}(\bbR^n)}^p+\frac{h_0^{-\alpha p}}{\alpha}\|g\|_{L^{p}(\bbR^n)}^p\right)
    \end{align*}
    for some constant $c=c(n,p)$. Combining the above two inequalities yields \eqref{ineq.seq1}. We now prove \eqref{ineq.seq2} for $\sigma\in(1,2)$. First employ \cite[Proposition 2.4]{BraLin17} along with  \cite[Lemma 2.2]{BraLin17} to see that 
    \begin{align*}
        \|\nabla g\|_{L^p(\bbR^n)}^p\leq c\| g\|_{L^p(\bbR^n)}^p+\frac{c}{(\sigma-1)^p}\left( \sup_{0<|h|<h_0}\left\|\frac{\delta_h^2g}{|h|^{\sigma}}\right\|^p_{L^{p}(\bbR^{n})}+h_0^{-\sigma p}\|g\|_{L^{p}(\bbR^n)}^p\right)
    \end{align*}
    for some constant $c=c(n,p)$. This along with \eqref{ass.lem.secemb} implies
    \begin{align*}
        \|\nabla g\|_{L^p(\bbR^n)}^p\leq c\| g\|_{L^p(\bbR^n)}^p+\frac{c}{(\sigma-1)^p}(M_0^p+h_0^{-\sigma p}\| g\|_{L^p(\bbR^n)}^p),
    \end{align*}
    where $c=c(n,p)$. We also employ \cite[Proposition 2.4]{BraLin17} along with  \cite[Lemma 2.2]{BraLin17} to find that 
    \begin{align*}
        \sup_{|h|>0}\left\|\frac{\delta_h \nabla g}{|h|^{\sigma-1}}\right\|_{L^{p}(\bbR^n)}^p&\leq \frac{c}{(2-\sigma)^{p}(\sigma-1)^p}\left( \sup_{0<|h|<h_0}\left\|\frac{\delta_h^2g}{|h|^{\sigma}}\right\|^p_{L^{p}(\bbR^{n})}+h_0^{-\sigma p}\|g\|_{L^{p}(\bbR^n)}^p\right)
    \end{align*}
    for some constant $c=c(n,p)$. By \cite[Proposition 2.7]{BraLin17}, we similarly get
    \begin{align*}
        [\nabla g]_{W^{\alpha,p}(\bbR^n)}^p&\leq c\left(\frac{h_0^{p(\sigma-1-\alpha)}}{(\sigma-1-\alpha)^p}\sup_{0<|h|<h_0}\left\|\frac{\delta_h\nabla g}{|h|^{\sigma-1}}\right\|_{L^{p}(\bbR^n)}^p+\frac{h_0^{-\alpha p}}{\alpha}\|\nabla g\|_{L^{p}(\bbR^n)}^p\right),
    \end{align*}
    where $c=c(n,p)$. Combining all the estimates along with \eqref{ass.lem.secemb} yields \eqref{ineq.seq2}.
\end{proof}

\begin{lemma}[\cite{BraLin17}, Proposition 2.6]\label{lem.embdiff} Let $p\in[1,\infty)$ and $\sigma\in(0,1)$. 
\begin{enumerate}
    \item For every $g\in W^{\sigma,p}(\bbR^n)$, 
\begin{equation}\label{lem.embdiff.a1}
    \sup_{|h|>0}\left\|\frac{\delta_h g}{|h|^{\sigma}}\right\|^{p}_{L^{p}(\bbR^n)}\leq c(1-\sigma)[g]_{W^{\sigma,p}(\bbR^n)}^p
\end{equation}
holds, where $c=c(n,p)$. In addition, for any $h_0>0$, we have 
\begin{equation}\label{lem.embdiff.a2}
    \sup_{0<|h|<h_0}\left\|\frac{\delta_h g}{|h|^{\sigma}}\right\|^{p}_{L^{p}(\bbR^n)}\leq c(1-\sigma)[g]_{W^{\sigma,p}(\bbR^n)}^p+ch_0^{-p}\|g\|_{L^p(\bbR^n)}^p
\end{equation}
for some constant $c=c(n,p)$.
\item Let $g\in W^{\sigma,p}(B_{R+h_0})$ for some $R>0$ and $h_0>0$. Then 
\begin{align*}
    \sup_{0<|h|<h_0}\left\|\frac{\delta_h g}{|h|^{\sigma}}\right\|^{p}_{L^{p}(B_R)}&\leq c(1-\sigma)[g]_{W^{\sigma,p}(B_{R+h_0})}^p\\
    &\quad+c\left[(1+Rh_0^{-1})^p(R+h_0)^{-\sigma p}+\sigma^{-1}h_0^{-\sigma p}\right]\|g\|_{L^{p}(B_{R+h_0})}^{p}
\end{align*}
holds, where $c=c(n,p)$.
\end{enumerate}

\end{lemma}

For any $g\in L^p(B_R(x_0))\cap L^{p-1}_{sp}(\bbR^n)$, let us denote 
\begin{align*}
    \widetilde{E}(g;B_R(x_0))=\left(\dashint_{B_R(x_0)}|g|^{p}\,dx\right)^{\frac{1}{p}}+\mathrm{Tail}(g;B_R(x_0)).
\end{align*}
We end this section with the following lemma which allows us to obtain estimates with respect to the excess functional $E(\cdot)$.
\begin{lemma}\label{lem.upest}
    Let $u\in W^{s,p}(B_{2R}(x_0))\cap L^{p-1}_{sp}(\bbR^{n})$ be a weak solution to 
    \begin{equation}\label{eq.upest}
        (-\Delta_p)^su=f\quad\text{in }B_{2R}(x_0),
    \end{equation}
    where $f\in L^{p'}(B_2)$. Then we have 
    \begin{align}\label{ineq2.upest}
        (1-s)^{\frac{1}{p}}R^{-\frac{n}{p}+s}[u]_{W^{s,p}(B_R(x_0))}\leq c{E}(u;B_{2R}(x_0))+cR^{\frac{sp}{p-1}}\left(\dashint_{B_{2R}(x_0)}|f|^{p'}\,dx\right)^{\frac{1}{p}}
    \end{align}
    for some constant $c=(n,s,p)$. In addition, for fixed constant $s_0\in(0,1)$, the constant $c$ depends only on $n,s_0$ and $p$ whenever $s\in[s_0,1)$.
\end{lemma}
\begin{proof}
As in the proof of \cite[Proposition 8.5]{CozziJFA}, observe that
\begin{equation*}
    \begin{aligned}
        (1-s)^{\frac{1}{p}}R^{-\frac{n}{p}+s}[u]_{W^{s,p}(B_R(x_0))}\leq c\widetilde{E}(u;B_{2R}(x_0))+cR^{\frac{sp}{p-1}}\left(\dashint_{B_{2R}(x_0)}|f|^{p'}\,dx\right)^{\frac{1}{p}}
    \end{aligned}
    \end{equation*}
holds for some constant $c=c(n,s_0,p)$.
Moreover, since $u-(u)_{B_{2}}$ is a also weak solution to \eqref{eq.upest}, we get \eqref{ineq2.upest}.
\end{proof}

%%%%%%%%%%%%%%%%%%%%%%%%%%%%%%%%

\section{Localization arguments} \label{sec:loc}

In this section, we localize our equation \eqref{pt : eq.main}. This allows us to assume that the solution $u$ enjoys global Sobolev regularity.
For the remainder of this section, we choose any parameter $s_0\in(0,1)$ and fix another parameter $s$ satisfying
\begin{equation}\label{s0assum}
    s\in[s_0,1).
\end{equation}
\begin{lemma}
\label{el : lem.loc}
Let $B_{5R}(x_{0})\Subset\Omega$ and $u\in W^{s,p}_{\mathrm{loc}}(\Omega)\cap L^{p-1}_{sp}(\bbR^{n})$ be a weak solution to \eqref{pt : eq.main}. Let us fix a cut-off function $\xi \in C_{c}^{\infty}\left(B_{4R}(x_{0})\right)$ with $\xi\equiv 1$ on $B_{3R}(x_{0})$. Then we have that $w\coloneqq u\xi\in W^{s,p}(\bbR^{n})$ is a weak solution to
\begin{equation}\label{eq:f+g}
    (-\Delta_p)^sw=f+g\quad\text{in }B_{2R}(x_{0}),
\end{equation}
where $g\in L^{p'}(B_{5R/2}(x_0))$ satisfies
\begin{equation}\label{ineq.loc}
\begin{aligned}
    \dashint_{B_{5R/2}(x_0)}|g|^{p'}\,dx&\leq c(1-s)^{p'}R^{-spp'}\dashint_{B_{5R/2}(x_0)}|u|^p\,dx\\
    &\quad+cR^{-spp'}\mathrm{Tail}(u;B_{3R}(x_0))^p
\end{aligned}
\end{equation}
for some constant $c=c(n,s_0,p)$, where the constant $s_0$ is determined in \eqref{s0assum}. In particular, for any $x_1,x_2\in B_{5R/2}(x_0)$,
\begin{align}\label{ineq1.loc}
    |g(x_1)-g(x_2)|
    &\leq c(1-s)\int_{\bbR^{n}\setminus B_{3R}(x_0)}\phi'_{\frac{|u(x_1)-(u\xi)(y)|}{|x_1-y|^{s}}}\left(\frac{|u(x_1)-u(x_2)|}{|x_1-y|^{s}}\right)\frac{\,dy}{|x_1-y|^{n+s}}\nonumber\\
    & +c(1-s)\int_{\bbR^{n}\setminus B_{3R}(x_0)}\phi'_{\frac{|u(x_1)-u(y)|}{|x_1-y|^{s}}}\left(\frac{|u(x_1)-u(x_2)|}{|x_1-y|^{s}}\right)\frac{\,dy}{|x_1-y|^{n+s}}\nonumber\\
    &+c(1-s)\phi'\left(\frac{|u(x_2)|}{R^{s}}\right)\frac{1}{R^{s}}\frac{|x_1-x_2|}{R}\\
    &+cR^{-sp}\mathrm{Tail}(u;B_{3R}(x_0))^{p-1}\frac{|x_1-x_2|}{R}\nonumber
\end{align}
for some constant $c=c(n,p)$.
\end{lemma}
\begin{proof}
Let us take a test function $\psi\in W^{s,p}(B_{2R}(x_0))$ which has a compact support in $B_{2R}(x_0)$. Recall $A(X)=\phi'(X)\text{ for any }X\in\bbR.$
Then we have 
\begin{equation}
\label{el : lem.loc.eq1}
\begin{aligned}
    &(1-s)\int_{\bbR^{n}}\int_{\bbR^{n}}{\phi'}\left(\frac{w(x)-w(y)}{|x-y|^{s}}\right)\frac{\psi(x)-\psi(y)}{|x-y|^{s}}\frac{\,dx\,dy}{|x-y|^{n}}\\
    &\quad-\int_{\bbR^{n}}f\psi\,dx\\
    &=(1-s)\int_{\bbR^{n}}\int_{\bbR^{n}}{A}\left(\frac{w(x)-w(y)}{|x-y|^{s}}\right)\frac{\psi(x)-\psi(y)}{|x-y|^{s}}\frac{\,dx\,dy}{|x-y|^{n}}\\
    &\quad-(1-s)\int_{\bbR^{n}}\int_{\bbR^{n}}{A}\left(\frac{u(x)-u(y)}{|x-y|^{s}}\right)\frac{\psi(x)-\psi(y)}{|x-y|^{s}}\frac{\,dx\,dy}{|x-y|^{n}}\eqqcolon I,
\end{aligned}
\end{equation}
since $u$ is a weak solution to \eqref{pt : eq.main}. For $I$, since $w(x)=u(x)$ in $B_{3R}(x_0)$ and $\psi(x)\equiv 0$ on $\mathbb{R}^{n}\setminus B_{2R}(x_0)$, we obtain
\begin{align*}
    I&=2(1-s)\int_{B_{2R}(x_0)}\int_{\mathbb{R}^{n}\setminus B_{3R}(x_0)}{A}\left(\frac{w(x)-w(y)}{|x-y|^{s}}\right)\frac{\psi(x)}{|x-y|^{s}}\frac{\,dy\,dx}{|x-y|^{n}}\\
    &\quad-2(1-s)\int_{B_{2R}(x_0)}\int_{\mathbb{R}^{n}\setminus B_{3R}(x_0)}{A}\left(\frac{u(x)-u(y)}{|x-y|^{s}}\right)\frac{\psi(x)}{|x-y|^{s}}\frac{\,dy\,dx}{|x-y|^{n}},
\end{align*}
where the fact that $A$ is an odd function is used. As a result, the equality \eqref{el : lem.loc.eq1} can be rewritten as
\begin{align*}
    &(1-s)\int_{\bbR^{n}}\int_{\bbR^{n}}A\left(\frac{w(x)-w(y)}{|x-y|^{s}}\right)\frac{\psi(x)-\psi(y)}{|x-y|^{s}}\frac{\,dx\,dy}{|x-y|^{n}}\\
    &=\int_{B_{2R}(x_0)}({f}+g)\psi\,dx,
\end{align*}
where 
\begin{align*}
    g(x)&=2(1-s)\int_{\mathbb{R}^{n}\setminus B_{3R}(x_0)}{A}\left(\frac{w(x)-w(y)}{|x-y|^{s}}\right)\frac{\,dy}{|x-y|^{n+s}}\\
    &\quad-2(1-s)\int_{\mathbb{R}^{n}\setminus B_{3R}(x_0)}{A}\left(\frac{u(x)-u(y)}{|x-y|^{s}}\right)\frac{\,dy}{|x-y|^{n+s}}.
\end{align*}

This implies that $w\in W^{s,p}(\bbR^n)$ is a weak solution to 
\begin{equation*}
    (-\Delta_p)^sw=f+g\quad\text{in }B_{2R}(x_0)
\end{equation*}
satisfying 
\begin{equation}\label{norm.loc}
\begin{aligned}
[w]_{W^{s,p}(\bbR^n)}^{p}&\leq \int_{\setR^n}\int_{\setR^n}\left|\dfrac{u(x)-u(y)}{\abs{x-y}^s}\right|^p\xi(x)^p\dfrac{\,dx\,dy}{\abs{x-y}^n}\\
&\quad+\int_{\setR^n}\int_{\setR^n}\abs{u(y)}^p\left|\dfrac{\xi(x)-\xi(y)}{\abs{x-y}^s}\right|^p\dfrac{\,dx\,dy}{\abs{x-y}^n}\\
&\leq [u]_{W^{s,p}(B_{5R}(x_0))}^{p}+cs^{-1}R^{-sp}\|u\|_{L^{p}(B_{5R}(x_0))}^{p}
\end{aligned}
\end{equation}
for some constant $c=c(n,p)$.
We are going to prove \eqref{ineq.loc} and $g\in L^{p'}(B_{5R/2}(x_0))$.
To do this, first note that
\begin{align}
\label{el : lem.loc.ineq1}
|x-y|\geq {|y-x_0|}/{6}
\end{align}
for any $x\in B_{{5R}/{2}}(x_0)$ and $y\in \bbR^{n}\setminus B_{3R}(x_0)$. Using \eqref{defn.avftn} and \eqref{el : lem.loc.ineq1} along with the fact that $|w|\leq |u|$, 
\begin{align*}
    |g(x)|&\leq c(1-s)\int_{\bbR^{n}\setminus B_{3R}(x_0)}\phi'\left(\frac{|u(x)|}{|y-x_0|^{s}}\right)\frac{\,dy}{{|y-x_0|^{n+s}}}\\
    &\quad+c(1-s)\int_{\bbR^{n}\setminus B_{3R}(x_0)}\phi'\left(\frac{|u(y)|}{|y-x_0|^{s}}\right)\frac{\,dy}{{|y-x_0|^{n+s}}}
\end{align*}
holds for any $x\in B_{5R/2}(x_0)$, where $c=c(n,p)$.  Using the inequality 
\begin{equation}\label{ineq1.loclem}
    \int_{\bbR^{n}\setminus B_{3R}(x_0)}\frac{1}{|y-x_0|^{n+sm}}\,dy\leq \frac{c(n)}{sm}\frac{1}{R^{sm}}
\end{equation}
with the choice of $m=p$, for any $x\in B_{5R/2}(x_0)$ we obtain
\begin{align*}
\abs{g(x)}\leq c(1-s)R^{-sp}\abs{u(x)}^{p-1}+cR^{-sp}\mathrm{Tail}(u;B_{3R}(x_0))^{p-1}.
\end{align*}
After a few simple calculations, \eqref{ineq.loc} is attained. Moreover, $g\in L^{p'}(B_{5R/2}(x_0))$ since $u\in L^{p-1}_{sp}(\setR^n)$.

We are now in the position to prove \eqref{ineq1.loc}.
For any $x_{1},x_{2}\in B_{{5R}/{2}}(x_0)$, observe that
\begin{align*}
    g(x_{1})-g(x_{2})&=2(1-s)\int_{\mathbb{R}^{n}\setminus B_{3R}(x_0)}{A}\left(\frac{w(x_{1})-w(y)}{|x_{1}-y|^{s}}\right)\frac{\,dy}{|x_{1}-y|^{n+s}}\\
    &\quad-2(1-s)\int_{\mathbb{R}^{n}\setminus B_{3R}(x_0)}{A}\left(\frac{w(x_{2})-w(y)}{|x_{2}-y|^{s}}\right)\frac{\,dy}{|x_{2}-y|^{n+s}}\\
    &\quad-2(1-s)\int_{\mathbb{R}^{n}\setminus B_{3R}(x_0)}{A}\left(\frac{u(x_{1})-u(y)}{|x_{1}-y|^{s}}\right)\frac{\,dy}{|x_{1}-y|^{n+s}}\\
    &\quad+2(1-s)\int_{\mathbb{R}^{n}\setminus B_{3R}(x_0)}{A}\left(\frac{u(x_{2})-u(y)}{|x_{2}-y|^{s}}\right)\frac{\,dy}{|x_{2}-y|^{n+s}}.
\end{align*}
We first estimate the term $J$ given by
\begin{equation}
\label{el : defn.j1}
\begin{aligned}
    J&\coloneqq\int_{\mathbb{R}^{n}\setminus B_{3R}(x_0)}{A}\left(\frac{w(x_{1})-w(y)}{|x_{1}-y|^{s}}\right)\frac{\,dy}{|x_{1}-y|^{n+s}}\\
    &\quad-\int_{\mathbb{R}^{n}\setminus B_{3R}(x_0)}{A}\left(\frac{w(x_{2})-w(y)}{|x_{2}-y|^{s}}\right)\frac{\,dy}{|x_{2}-y|^{n+s}}.
\end{aligned}
\end{equation}
Let us rewrite $J$ as
\begin{align*}
    J&=\int_{\mathbb{R}^{n}\setminus B_{3R}(x_0)}\left[{A}\left(\frac{w(x_{1})-w(y)}{|x_{1}-y|^{s}}\right)-{A}\left(\frac{w(x_{2})-w(y)}{|x_{1}-y|^{s}}\right)\right]\frac{\,dy}{|x_{1}-y|^{n+s}}\\
    &\quad+\int_{\mathbb{R}^{n}\setminus B_{3R}(x_0)}\left[{A}\left(\frac{w(x_{2})-w(y)}{|x_{1}-y|^{s}}\right)-{A}\left(\frac{w(x_{2})-w(y)}{|x_{2}-y|^{s}}\right)\right]\frac{\,dy}{|x_{1}-y|^{n+s}}\\
    &\quad+\int_{\mathbb{R}^{n}\setminus B_{3R}(x_0)}{A}\left(\frac{w(x_{2})-w(y)}{|x_{2}-y|^{s}}\right)\left[\frac{\,dy}{|x_{1}-y|^{n+s}}-\frac{\,dy}{|x_{2}-y|^{n+s}}\right]\eqqcolon \sum_{i=1}^{3}J_{i}.
\end{align*}
Then the estimates to $J_{i}$ for each $i=1,2$ and 3 are as follows.

\textbf{Estimate of $J_{1}$.}
By \eqref{equi.shif} and using the fact that $\xi(x_1)=\xi(x_2)=1$ for any $x_1,x_2\in B_{5R/2}(x_0)$, we have 
\begin{align}\label{ineq.j1}
    |J_1|\leq c\int_{\bbR^{n}\setminus B_{3R}(x_0)}\phi'_{\frac{|u(x_1)-w(y)|}{|x_1-y|^{s}}}\left(\frac{|u(x_1)-u(x_2)|}{|x_1-y|^{s}}\right)\frac{\,dy}{|x_1-y|^{n+s}}
\end{align}
for some constant $c=c(p)$.

\textbf{Estimate of $J_2$.}
We first note that for any $m\geq0$, $x_1,x_2\in B_{5R/2}(x_0)$ and $y\in \bbR^n\setminus B_{3R}(x_0)$,
\begin{align}
\label{el : lem.loc.ineq2}
    \left|\frac{1}{|x_{1}-y|^{m+s}}-\frac{1}{|x_{2}-y|^{m+s}}\right|=\left|\int_{|x_{1}-y|}^{|x_{2}-y|}\frac{s+m}{t^{s+m+1}}\,dt\right|\leq c\frac{|x_{2}-x_{1}|}{|y-x_0|^{s+m+1}}
\end{align}
holds for some constant $c=c(m)$, where \eqref{el : lem.loc.ineq1} is used.
By Lemma \ref{el : lem.eleineq}, \eqref{el : lem.loc.ineq1}, \eqref{el : lem.loc.ineq2} with $m=0$ and \eqref{ineq1.loclem} with $m=1+s^{-1}$, we get

\begin{align}\label{ineq.j2}
    |J_2|&\leq c\int_{\bbR^{n}\setminus B_{3R}(x_0)}\phi''\left(\frac{|w(x_2)-w(y)|}{|x_1-y|^{s}}+\frac{|w(x_2)-w(y)|}{|x_2-y|^{s}}\right)\nonumber\\
    &\qquad\quad\times\left|\frac{|w(x_2)-w(y)|}{|x_1-y|^{s}}-\frac{|w(x_2)-w(y)|}{|x_2-y|^{s}}\right|\frac{\,dy}{|x_1-y|^{n+s}}\nonumber\\
    &\leq c\int_{\bbR^{n}\setminus B_{3R}(x_0)}\phi''\left(\frac{|w(x_2)-w(y)|}{|y-x_0|^{s}}\right)\left|\frac{|w(x_2)-w(y)|}{|x_1-y|^{s}}-\frac{|w(x_2)-w(y)|}{|x_2-y|^{s}}\right|\nonumber\\
    &\qquad\qquad\times \frac{\,dy}{|y-x_0|^{n+s}}\nonumber\\
    &\leq c\int_{\bbR^{n}\setminus B_{3R}(x_0)}\phi'\left(\frac{|w(x_2)-w(y)|}{|y-x_0|^{s}}\right)|x_1-x_2|\frac{\,dy}{|y-x_0|^{n+s+1}}\\
    &\leq c\left[\phi'\left(\frac{|w(x_2)|}{R^{s}}\right)\frac{1}{R^{s}}+\int_{\bbR^{n}\setminus B_{3R}(x_0)}\phi'\left(\frac{|w(y)|}{|y-x_0|^{s}}\right)\frac{\,dy}{|y-x_0|^{n+s}}\right]\frac{|x_1-x_2|}{R}\nonumber
\end{align}
for some constant $c=c(n,p)$.

\textbf{Estimate of $J_3$.}
Using \eqref{el : lem.loc.ineq1} and \eqref{el : lem.loc.ineq2} with $m=n$, we get 
\begin{align}\label{ineq.j3}
    |J_3|\leq c\left[\phi'\left(\frac{|w(x_2)|}{R^{s}}\right)\frac{1}{R^{s}}+\int_{\bbR^{n}\setminus B_{3R}(x_0)}\phi'\left(\frac{|w(y)|}{|y-x_0|^{s}}\right)\frac{\,dy}{|y-x_0|^{n+s}}\right]\frac{|x_1-x_2|}{R}
\end{align}
for some constant $c=c(n,p)$. As in the estimates of $J_1$, $J_2$ and $J_3$ with $w$ replaced by $u$, the following term 
\begin{equation*}
\begin{aligned}
    \int_{\mathbb{R}^{n}\setminus B_{3}}{\phi'}\left(\frac{u(x_{1})-u(y)}{|x_{1}-y|^{s}}\right)\frac{\,dy}{|x_{1}-y|^{n+s}}-\int_{\mathbb{R}^{n}\setminus B_{3}}{\phi'}\left(\frac{u(x_{2})-u(y)}{|x_{2}-y|^{s}}\right)\frac{\,dy}{|x_{2}-y|^{n+s}}
\end{aligned}
\end{equation*}
is also estimated to obtain \eqref{ineq.j1}, \eqref{ineq.j2} and \eqref{ineq.j3} with $w$ replaced by $u$.
Therefore, combining all the estimates along with the fact that $|w|\leq|u|$ yields \eqref{ineq1.loc}.
\end{proof}
%%%%%%%%%%%%%%%%%%%%%%%%%%%%%%%%%%%%%%%%%%%%
%%%%%%%%%%%%%%%%%%

%%%%%%%%%%%%%%%%%%%%%%%%%%%%%%%%%%%%%%%%%%%%%
Using the previous lemma, we now prove the following Sobolev regularity of the right-hand side $g$ in \eqref{eq:f+g} in the case when $p\geq2$.
\begin{corollary}
\label{el : lem.loc.deg}
Let $B_{5R}(x_{0})\Subset\Omega$ and $u\in W^{\sigma,p}_{\mathrm{loc}}(\Omega)\cap L^{p-1}_{sp}(\bbR^{n})$ be a weak solution to \eqref{pt : eq.main} for some $p\geq2$ and $\sigma\in[s,1)$. Fix a cut-off function $\xi \in C_{c}^{\infty}\left(B_{4R}(x_{0})\right)$ with $\xi\equiv 1$ on $B_{3R}(x_{0})$. Then $w\coloneqq u\xi\in W^{\sigma,p}(\bbR^{n})$ is a weak solution to
\begin{equation*}
    (-\Delta_p)^sw=f+g\quad\text{in }B_{2R}(x_{0}),
\end{equation*}
where $g\in W^{t,p'}(B_{5R/2}(x_0))$ for any $t\in(0,\sigma)$. In addition, the estimate 
\begin{align*}
    (\sigma-t)[g]_{W^{t,p'}(B_{\frac{5R}{2}}(x_0))}^{p'}&\leq c(1-s)^{p'}R^{-p'(t+sp)}\|u\|^{p}_{L^{p}(B_{3R}(x_0))}\\
    &\quad+cR^{n-p'(t+sp)}\mathrm{Tail}(u;B_{3R}(x_0))^{p}\\
    &\quad+c(1-s)^{p'}(\sigma-t)R^{-p'(t+sp)+\sigma p}[u]_{W^{\sigma,p}(B_{3R}(x_0))}^{{p}}
\end{align*}
holds for some $c=c(n,s_0,p)$, where the constant $s_0$ is determined in \eqref{s0assum}.
Moreover, if $\sigma=1$, then we obtain $g\in W^{1,p'}(B_{5R/2}(x_0))$ with the estimate
\begin{align*}
        \|\nabla g\|^{p'}_{L^{p'}(B_{\frac{5R}{2}}(x_0))}&\leq c(1-s)^{p'}R^{-p'(1+sp)}\|u\|^p_{L^{p}(B_{3R}(x_0))}\\
         &\quad+cR^{n-p'(1+sp)}\mathrm{Tail}(u;B_{3R}(x_0))^{p}\\
        &\quad+c(1-s)^{p'}R^{-p'(1+sp)+p}\|\nabla u\|^{p}_{L^{p}(B_{3R}(x_0))}
    \end{align*}
    for some constant $c=c(n,s_0,p)$.
\end{corollary}

\begin{proof}
Let us consider the case $t\in(0,\sigma)$. By Lemma \ref{el : lem.loc}, $w\in W^{\sigma,p}(\bbR^{n})$ is a weak solution to 
\begin{equation*}
    (-\Delta_p)^sw=f+g\quad\text{in }B_{2R}(x_{0})
\end{equation*}
with 
\begin{align}\label{esti.g}
    |g(x_1)-g(x_2)|
    &\leq c(1-s)\int_{\bbR^{n}\setminus B_{3R}(x_0)}\phi'_{\frac{|u(x_1)-w(y)|}{|x_1-y|^{s}}}\left(\frac{|u(x_1)-u(x_2)|}{|x_1-y|^{s}}\right)\frac{\,dy}{|x_1-y|^{n+s}}\nonumber\\
    &\, +c(1-s)\int_{\bbR^{n}\setminus B_{3R}(x_0)}\phi'_{\frac{|u(x_1)-u(y)|}{|x_1-y|^{s}}}\left(\frac{|u(x_1)-u(x_2)|}{|x_1-y|^{s}}\right)\frac{\,dy}{|x_1-y|^{n+s}}\nonumber\\
    &\,+c(1-s)\phi'\left(\frac{|u(x_2)|}{R^{s}}\right)\frac{1}{R^{s}}\frac{|x_1-x_2|}{R}\nonumber\\
    &\,+c(1-s)\int_{\bbR^{n}\setminus B_{3R}(x_0)}\phi'\left(\frac{|u(y)|}{|y-x_0|^{s}}\right)\frac{\,dy}{|y-x_0|^{n+s}}\frac{|x_1-x_2|}{R}\nonumber\\
    &\eqqcolon\sum_{i=1}^{4}J_i
\end{align}
for any $x_1,x_2\in B_{5R/2}(x_0)$, where $c=c(n,p)$. By $p\geq 2$, \eqref{equi.shif}, Lemma \ref{el : lem.eleineq}, \eqref{el : lem.loc.ineq1}, \eqref{ineq1.loclem} with $m$ replaced by $p$,
we get
\begin{equation*}
\begin{aligned}
    |J_1|&\leq c(1-s)\int_{\bbR^{n}\setminus B_{3R}(x_0)}\frac{(|u(x_1)|+|u(x_2)|)^{p-2}}{|y-x_0|^{s(p-2)}}\frac{|u(x_1)-u(x_2)|}{|y-x_0|^{s}}\frac{\,dy}{|y-x_0|^{n+s}}\\
    &\quad+ c(1-s)\int_{\bbR^{n}\setminus B_{3R}(x_0)}\frac{|w(y)|^{p-2}}{|y-x_0|^{s(p-2)}}\frac{|u(x_1)-u(x_2)|}{|y-x_0|^{s}}\frac{\,dy}{|y-x_0|^{n+s}}\\
    &\leq \frac{c(1-s)^{\frac{1}{p-1}}}{R^{sp}}\left[(1-s)^{\frac{1}{p-1}}(|u(x_1)|+|u(x_2)|)+\mathrm{Tail}(u;B_{3})\right]^{p-2}|u(x_1)-u(x_2)|
\end{aligned}
\end{equation*}
for some constant $c=c(n,s_0,p)$ with the help of H\"older's inequality and the fact that $|w|\leq|u|$ for the last inequality. Similarly, we estimate $J_2$ as 
\begin{align*}
    |J_2| &\leq \frac{c(1-s)^{\frac{1}{p-1}}}{R^{sp}}\left[(1-s)^{\frac{1}{p-1}}(|u(x_1)|+|u(x_2)|)+\mathrm{Tail}(u;B_{3})\right]^{p-2}|u(x_1)-u(x_2)|
\end{align*}
with some constant $c=c(n,s_0,p)$. After a few simple calculations along with the fact that $|\phi'(a)|\leq c|a|^{p-1}$ for some constant $c=c(p)$, we estimate $J_3+J_4$ as 
\begin{align}\label{ineq.estj3j4}
    |J_3+J_4|\leq cR^{-sp}\left[(1-s)|u(x_2)|^{p-1}+\mathrm{Tail}(u;B_{3R}(x_0))^{p-1}\right]\frac{|x_1-x_2|}{R}.
\end{align}
Therefore, denoting
\begin{equation*}
U=(1-s)^{\frac{1}{p-1}}(|u(x_1)|+|u(x_2)|)+\mathrm{Tail}(u;B_{3R}(x_0)),
\end{equation*}
and using H\"older's inequality together with the fact that $t<\sigma$, we have 
\begin{equation}\label{j1j2sum}
\begin{aligned}
    &\sum_{i=1}^{2}\int_{B_{\frac{5R}{2}}(x_0)}\int_{B_{\frac{5R}{2}}(x_0)}|J_i|^{p'}\frac{\,dx_1\,dx_2}{|x_1-x_2|^{n+tp'}}\\
    &\leq c\frac{(1-s)^{\frac{p'}{p-1}}}{R^{spp'}}\int_{B_{\frac{5R}{2}}(x_0)}\int_{B_{\frac{5R}{2}}(x_0)}\frac{U^{p'(p-2)}|u(x_1)-u(x_2)|^{p'}}{|x_1-x_2|^{n+tp'}}\,dx_1\,dx_2\\
    &\leq c\frac{(1-s)^{\frac{p'}{p-1}}}{R^{spp'}}\left(\int_{B_{\frac{5R}{2}}(x_0)}\int_{B_{\frac{5R}{2}}(x_0)}\frac{U^{p}}{|x_1-x_2|^{n+\frac{(t-\sigma)p}{p-2}}}\,dx_1\,dx_2\right)^{\frac{p-2}{p-1}}\\
    &\qquad\times\left(\int_{B_{\frac{5R}{2}}(x_0)}\int_{B_{\frac{5R}{2}}(x_0)}\frac{|u(x_1)-u(x_2)|^{p}}{|x_1-x_2|^{n+\sigma p}}\,dx_1\,dx_2\right)^{\frac{1}{p-1}}\\
    &\leq \frac{c(1-s)^{\frac{p'}{p-1}}}{(\sigma-t)^{\frac{p-2}{p-1}}}R^{p'(\sigma-t-sp)}\|(1-s)^{\frac{1}{p-1}}u\|^{p'(p-2)}_{L^{p}\big(B_{\frac{5R}{2}}(x_0)\big)}[u]^{p'}_{W^{\sigma,p}\big(B_{\frac{5R}{2}}(x_0)\big)}\\
    &\quad+\frac{c(1-s)^{\frac{p'}{p-1}}}{(\sigma-t)^{\frac{p-2}{p-1}}}R^{n+p'(\sigma-t-sp)}\mathrm{Tail}(u;B_{3R}(x_0))^{p'(p-2)}[u]^{p'}_{W^{\sigma,p}\big(B_{\frac{5R}{2}}(x_0)\big)}
\end{aligned}
\end{equation}
for some constant $c=c(n,p)$.

On the other hand,
\begin{equation}\label{j3j4sum}
\begin{aligned}
    &\sum_{i=3}^{4}\int_{B_{\frac{5R}{2}}(x_0)}\int_{B_{\frac{5R}{2}}(x_0)}|J_i|^{p'}\frac{\,dx_1\,dx_2}{|x_1-x_2|^{n+tp'}}\\
    &\leq \int_{B_{\frac{5R}{2}}(x_0)}\int_{B_{\frac{5R}{2}}(x_0)}\frac{c
    (1-s)^{p'}|u(x_2)|^{p}+\mathrm{Tail}(u;B_{3R}(x_0))^{p}}{R^{(sp+1)p'}|x_1-x_2|^{n+(t-1)p'}}{\,dx_1\,dx_2}{}\\
    &\leq \frac{cR^{p'(-t-sp)}}{1-t}\left((1-s)^{p'}\|u\|^{p}_{L^{p}\big(B_{\frac{5R}{2}}(x_0)\big)}+ R^n\mathrm{Tail}(u;B_{3R}(x_0))^{p}\right)
\end{aligned}
\end{equation}
holds for some constant $c=c(n,p)$.
Combining all the estimates \eqref{j1j2sum} and \eqref{j3j4sum} along with Young's inequality yields the desired estimate.
We now consider the case when $\sigma=1$. Recall that
    \begin{align*}
        g(x)& =
        2(1-s)\int_{\mathbb{R}^{n}\setminus B_{3R}(x_0)}{\phi'}\left(\frac{w(x)-w(y)}{|x-y|^{s}}\right)\frac{\,dy}{|x-y|^{n+s}}\\
    &\quad-2(1-s)\int_{\mathbb{R}^{n}\setminus B_{3R}(x_0)}{\phi'}\left(\frac{u(x)-u(y)}{|x-y|^{s}}\right)\frac{\,dy}{|x-y|^{n+s}}\coloneqq g_1(x)-g_2(x).
    \end{align*}
    For any $x\in B_{5R/2}(x_0)$ and $|h|<R/1000$, the fundamental theorem of calculus yields
    \begin{align*}
        \delta_h g_1(x)&= 2(1-s)\int_{\bbR^{n}\setminus B_{3R}(x_0)}\int_{0}^1\phi''\left(\frac{(tw(x+h)+(1-t)w(x))-w(y)}{|x+h-y|^{s}}\right)\,dt\\
        &\quad\qquad\qquad\times\frac{w(x+h)-w(x)}{|x+h-y|^{n+2s}}\,dy\\
        &-2(1-s)\int_{\bbR^{n}\setminus B_{3R}(x_0)}\int_{0}^1\phi''\left(\frac{w(x)-w(y)}{|x+th-y|^{s}}\right)\frac{s(x+th-y)\cdot h}{|x+th-y|^{s+2}}\,dt\\
        &\qquad\quad\qquad\qquad\qquad\qquad\qquad\times  \frac{(w(x)-w(y))}{|x+h-y|^{n+s}}\,dy\\
        &-2(1-s)\int_{\bbR^{n}\setminus B_{3R}(x_0)}\phi'\left(\frac{w(x)-w(y)}{|x-y|^{s}}\right)\int_{0}^{1}\frac{(n+s)(x+th-y)\cdot h}{|x+th-y|^{n+s+2}}\,dt\,dy\\
        &\eqqcolon \sum_{i=1}^{3}I_i.
    \end{align*}
    We note that for any $t\in[0,1]$, $x\in B_{5R/2}(x_0)$, $y\in\bbR^n\setminus B_{3R}(x_0)$ and $|h|<R/1000$,
    \begin{equation*}
        |x+th-y|\geq c|y-x_0|
    \end{equation*} holds for some constant $c$.
    Therefore, using $p\geq 2$, $\phi''(t)=(\phi')'(t)\leq c t^{p-2}$ for $t\geq 0$ and Young's inequality, we have 
    \begin{align*}
        |I_1|&\leq c(1-s)\int_{\bbR^{n}\setminus B_{3R}(x_0)}\left[\frac{(|w(x)|+|w_h(x)|)^{p-2}|\delta_h w(x)|}{|y-x_0|^{n+sp}}+\frac{|w(y)|^{p-2}|\delta_h w(x)|}{|y-x_0|^{n+sp}}\right]dy\\
        &\leq c(1-s)R^{-sp}{(|w(x)|+|w_h(x)|)^{p-2}|\delta_hw(x)|}\\
        &\quad+cR^{-sp}\mathrm{Tail}(w;B_{3R}(x_0))^{p-2}|\delta_h w(x)|\\
        &\leq c(1-s)R^{-(1+sp)}\left[{(|w(x)|+|w_h(x)|)^{p-1}+(R|\delta_hw(x)|)^{p-1}}\right]\\
        &\quad+ cR^{-(1+sp)}\left[\mathrm{Tail}(w;B_{3R}(x_0))^{p-1}+(R|\delta_hw(x)|)^{p-1}\right]
    \end{align*}
    for some constant $c=c(n,s_0,p)$.
    On the other hand, with $\phi''(t)t\leq ct^{p-1}$ for $t\geq 0$, $I_2$ and $I_3$ are estimated as
    \begin{align*}
        |I_2|+|I_3|&\leq c|h|(1-s)R^{-sp-1}(|w(x)|+|w_h(x)|)^{p-1}\\
        &\quad+c|h|R^{-sp-1}\mathrm{Tail}(w;B_{3R}(x_0))^{p-1}
    \end{align*}
    for some constant $c=c(n,s_0,p)$. Since $w\in W^{1,p}(\bbR^n)$, we observe 
    \begin{align*}
        \int_{B_{5R/2}(x_0)}|\delta_h w|^p\,dx\leq \abs{h}^p\int_{B_{3R}(x_0)}|\nabla w|^p\,dx.
    \end{align*}
    Using this along with the estimate $I_{i}$ for each $i=1,2$ and $3$, there holds
    \begin{align*}
        \left\|\frac{\delta_h g_1}{|h|}\right\|_{L^{p'}(B_{5R/2}(x_0))}^{p'}&\leq c(1-s)^{p'}R^{-p'(1+sp)}\|w\|^p_{L^{p}(B_{3R}(x_0))}\\
        &\quad+cR^{n-p'(1+sp)}\mathrm{Tail}(w;B_{3R}(x_0))^{p}\\
        &\quad+c(1-s)^{p'}R^{-p'(1+sp)+p}\|\nabla w\|^{p}_{L^{p}(B_{3R}(x_0))}
    \end{align*}
    for some constant $c=c(n,s_0,p)$. Thus, by the standard difference quotient approximations (See \cite[Page 292]{Eva98}) together with the fact that $w=u\xi$,
    \begin{align*}
        \|\nabla g_1\|^{p'}_{L^{p'}(B_{{5R}/{2}}(x_0))}&\leq c(1-s)^{p'}R^{-p'(1+sp)}\|u\|^p_{L^{p}(B_{3R}(x_0))}\\
        &\quad+cR^{n-p'(1+sp)}\mathrm{Tail}(u;B_{3R}(x_0))^{p}\\
        &\quad+c(1-s)^{p'}R^{-p'(1+sp)+p}\|\nabla u\|^{p}_{L^{p}(B_{3R}(x_0))}
    \end{align*}
    for some constant $c=c(n,s_0,p)$. Similarly, the same upper bound for the term $\|\nabla g_2\|^{p'}_{L^{p'}(B_{{5R}/{2}}(x_0))}$ holds. Therefore, we obtain the desired estimate. 
\end{proof}

We next provide an analogous version of Corollary \ref{el : lem.loc.deg} when $p\leq2$.
\begin{corollary}
\label{el : lem.loc.sin}
Let $B_{5R}(x_{0})\Subset\Omega$ and $u\in W^{\sigma,p}_{\mathrm{loc}}(\Omega)\cap L^{p-1}_{sp}(\bbR^{n})$ be a weak solution to \eqref{pt : eq.main} for some $p\leq2$ and $\sigma\in[s,1)$. Fix a cut-off function $\xi \in C_{c}^{\infty}\left(B_{4R}(x_{0})\right)$ with $\xi\equiv 1$ on $B_{3R}(x_{0})$. Then $w\coloneqq u\xi\in W^{\sigma,p}(\bbR^{n})$ is a weak solution to
\begin{equation*}
    (-\Delta_p)^sw=f\quad\text{in }B_{2R}(x_{0}),
\end{equation*}
where $f\in W^{\sigma(p-1),p'}(B_{5R/2}(x_0))$. In addition, we have the estimate 
\begin{align*}
    [f]_{W^{\sigma(p-1),p'}(B_{\frac{5R}{2}}(x_0))}^{p'}&\leq c(1-s)^{p'}R^{-spp'-\sigma p}(1-\sigma(p-1))^{-1}\|u\|^{p}_{L^{p}(B_{3R}(x_0))}\\
    &\quad+cR^{n-spp'-\sigma p}(1-\sigma(p-1))^{-1}\mathrm{Tail}(u;B_{3R}(x_0))^{p}\\
    &\quad+c(1-s)^{p'}R^{-spp'}[u]_{W^{\sigma,p}(B_{3R}(x_0))}^{{p}},
\end{align*}
where $c=c(n,s_0,p)$.
\end{corollary}
\begin{proof}
Lemma \ref{el : lem.loc} yields that $w\in W^{\sigma,p}(\bbR^{n})$ is a weak solution to 
\begin{equation*}
    (-\Delta_p)^sw=f+g\quad\text{in }B_{2R}(x_{0})
\end{equation*}
with \eqref{esti.g}. By \eqref{ineq2.shif}, we observe 
\begin{align*}
    J_1+J_2&\leq c(1-s)\int_{\bbR^n\setminus B_{3R}(x_0)}\frac{|u(x_1)-u(x_2)|^{p-1}}{|x_1-y|^{n+sp}}\,dy\\
    &\leq cR^{-sp}(1-s)|u(x_1)-u(x_2)|^{p-1}
\end{align*}
for some constant $c=c(n,s_0,p)$, where $J_1$ and $J_2$ are determined in \eqref{esti.g}. Thus
\begin{align*}
    &\sum_{i=1}^{2}\int_{B_{\frac{5R}{2}}(x_0)}\int_{B_{\frac{5R}{2}}(x_0)}|J_i|^{p'}\frac{\,dx_1\,dx_2}{|x_1-x_2|^{n+\sigma p}}\\
    &\leq c(1-s)^{p'}R^{-spp'}\int_{B_{\frac{5R}{2}}(x_0)}\int_{B_{\frac{5R}{2}}(x_0)}{|u(x_1)-u(x_2)|^p}\frac{\,dx_1\,dx_2}{|x_1-x_2|^{n+\sigma p}},
\end{align*}
where $c=c(n,s_0,p)$. Using \eqref{ineq.estj3j4}, we arrive at
\begin{align*}
    &\sum_{i=3}^{4}\int_{B_{\frac{5R}{2}}(x_0)}\int_{B_{\frac{5R}{2}}(x_0)}|J_i|^{p'}\frac{\,dx_1\,dx_2}{|x_1-x_2|^{n+\sigma p}}\\
    &\leq \frac{c}{1-(p-1)\sigma}R^{-spp'-\sigma p}\left((1-s)^{p'}\|u\|^{p}_{L^{p}\big(B_{\frac{5R}{2}}(x_0)\big)}+R^n\mathrm{Tail}(u;B_{3R}(x_0))^{p}\right)
\end{align*}
for some constant $c=c(n,s_0,p)$, where $J_3$ and $J_4$ are determined in \eqref{esti.g}.
Combining the above two estimates, the desired estimate is obtained.
\end{proof}
%%%%%%%%%%%%%%%%%%%%%%%%%%%%%%%%%%%%%%
\section{Non-differentiable data} \label{sec:nddata}
In this section, we establish higher differentiability of weak solutions to \eqref{pt : eq.main} in the case when the right-hand side $f$ belongs to $L^{p'}$.

For the remaining sections, for any $t\in[0,1]$ we shall use the notation
\begin{align}\label{nota}
    \sigma_{x,y} u =  \frac{u(x)+u(y)}{2} \quad\text{and}\quad\delta_{x,y}^{t}u=\frac{u(x)-u(y)}{|x-y|^{t}}.
\end{align}

Denoting $B_r=B_r(0)$, we first mention the scaling properties of our equation.
\begin{lemma}\label{lem.scale}
Let $f\in L^{p'}(B_{R}(x_{0}))$ be given and $w\in W^{s,p}(B_{R}(x_{0}))\cap L^{p-1}_{sp}(\bbR^{n})$ be a weak solution to 
\begin{equation*}
    (-\Delta_p)^sw=f\quad\text{in }B_{R}(x_{0}).
\end{equation*}
Then we see that $w_{R}(x)=\frac{w(R x+x_{0})}{R^{s}}\in W^{s,p}(B_{1})\cap L^{p-1}_{sp}(\bbR^{n})$ is a weak solution to 
\begin{equation*}
    (-\Delta_p)^sw_{R}=f_{R}\quad\text{in }B_{1},
\end{equation*}
where $f_{R}(x)=R^{s}f(R x+x_{0})\in L^{p'}(B_{1})$.
\end{lemma}

Relying on different quotients techniques, we provide the following estimate in terms of the $V$-function defined in \eqref{defn.avftn}.
\begin{lemma}\label{lem.combvest}
    Let $u\in W^{\gamma,p}(\bbR^{n})$ be a weak solution to 
    \begin{equation}\label{maineq.comb}
        (-\Delta_p)^su=f\quad\text{in }B_{2}
    \end{equation}
    with $f\in L^{p'}(B_{2})$ and $\gamma\in[s,1)$. Let us take a cut-off function $\psi\in C_{c}^{\infty}(B_{3/4})$ with $\psi\equiv 1$ in $B_{1/2}$ and
\begin{equation}\label{test.sing}
    |\nabla \psi|+ |\nabla^{2}\psi| \leq c 
\end{equation}
for some constant $c=c(n)$. Then for any $|h|<1/1000$, we have
    \begin{equation}\label{ineq.vp}
    \begin{aligned}
        &(1-s)\int_{\bbR^{n}}\int_{\bbR^{n}}\left|\delta_h V\left(\delta_{x,y}^su\right)\right|^{2}\psi(x)\psi(y)\frac{\,dx\,dy}{|x-y|^{n}}\\
        &\leq c\left(\left((1-s)^{\frac1p}[u]_{W^{s,p}(\bbR^{n})}\right)^{p-1}+\|f\|_{L^{p'}(B_1)}\right)\\
        &\quad\times \left(|h|^{\gamma}(1-\gamma)^{\frac{1}{p}}[\delta_h (u\psi)]_{W^{\gamma,p}(\bbR^{n})}+|h|^2\|u\|_{L^{p}(\bbR^{n})}+|h|\left\|{\delta_h u}\right\|_{L^{p}(\bbR^{n})}\right)
    \end{aligned}
    \end{equation}
    for some constant $c=c(n,s_0,p)$, where the constant $s_0$ is determined in \eqref{s0assum}.
\end{lemma}
\begin{proof}
 Let us fix $h\in B_{1/1000}\setminus \{0\}$. 
By recalling definitions \eqref{defn.avftn} and testing $\delta_{-h}\left(\psi^{2}\delta_h u\right)$ to \eqref{maineq.comb}, we get 
\begin{align*}
    0&=(1-s)\int_{\bbR^{n}}\int_{\bbR^{n}}\delta_h A(\delta_{x,y}^{s}u){\delta_{x,y}^{s}(\psi^{2}\delta_{h}u)}\frac{\,dx\,dy}{|x-y|^{n}}-\int_{B_2}f\delta_{-h}\left(\psi^{2}\delta_h u\right)\,dx\\
    &\eqqcolon (1-s)I_1-I_2.
\end{align*}

\noindent
\textbf{Estimate of $I_1$.}
We first rewrite $I_1$ as 
\begin{align*}
    {I}_1 &= \int_{\RRn} \int_{\RRn} \delta_h A(\delta^s_{x,y} u) (\sigma_{x,y} \psi^2) \delta^s_{x,y} \delta_h u \dxyn\\
    &\quad+ \int_{\RRn} \int_{\RRn} \delta_h A(\delta^s_{x,y}u) (\delta^s_{x,y} \psi^2) \sigma_{x,y} \delta_h u \dxyn
    \\
    &\eqqcolon I_{1,1}+I_{1,2}.
\end{align*}
Using the fact that $\delta^{s}_{x,y}\psi^2=\psi(x)\delta^s_{x,y}\psi+\psi(y)\delta^s_{x,y}\psi$, we can further estimate $I_{1,2}$ as
\begin{align*}
    I_{1,2}&=\int_{\bbR^n}\int_{\bbR^n}\delta_h A(\delta_{x,y}^{s}u)\psi(x)(\delta^{s}_{x,y}\psi )\left(\frac{\delta_h u(y)-\delta_h u(x)}{2}\right)\dxyn\\
    &\quad+\int_{\bbR^n}\int_{\bbR^n}\delta_h A(\delta_{x,y}^{s}u)\psi(x)(\delta^{s}_{x,y}\psi)(\delta_h u(x))\dxyn\\
    &\quad+\int_{\bbR^n}\int_{\bbR^n}\delta_h A(\delta_{x,y}^{s}u)\psi(y)(\delta^{s}_{x,y}\psi) \left(\frac{\delta_h u(x)-\delta_h u(y)}{2}\right)\dxyn\\
    &\quad+\int_{\bbR^n}\int_{\bbR^n}\delta_h A(\delta_{x,y}^{s}u)\psi(y)(\delta^{s}_{x,y}\psi)  (\delta_h u(y))\dxyn.
\end{align*}
    By interchanging $x$ and $y$ along with the fact that $A(-X)=-A(X)$, after some manipulations one can see that 
    \begin{align*}
        I_{1,1}+I_{1,2}&=\int_{\RRn} \int_{\RRn} \delta_h A(\delta^s_{x,y} u) \psi(x)\psi(y) (\delta^s_{x,y} \delta_h u )\dxyn\\
        &\quad+2\int_{\bbR^n}\int_{\bbR^n}\delta_h A(\delta_{x,y}^{s}u)\psi(x)(\delta^{s}_{x,y}\psi)  (\delta_h u(x))\dxyn\eqqcolon J_1+J_2.
    \end{align*}

    Using Lemma \ref{el : lem.eleineq}, we estimate $J_1$ as follows
    \begin{align*}
        J_1\geq \int_{\bbR^{n}}\int_{\bbR^{n}}\frac{1}{c}|\delta_h(V(\delta_{x,y}^{s}u))|^{2}\psi(x)\psi(y)\frac{\,dx\,dy}{|x-y|^{n}},
    \end{align*}
    where $c=c(p)$. On the other hand, we rewrite $J_2$ as 
    \begin{align*}
        J_2&= 2\int_{\bbR^{n}}\int_{\bbR^{n}}\delta_h A(\delta_{x,y}^{s}u)\,(\delta_{h}(u\psi)(x))\,({\delta_{x,y}^{s}\psi})\frac{\,dx\,dy}{|x-y|^{n}}\\
        &\quad-2\int_{\bbR^{n}}\int_{\bbR^{n}}\delta_h A(\delta_{x,y}^{s}u)\,((u_h \delta_h \psi)(x))\,({\delta_{x,y}^{s}\psi})\frac{\,dx\,dy}{|x-y|^{n}}\eqqcolon J_{2,1}+J_{2,2}.
    \end{align*}
    By the changing variables, we deduce that
    \begin{align*}
        J_{2,1}&=2\int_{\bbR^{n}}\int_{\bbR^{n}} A(\delta_{x,y}^{s}u_h)\,(\delta_{h}(u\psi)(x))\,({\delta_{x,y}^{s}\psi})\frac{\,dx\,dy}{|x-y|^{n}}\\
        &\quad-2\int_{\bbR^{n}}\int_{\bbR^{n}}A(\delta_{x,y}^{s}u_h)\,(\delta_{h}(u_h\psi_h)(x))\,({\delta_{x,y}^{s}\psi_h})\frac{\,dx\,dy}{|x-y|^{n}}\\
        &=-2\int_{\bbR^{n}}\int_{\bbR^{n}}A(\delta_{x,y}^{s}u_h)\,(\delta^{2}_{h}(u\psi)(x))\,({\delta_{x,y}^{s}\psi})\frac{\,dx\,dy}{|x-y|^{n}}\\
        &\quad-2\int_{\bbR^{n}}\int_{\bbR^{n}}A(\delta_{x,y}^{s}u_h)\,(\delta_{h}(u_h\psi_h)(x))\,({\delta_{x,y}^{s}\delta_h\psi})\frac{\,dx\,dy}{|x-y|^{n}}\\
        &\eqqcolon J_{2,1,1}+J_{2,1,2}.
    \end{align*}
    Using H\"older's inequality, we estimate $J_{2,1,1}$ as 
    \begin{align}\label{ineq.j211}
        |J_{2,1,1}|&\leq c[u]_{W^{s,p}(\bbR^{n})}^{p-1}\left(\int_{\bbR^{n}}\int_{\bbR^n}|(\delta^2_h(u\psi)(x))\,(\delta^s_{x,y}\psi)|^{p}\dxyn\right)^{\frac{1}{p}}.
    \end{align}
    Next, observe that 
    \begin{align}\label{obser}
        \int_{\bbR^{n}}\int_{\bbR^n}|(\delta^2_h(u\psi)(x))\,(\delta^s_{x,y}\psi)|^{p}\dxyn&\leq \int_{B_1}\int_{B_1}|(\delta^2_h(u\psi)(x))\,(\delta^s_{x,y}\psi)|^{p}\dxyn\nonumber\\
        &\quad+\int_{B_1}\int_{\bbR^n\setminus B_1}|\delta^2_h(u\psi)(x)|^p\frac{|\psi(y)|^p\,dx\,dy}{|x-y|^{n+sp}}\nonumber\\
        &\quad+\int_{\bbR^n\setminus B_1}\int_{B_1}|\delta^2_h(u\psi)(x)|^p\frac{|\psi(x)|^p\,dx\,dy}{|x-y|^{n+sp}}.
    \end{align}
    In view of \eqref{test.sing} and the fact that $\psi\equiv 0$ on $\bbR^{n}\setminus B_{3/4}$ and $\delta^2_h(u\psi)\equiv 0$ on $\bbR^n\setminus B_{7/8}$, we further estimate the right-hand side of the above as 
    \begin{align*}
        \int_{\bbR^{n}}\int_{\bbR^n}|(\delta^2_h(u\psi)(x))\,(\delta^s_{x,y}\psi)|^{p}\dxyn\leq c(1-s)^{-1}\|\delta^2_h(u\psi)\|_{L^p(\bbR^n)}^{p}
    \end{align*}
    for some constant $c=c(n,s_0,p)$.
    Thus, plugging this into \eqref{ineq.j211} along with \eqref{lem.embdiff.a1} in Lemma \ref{lem.embdiff}, we have 
    \begin{align*}
        |J_{2,1,1}|&\leq c(1-s)^{-\frac1p}[u]_{W^{s,p}(\bbR^{n})}^{p-1}\left(\int_{\bbR^{n}}|\delta^2_h(u\psi)|^{p}\,dx\right)^{\frac{1}{p}}\\
        &\leq c(1-s)^{-\frac1p}(1-\gamma)^{\frac{1}{p}}|h|^{\gamma}[u]_{W^{s,p}(\bbR^{n})}^{p-1}[\delta_h(u\psi)]_{W^{\gamma,p}(\bbR^n)}
    \end{align*}
    for some $c=c(n,s_0,p)$.

    With aid of H\"older's inequality, we estimate $J_{2,1,2}$ as 
    \begin{align*}
        |J_{2,1,2}|&\leq c[u]_{W^{s,p}(\bbR^{n})}^{p-1}\left(\int_{\bbR^{n}}\int_{\bbR^n}|(\delta_h(u_h\psi_h)(x))\,(\delta^s_{x,y}\delta_h\psi)|^{p}\dxyn\right)^{\frac{1}{p}}.
    \end{align*}
    Note that for some $c=c(n,p)$,
    \begin{align}\label{ineq.test.grhe}
        |\delta_{x,y}^{1}\delta_h\psi|=\frac{|\delta_h \psi(x)-\delta_h \psi(y)|}{|x-y|}\leq c|h|\|\nabla^{2}\psi\|_{L^{\infty}}\quad\text{and}\quad|\delta_{h}\psi(x)|\leq |h|\|\nabla \psi\|_{L^{\infty}}.
    \end{align}
    Following the same lines as in \eqref{obser} along with \eqref{ineq.test.grhe} and \eqref{test.sing} yields
    \begin{align*}
        |J_{2,1,2}|&\leq c(1-s)^{-\frac1p}|h|[u]_{W^{s,p}(\bbR^n)}^{p-1}\|\delta_{h}(u\psi)\|_{L^{p}(\bbR^{n})}\\
        &\leq c(1-s)^{-\frac1p}|h|[u]_{W^{s,p}(\bbR^n)}^{p-1}\left(\|\delta_{h}u\|_{L^{p}(\bbR^{n})}+\abs{h}\|u\|_{L^{p}(\bbR^{n})}\right),
    \end{align*}
    where $c=c(n,s_0,p)$.
    Combining all the estimates $J_{2,1,i}$ for $i \in \{1,2\}$, we obtain
    \begin{align*}
        |J_{2,1}|&\leq c(1-s)^{-\frac1p}(1-\gamma)^{\frac{1}{p}}|h|^{\gamma}[u]_{W^{s,p}(\bbR^{n})}^{p-1}[\delta_h(u\psi)]_{W^{\gamma,p}(\bbR^n)}\\
        &\quad+c(1-s)^{-\frac1p}|h|[u]_{W^{s,p}(\bbR^n)}^{p-1}\|\delta_{h}(u\psi)\|_{L^{p}(\bbR^{n})},
    \end{align*}
    where $c=c(n,s_0,p)$.

    Changing variables, we next estimate $J_{2,2}$ as 
    \begin{align*}
        J_{2,2}&=-2\int_{\bbR^{n}}\int_{\bbR^{n}}A\left(\delta_{x,y}^{s}u_h\right){(\delta_{x,y}^{s}\psi)\,\delta_h \psi(x)\,u_h(x)}\frac{\,dx\,dy}{|x-y|^{n}}\\
        &\quad+2\int_{\bbR^{n}}\int_{\bbR^{n}}A\left(\delta_{x,y}^{s}u_h\right){(\delta^s_{x,y}\psi_h)\delta_h\psi_h(x)\,u_{2h}(x)}\frac{\,dx\,dy}{|x-y|^{n}}\\
        &=2\int_{\bbR^{n}}\int_{\bbR^{n}}A\left(\delta_{x,y}^{s}u_h\right){(\delta_{x,y}^{s}\delta_h\psi)\,\delta_h \psi(x)\,u_h(x)}\frac{\,dx\,dy}{|x-y|^{n}}\\
        &\quad+2\int_{\bbR^{n}}\int_{\bbR^{n}}A\left(\delta_{x,y}^{s}u_h\right){(\delta_{x,y}^{s}\psi_h)\,\delta^{2}_h \psi(x)\,u_h(x)}\frac{\,dx\,dy}{|x-y|^{n}}\\
        &\quad+2\int_{\bbR^{n}}\int_{\bbR^{n}}A\left(\delta_{x,y}^{s}u_h\right){(\delta_{x,y}^{s}\psi_h)\,\delta_h \psi_h(x)\,\delta_h u_h(x)}\frac{\,dx\,dy}{|x-y|^{n}}\\
        &\eqqcolon \sum_{i=1}^{3}J_{2,2,i}.
    \end{align*}
    Using H\"older's inequality, \eqref{ineq.test.grhe} and the fact that 
    \begin{equation*}
        |x-y|\geq |y|/8\quad\text{for any }x\in B_{7/8} \text{ and }y\in \bbR^n\setminus B_1,
    \end{equation*} we estimate $J_{2,2,1}$ as 
    \begin{align*}
        |J_{2,2,1}|&\leq c|h|[u]_{W^{s,p}(\bbR^n)}^{p-1}\left(\int_{\bbR^n}\int_{B_{7/8}}(|\delta^s_{x,y}\delta_h\psi||u_h(x)|)^{p}\dxyn\right)^{\frac{1}{p}}\\
        &\leq c|h|^{2}[u]_{W^{s,p}(\bbR^n)}^{p-1}\left(\int_{B_1}\int_{B_{7/8}}|u_h(x)|^{p}\frac{\,dx\,dy}{|x-y|^{n-p+sp}}\right)^{\frac{1}{p}}\\
        &\quad+c|h|^{2}[u]_{W^{s,p}(\bbR^n)}^{p-1}\left(\int_{\bbR^{n}\setminus B_1}\int_{B_{7/8}}|u_h(x)|^{p}\frac{\,dx\,dy}{|y|^{n+sp}}\right)^{\frac{1}{p}}\\
        &\leq c(1-s)^{-\frac1p}|h|^{2}[u]_{W^{s,p}(\bbR^n)}^{p-1}\|u\|_{L^{p}(\bbR^n)}
    \end{align*}
    for some constant $c=c(n,s_0,p)$. Similarly, $J_{2,2,2}$ and $J_{2,2,3}$ are estimated as follows
    \begin{align*}
        |J_{2,2,2}|\leq c(1-s)^{-\frac1p}|h|^{2}[u]_{W^{s,p}(\bbR^n)}^{p-1}\|u\|_{L^{p}(\bbR^{n})}
    \end{align*}
   and
    \begin{align*}
        |J_{2,2,3}|\leq c(1-s)^{-\frac1p}|h|[u]_{W^{s,p}(\bbR^n)}^{p-1}\left\|{\delta_h u}\right\|_{L^{p}(\bbR^{n})},
    \end{align*}
    respectively, where $c=c(n,s_0,p)$.
    Combining all the estimates $J_{2,2,i}$ for each $i \in \{1,2,3\}$, we have 
    \begin{align*}
    |J_{2,2}|\leq c(1-s)^{-\frac1p}[u]_{W^{s,p}(\bbR^n)}^{p-1}\left(|h|^{2}\|u\|_{L^{p}(\bbR^{n})}+|h|\left\|{\delta_h u}\right\|_{L^{p}(\bbR^{n})}\right).
    \end{align*}
Therefore, it follows that
\begin{align*}
    I_1&\geq\frac{1}{c} \int_{\bbR^{n}}\int_{\bbR^{n}}|\delta_h(V(\delta_{x,y}^{s}u))|^{2}\frac{\psi(x)\psi(y)\,dx\,dy}{|x-y|^{n}}\\
    &\quad-c(1-s)^{-\frac1p}(1-\gamma)^{\frac1p}|h|^{\gamma}[u]^{p-1}_{W^{s,p}(\bbR^{n})}[\delta_h(u\psi)]_{W^{\gamma,p}(\bbR^{n})}\\
    &\quad-c(1-s)^{-\frac1p}\left(|h|[u]^{p-1}_{W^{s,p}(\bbR^{n})}\|\delta_{h}u\|_{L^{p}(\bbR^{n})}+|h|^{2}[u]^{p-1}_{W^{s,p}(\bbR^{n})}\|u\|_{L^{p}(\bbR^{n})}\right)
\end{align*}
for some constant $c=c(n,s_0,p)$.

We now rewrite $I_2$ as 
\begin{align*}
    I_2&=\int_{\bbR^{n}}f(x)\left[\psi_{-h}^2(x)\big(u(x)-u_{-h}(x)\big)-\psi^{2}(x)\big(u_h(x)-u(x)\big)\right]\,dx\\
    &=\int_{\bbR^{n}}f(x)\left[\delta_{-h}\psi^{2}(x)\delta_h u(x)-\psi^{2}_{-h}(x)\delta^{2}_h u(x-h)\right]\,dx\\
    &=\int_{\bbR^{n}}f(x)\delta_{-h}\psi^{2}(x)\delta_h u(x)\,dx-\int_{\bbR^{n}}f_h(x)\psi^{2}(x)\delta^{2}_h u(x)\,dx\eqqcolon I_{2,1}+I_{2,2}.
\end{align*}
By \eqref{ineq.test.grhe} and H\"older's inequality, we estimate $I_{2,1}$ as
\begin{align*}
    |I_{2,1}|
    &\leq c|h|\|f\|_{L^{p'}(B_1)}\|\delta_{h}u\|_{L^p(\bbR^{n})}
\end{align*}
for some constant $c=c(n,p)$. Again by \eqref{ineq.test.grhe}, H\"older's inequality and \eqref{lem.embdiff.a1} in Lemma \ref{lem.embdiff}, $I_{2,2}$ is estimated as follows: 
\begin{align*}
    |I_{2,2}|&=\left|\int_{\bbR^n}f_h(x)\psi(x)\left[\delta^{2}_h(u\psi)(x)-u_{2h}(x)\delta^2_{h}\psi(x)-2\delta_{h}u_h(x)\delta_h \psi(x)\right]\,dx\right|\\
    &\leq c\|f\|_{L^{p'}(B_1)}\|\delta^{2}_{h}(u\psi)\|_{L^p(\bbR^{n})}+c|h|^{2}\|f\|_{L^{p'}(B_1)}\|u\|_{L^{p}(\bbR^{n})}\\
    &\quad+ c|h|\|f\|_{L^{p'}(B_1)}\|\delta_h u\|_{L^{p}(\bbR^{n})}\\
    &\leq c\|f\|_{L^{p'}(B_1)}\left((1-\gamma)^{\frac{1}{p}}|h|^{\gamma}[\delta_{h}(u\psi)]_{W^{\gamma,p}(\bbR^{n})}+|h|^{2}\|u\|_{L^{p}(\bbR^{n})}\right)\\
    &\quad+c\|f\|_{L^{p'}(B_1)}|h|\|\delta_h u\|_{L^{p}(\bbR^{n})}
\end{align*}
for some constant $c=c(n,p)$.
Combining all the above estimates for $I_1$, $I_{2,1}$ and $I_{2,2}$ now finally yields \eqref{ineq.vp}.
\end{proof}

%%%%%%%%%%%%%%%%%%%%%%
Using the above lemma, we obtain a more refined estimate. The proofs are different for in the cases $p\leq 2$ and $p\geq 2$.
\begin{lemma}\label{lem.vrefd}
    Let $p\geq2$ and assumed that $u\in W^{s,p}(\bbR^{n})$ is a weak solution to 
    $$
        (-\Delta_p)^su=f\quad\text{in }B_2
    $$
    with $f\in L^{p'}(B_{2})$. Fix a cut-off function $\psi\in C_c^{\infty}(B_{3/4})$ with $\psi\equiv1$ in $B_{1/2}$ and \eqref{test.sing}. Then for any $|h|<1/1000$, we have
\begin{align*}
        &(1-s)\int_{\bbR^{n}}\int_{\bbR^n}|\delta_h V(\delta_{x,y}^{s}u)|^{2}\psi(x)\psi(y)\frac{\,dx\,dy}{|x-y|^{n}}\\
        &\leq c|h|^{sp'}\left((1-s)[u]^{p}_{{W}^{s,p}(\bbR^{n})}+\|u\|^p_{L^p(\bbR^n)}+\|f\|^{p'}_{L^{p'}(B_1)}\right)
    \end{align*}
    for some constant $c=c(n,s_0,p)$, where the constant $s_0$ is determined in \eqref{s0assum}.
\end{lemma}

\begin{proof}
First, observe that 
\begin{equation}\label{ineq.simple0}
\begin{aligned}
    |\delta_h(u\psi)(x)-\delta_h(u\psi)(y)|&\leq |(\delta_h u(x)-\delta_h u(y))\psi(x)|\\
    &\quad+|\delta_h u(y)(\psi(x)-\psi(y))|\\
    &\quad+ |(u_h(x)-u_h(y))\delta_h \psi(x)|\\
    &\quad+|u_h(y)(\delta_h \psi(x)-\delta_h \psi(y))|.
\end{aligned}
\end{equation}
We denote $\psi_m=\min\{\psi(x),\psi(y)\}$ to see that
\begin{equation}\label{ineq.simple}
\begin{aligned}
    |\delta_h(u\psi)(x)-\delta_h(u\psi)(y)|&\leq |(\delta_h u(x)-\delta_h u(y))\psi_m|\\   
    &\quad +|(\delta_h u(x)-\delta_h u(y))(\psi(x)-\psi(y))|\\
     &\quad+|\delta_h u(y)(\psi(x)-\psi(y))|\\
    &\quad+ |(u_h(x)-u_h(y))\delta_h \psi(x)|\\
    &\quad+|u_h(y)(\delta_h \psi(x)-\delta_h \psi(y))|.
\end{aligned}
\end{equation}
From this together with \eqref{test.sing} and \eqref{ineq.test.grhe}, it follows that
\begin{align*}
    [\delta_h(u\psi)]^{p}_{W^{s,p}(\bbR^{n})}&\leq c\int_{\bbR^{n}}\int_{\bbR^n}\frac{|\delta_h\delta_{x,y}^{s} u|^{p}\psi^{p}_m}{|x-y|^{n}}\,dx\,dy+\frac{c}{1-s}\|\delta_h u\|^{p}_{L^{p}(\bbR^{n})}\\
    &\quad+c|h|^{p}\left([u]^{p}_{{W}^{s,p}(\bbR^{n})}+(1-s)^{-1}{\|u\|^p_{L^p(\bbR^n)}}\right)
\end{align*}
for some constant $c=c(n,s_0,p)$.
We now further estimate the first term in the right-hand side of the above inequality.
Using Lemma \ref{el : lem.eleineq} along with the fact that $\psi_m^{p}\leq \psi_m^{2}\leq \psi(x)\psi(y)$, we obtain
    \begin{equation}\label{ineq1.vref}
    \begin{aligned}
        [\delta_h(u\psi)]^{p}_{W^{s,p}(\bbR^{n})}&\leq c\int_{\bbR^{n}}\int_{\bbR^{n}}|\delta_h V(\delta_{x,y}^{s}u)|^{2}\psi(x)\psi(y)\frac{\,dx\,dy}{|x-y|^{n}}+\frac{c}{1-s}\|\delta_h u\|^{p}_{L^{p}(\bbR^{n})}\\
        &\quad+c|h|^{p}\left([u]^{p}_{{W}^{s,p}(\bbR^{n})}+(1-s)^{-1}{\|u\|^p_{L^p(\bbR^n)}}\right).
    \end{aligned}
\end{equation}
Plugging \eqref{ineq1.vref} into \eqref{ineq.vp} with $\gamma$ replaced by $s$ and using Young's inequality along with the fact that $|h|^{sp}+|h|^{2}+|h|^{1+s}\leq c|h|^{sp'}$ for some constant $c=c(p)$, we deduce
\begin{align*}
        &(1-s)\int_{\bbR^{n}}\int_{\bbR^{n}}|\delta_h( V(\delta_{x,y}^{s}u))|^{2}\psi(x)\psi(y)\frac{\,dx\,dy}{|x-y|^{n}}\\
        &\leq \frac{1-s}{2}\int_{\bbR^{n}}\int_{\bbR^{n}}|\delta_h( V(\delta_{x,y}^{s}u))|^{2}\psi(x)\psi(y)\frac{\,dx\,dy}{|x-y|^{n}}\\
        &\quad+c|h|^{sp'}\left((1-s)[u]^{p}_{{W}^{s,p}(\bbR^{n})}+\|u\|^p_{L^p(\bbR^n)}+\|f\|^{p'}_{L^{p'}(B_1)}+\left\|\frac{\delta_h u}{|h|^{s}}\right\|_{L^{p}(\bbR^{n})}^p\right)
    \end{align*}
    for some constant $c=c(n,s_0,p)$.
Applying \eqref{lem.embdiff.a1} in Lemma \ref{lem.embdiff} to the last term in the right-hand side of the above inequality, we arrive at the desired estimate.
\end{proof}

We prove a similar version of Lemma \ref{lem.vrefd} when $p\leq 2$.
\begin{lemma}\label{lem.vrefs}
    Let $p\leq2$ and $u\in W^{\gamma,p}(\bbR^{n})$ be a weak solution to 
    \begin{equation*}
        (-\Delta_p)^su=f\quad\text{in }B_2
    \end{equation*}
    with $f\in L^{p'}(B_2)$ and $\gamma\in[s,1)$. Let us take a cut-off function $\psi\in C_c^{\infty}(B_{3/4})$ with $\psi\equiv1$ on $B_{1/2}$ and \eqref{test.sing}. Then for any $|h|<1/1000$, we have
\begin{align*}
        &(1-s)\int_{\bbR^{n}}\int_{\bbR^n}|\delta_h V(\delta_{x,y}^{s}u)|^{2}\psi(x)\psi(y)\frac{\,dx\,dy}{|x-y|^{n}}\\
        &\leq c|h|^{2\sigma}\left((1-\sigma)[u]^{p}_{W^{\gamma,p}(\bbR^{n})}+(1-s)[u]^{p}_{W^{s,p}(\bbR^{n})}+\|u\|_{L^p(\bbR^n)}^p+\|f\|^{p'}_{L^{p'}(B_1)}\right)
    \end{align*}
    for some constant $c=c(n,s_0,p)$, where we denote $\sigma=\gamma+(s-\gamma)p/2$ and the constant $s_0$ is determined in \eqref{s0assum}.
\end{lemma}
\begin{proof}
Using \eqref{ineq.simple} and the fact that $\delta_h\psi \equiv 0$ on $\bbR^n\setminus B_1$, note that
\begin{equation}\label{ineq1.lem.vrefs}
\begin{aligned}
    [\delta_h(u\psi)]^{p}_{W^{\sigma,p}(\bbR^{n})}&\leq c\int_{\bbR^{n}}\int_{\bbR^n}\frac{|\delta_h\delta_{x,y}^{\sigma} u|^{p}\psi^{p}_m}{|x-y|^{n}}\,dx\,dy+\frac{c\|\delta_h u\|^{p}_{L^{p}(\bbR^{n})}}{1-\sigma}\\
    &\quad+c|h|^{p}\int_{\bbR^n}\int_{B_1}\frac{|u_h(x)-u_h(y)|^{p}}{|x-y|^{n+\sigma p}}\,dx\,dy+\frac{c\|u\|^{p}_{L^{p}(\bbR^{n})}}{1-\sigma}
\end{aligned}
\end{equation}
for some constant $c=c(n,s_0,p)$. By the fact that $\sigma\leq \gamma$ and 
\begin{equation*}
    |x-y|\geq |y|/2\quad\text{for any }x\in B_1 \text{ and }y\in \bbR^n\setminus B_2,
\end{equation*}
we observe
\begin{equation*}
\begin{aligned}
    \int_{\bbR^n}\int_{B_1}\frac{|u_h(x)-u_h(y)|^{p}}{|x-y|^{n+\sigma p}}\,dx\,dy&=\int_{B_2}\int_{B_1}\frac{|u_h(x)-u_h(y)|^{p}}{|x-y|^{n+\sigma p}}\,dx\,dy\\
    &\quad+\int_{\bbR^n\setminus B_2}\int_{B_1}\frac{|u_h(x)-u_h(y)|^{p}}{|x-y|^{n+\sigma p}}\,dx\,dy\\
    &\leq c\int_{B_2}\int_{B_1}\frac{|u_h(x)-u_h(y)|^{p}}{|x-y|^{n+\gamma p}}\,dx\,dy\\
    &\quad+c\int_{\bbR^n\setminus B_2}\int_{B_1}\frac{|u_h(x)|^{p}+|u_h(y)|^{p}}{|y|^{n+\sigma p}}\,dx\,dy\\
    &\leq c\|u\|^{p}_{W^{\gamma,p}(\bbR^n)},
\end{aligned}
\end{equation*}
where $c=c(n,s_0,p)$. 

On the other hand, by H\"older's inequality and Lemma \ref{el : lem.eleineq} we have
\begin{equation}\label{ineq2.vref}
\begin{aligned}
    &\int_{\bbR^{n}}\int_{\bbR^n}\frac{|\delta_{x,y}^{\sigma}\delta_h u|^{p}\psi^{p}_m}{|x-y|^{n}}\,dx\,dy\\
    &=\int_{{\bbR^{n}}}\int_{\bbR^{n}}\frac{| \delta_{x,y}^{s}\delta_h u|^{p}\psi^{p}_m(|\delta_{x,y}^{s}u|+|\delta_{x,y}^{s}u_h|)^{\frac{p(p-2)}{2}}}{|x-y|^{n+\frac{(\gamma-s)p(2-p)}{2}}}(|\delta_{x,y}^{s}u|+|\delta_{x,y}^{s}u_h|)^{\frac{p(2-p)}{2}}\,dx\,dy\\
    &\leq \Bigg(\int_{\bbR^{n}}\int_{\bbR^{n}}(|\delta_{x,y}^{s}u|+|\delta_{x,y}^{s}u_h|)^{p-2}\frac{|\delta_{x,y}^{s}\delta_{h}u|^{2}\psi^{2}_m}{|x-y|^{n}}\,dx\,dy\Bigg)^{\frac{p}{2}}[u]_{W^{\gamma,p}(\bbR^{n})}^{\frac{p(2-p)
    }{2}}\\
    &\leq c\left(\int_{\bbR^{n}}\int_{\bbR^{n}}|\delta_h V(\delta_{x,y}^{s}u)|^{2}\frac{\psi_m^{2}\,dx\,dy}{|x-y|^{n}}\right)^{\frac{p}{2}}[u]^{\frac{p(2-p)}{p}}_{W^{\gamma,p}(\bbR^{n})}
\end{aligned}
\end{equation}
for some constant $c=c(p)$, where we denote $\psi_m=\min\{\psi(x),\psi(y)\}$.
Therefore, inserting the above two inequalities into the right-hand side of \eqref{ineq1.lem.vrefs} yields
\begin{equation*}
\begin{aligned}
    [\delta_h(u\psi)]^{p}_{W^{\sigma,p}(\bbR^{n})}&\leq c\left(\int_{\bbR^{n}}\int_{\bbR^{n}}|\delta_h(V(\delta_{x,y}^{s}u))|^{2}\frac{\psi_m^{2}\,dx\,dy}{|x-y|^{n}}\right)^{\frac{p}{2}}[u]^{\frac{p(2-p)}{2}}_{W^{\gamma,p}(\bbR^{n})}\\
    &\quad+\frac{c\|\delta_h u\|^{p}_{L^{p}(\bbR^{n})}}{1-\sigma}+c|h|^{p}\left([ u]^{p}_{W^{\gamma,p}(\bbR^{n})}+\frac{c\|u\|_{L^p(\bbR^n)}^p}{1-\sigma}\right)
\end{aligned}
\end{equation*}
for some constant $c=c(n,s_0,p)$. We now plug this into \eqref{ineq.vp} with $\gamma$ replaced by  $\sigma$ to see that 
\begin{align*}
    &(1-s)\int_{\bbR^{n}}\int_{\bbR^{n}}|\delta_h( V(\delta_{x,y}^{s}u))|^{2}\psi(x)\psi(y)\frac{\,dx\,dy}{|x-y|^{n}}\\
    &\leq c|h|^{\sigma}T(1-\sigma)^{\frac{1}{p}} [u]^{\frac{(2-p)}{2}}_{W^{\gamma,p}(\bbR^n)}\left(\int_{\bbR^{n}}\int_{\bbR^{n}}|\delta_h(V(\delta_{x,y}^{s}u))|^{2}\frac{\psi_m^{2}\,dx\,dy}{|x-y|^{n}}\right)^{\frac{1}{2}}\\
    &\,\,+c\abs{h}^{\sigma}T\\
    &\,\,\,\,\times\left[|h|(1-\sigma)^{\frac1p}[u]_{W^{\gamma,p}(\bbR^{n})}+(1+|h|^{1-\sigma})\|\delta_h u\|_{L^{p}(\bbR^{n})}+(|h|+|h|^{2-\sigma})\| u\|_{L^{p}(\bbR^{n})}\right]
\end{align*}
with $T=(1-s)^{\frac{1}{p'}}[u]^{p-1}_{W^{s,p}(\bbR^n)}+\|f\|_{L^{p'}(B_1)}$. We now use \eqref{lem.embdiff.a1} in Lemma \ref{lem.embdiff} with $\sigma$ replaced by $\gamma$ and Young's inequality twice with exponents $(2,2)$ and then $\left(\frac{p}{2(p-1)},\frac{p}{2-p}\right)$ to obtain 
\begin{align*}
     &(1-s)\int_{\bbR^{n}}\int_{\bbR^{n}}|\delta_h( V(\delta_{x,y}^{s}u))|^{2}\psi(x)\psi(y)\frac{\,dx\,dy}{|x-y|^{n}}\\
     &\leq \frac{1-\sigma}{2}\int_{\bbR^{n}}\int_{\bbR^{n}}|\delta_h( V(\delta_{x,y}^{s}u))|^{2}\psi^{2}_m\frac{\,dx\,dy}{|x-y|^{n}}+c(1-\sigma)|h|^{2\sigma}[u]_{W^{\gamma,p}(\bbR^n)}^{p}\\
     &\quad+c|h|^{2\sigma}\left((1-s)[u]_{W^{s,p}(\bbR^n)}^{p}+c\|u\|_{L^p(\bbR^n)}^p+c\|f\|_{L^{p'}(B_1)}^{p'}\right)
\end{align*}
for some constant $c=c(n,s_0,p)$, where we have used the fact that 
\begin{align*}
    |h|^2+|h|^{1+\gamma}+|h|^{\sigma+\gamma}+|h|^{1+\sigma}\leq c|h|^{2\sigma}
\end{align*}for some constant $c=c(n)$. Performing elementary calculations together with the fact that $s\leq  \sigma$, we arrive at the desired estimate.
\end{proof}

With Lemma \ref{lem.vrefd} and our localization arguments at hand, we now prove the following Lemma for $p\geq2$.

\begin{lemma}\label{loc.mainthmd}
    Let $u\in W^{s,p}(B_5)\cap L^{p-1}_{sp}(\bbR^{n})$ be a weak solution to 
    \begin{equation}\label{eq.scalef}
        (-\Delta_p)^su=f\quad\text{in }B_5,
    \end{equation}
    where $f\in L^{p'}(B_5)$ and $p\geq2$. Then we have 
    \begin{equation}\label{vlast1}
    \begin{aligned}
            \sup_{0<|h|<\frac{1}{1000}}\left\|\frac{\delta_{h}^{2}(u\eta)}{|h|^{sp'}}\right\|^{p}_{L^{p}(\bbR^{n})}&\leq c (1-s)[u]^{p}_{W^{s,p}(B_{5})}+c\|u\|_{L^p(B_5)}^p\\
            &\quad+c\mathrm{Tail}(u;B_{5})^{p}+c\|f\|_{L^{p'}(B_5)}^{p'}
    \end{aligned}
    \end{equation}
    for some constant $c=c(n,s_0,p)$ and for some cut-off function $\eta\in C_{c}^{\infty}(B_{3/4})$ satisfying $\eta\equiv 1$ in $B_{1/2}$ with $|\nabla \eta|\leq c$, where the constant $s_0$ is determined in \eqref{s0assum}. 
\end{lemma}

\begin{proof}
Let us choose a function $\xi\in C_{c}^{\infty}(B_{7/2})$ such that $\xi\equiv 1$ in $B_3$ and 
\begin{equation*}
    |\nabla \xi|\leq 4.
\end{equation*}
By Lemma \ref{el : lem.loc}, $w\coloneqq u\xi\in W^{s,p}(\bbR^{n})$ is a weak solution to 
\begin{equation*}
    (-\Delta_p)^sw=f+g\quad\text{in }B_{2},
\end{equation*}
where $g\in L^{p'}(B_{5/2})$ with 
\begin{equation}\label{fnormineqs}
\norm{g}^{p'}_{L^{p'}(B_{5/2})}\leq c(1-s)^{p'}\norm{u}^p_{L^p(B_3)}+c\mathrm{Tail}(u;B_3)^{p}
\end{equation}
for some constant $c=c(n,s_0,p)$. In addition, we observe from \eqref{norm.loc} that
\begin{align}\label{wspnorm}
    \|w\|_{W^{s,p}(\bbR^{n})}\leq c\|u\|_{W^{s,p}(B_{5})}.
\end{align}
Let us fix a cut-off function $\psi\in C_{c}^{\infty}(B_{3/4})$ such that $\psi\equiv 1$ in $B_{1/2}$ with \eqref{test.sing}. By Lemma \ref{lem.vrefd}, we have
\begin{equation}\label{ineq.v1}
\begin{aligned}
        &(1-s)\int_{\bbR^n}\int_{\bbR^n}|\delta_h V(\delta_{x,y}^{s}w)|^{2}\psi(x)\psi(y)\frac{\,dx\,dy}{|x-y|^{n}}\\
        &\leq c|h|^{sp'}\left((1-s)[w]^{p}_{W^{s,p}(\bbR^{n})}+\|w\|_{L^p(\bbR^n)}^p+\|f+g\|^{p'}_{L^{p'}(B_1)}\right)
    \end{aligned}
    \end{equation}
    for any $|h|<1/1000$, where $c=c(n,s_0,p)$. 
    
    Note from \eqref{lem.embdiff.a1} in Lemma \ref{lem.embdiff} and \eqref{ineq.simple} that
        \begin{align*}
            \left\|\frac{\delta_{h}^{2}(w\psi)}{|h|^{sp'}}\right\|^{p}_{L^{p}(\bbR^{n})}&\leq c(1-s)\left[\frac{\delta_{h}(w\psi)}{|h|^{s/(p-1)}}\right]^{p}_{W^{s,p}(\bbR^{n})}\nonumber\\
            &\leq c(1-s)\left[\frac{\psi_m\delta_{h}w}{|h|^{s/(p-1)}}\right]^{p}_{W^{s,p}({\bbR^n})}+c\left\|\frac{\delta_h w}{|h|^{s/(p-1)}}\right\|^{p}_{L^{p}(\bbR^{n})}\\
            &\quad+c(1-s)|h|^{p-sp'}[w]^{p}_{W^{s,p}(\bbR^n)}+c|h|^{p-sp'}\|w\|^p_{L^p(\bbR^n)}\nonumber\\
            &\leq c(1-s)\left(\left[\frac{\psi_m\delta_{h}w}{|h|^{s/(p-1)}}\right]^{p}_{W^{s,p}({\bbR^n})}+[w]^{p}_{W^{s,p}(\bbR^{n})}\right)+c\|w\|^p_{L^p(\bbR^n)}\nonumber
        \end{align*}
        for some constant $c=c(n,s_0,p)$, where $\psi_m=\min\{\psi(x),\psi(y)\}$ and \eqref{lem.embdiff.a1} are used along with $|h|^{-\frac{s}{p-1}}\leq |h|^{-s}$, if $|h|<\frac{1}{1000}$ for the last inequality.
        On the other hand, using Lemma \ref{el : lem.eleineq} together with the fact that $\psi^p_{m}\leq \psi^{2}_{m}\leq\psi(x)\psi(y)$ as well as \eqref{fnormineqs}--\eqref{ineq.v1},
        \begin{align*}
            \left\|\frac{\delta_{h}^{2}(w\psi)}{|h|^{sp'}}\right\|^{p}_{L^{p}(\bbR^{n})}&\leq c((1-s)[w]^{p}_{W^{s,p}(\bbR^{n})}+\|w\|_{L^p(\bbR^n)}^p+\|f+g\|^{p'}_{L^{p'}(B_1)})\\
            &\leq c \left((1-s)\|u\|^{p}_{W^{s,p}(B_{5})}+\|u\|_{L^p(B_5)}^p+\mathrm{Tail}(u;B_{5})^{p}+\|f\|_{L^{p'}(B_5)}^{p'}\right)
        \end{align*}
        holds for some constant $c=c(n,s_0,p)$.
        % By Lemma \ref{lem.secquo} together with the above estimate, we have  
        % \begin{equation*}
        % \begin{aligned}
        %     [w\psi]^p_{B^{sp',p}_{\infty}(\bbR^{n})}&\leq c\left[\sup_{0<|h|<1/1000}\left\|\frac{\delta_{h}^{2}(w\psi)}{|h|^{sp'}}\right\|^p_{L^{p}(\bbR^{n})}+\|w\psi\|^p_{L^{p}(\bbR^{n})}\right]\\
        %     &\leq c (1-s)[u]^{p}_{W^{s,p}(B_{4})}+c\|u\|^p_{L^{p}(B_4)}+c\mathrm{Tail}(u;B_{4})^{p}+c\|f\|_{L^{p'}(B_1)}^{p'}
        % \end{aligned}
        % \end{equation*}
        % for some constant $c=c(n,s_0,p)$.
        By taking $\eta=\xi\psi$, we obtain \eqref{vlast1}. 
\end{proof}
%%%%%%%%%
The following lemma yields an iterative scheme for the subquadratic case.
\begin{lemma}\label{loc.mainthmdbs}
    Let $u\in W^{\gamma,p}(B_5)\cap L^{p-1}_{sp}(\bbR^{n})$ be a weak solution to 
    \begin{equation}\label{eq.scalefs}
        (-\Delta_p)^su=f\quad\text{in }B_5,
    \end{equation}
    where $f\in L^{p'}(B_5)$, $p\leq 2$ and $\gamma\in[s,\min\left\{1,sp/(p-1)\right\})$. Then we have 
            \begin{equation}\label{vlast2}
            \begin{aligned}
            \sup_{0<|h|<\frac{1}{1000}}\left\|\frac{\delta_{h}^{2}(u\eta)}{|h|^{2\sigma}}\right\|^{p}_{L^{p}(\bbR^{n})}&\leq c (1-\sigma)[u]^{p}_{W^{\gamma,p}(B_{4})}+c\|u\|_{L^p(B_5)}^p\\
            &\quad+c\mathrm{Tail}(u;B_{5})^{p}+c\|f\|_{L^{p'}(B_5)}^{p'}
        \end{aligned}
        \end{equation}
    for some $c=c(n,s_0,p)$ and some $\eta\in C_{c}^{\infty}(B_{3/4})$ satisfying $\eta\equiv 1$ on $B_{1/2}$ with $|\nabla \eta|\leq c$, where $\sigma=\gamma+(s-\gamma)p/2$ and the constant $s_0$ is determined in \eqref{s0assum}. 
\end{lemma}
\begin{proof}Let us fix $h\in \bbR^n$ with $0<|h|<1/1000$.
As in the proof of Lemma \ref{loc.mainthmd}, $w=u\xi$ is a weak solution to 
\begin{equation*}
    (-\Delta_p)^sw=f+g\quad\text{in }B_2,
\end{equation*}
where $f,g\in L^{p'}(B_{5/2})$ with \eqref{fnormineqs} and
\begin{equation}\label{wgpnorm}
    \|w\|_{W^{\gamma,p}(\bbR^n)}\leq c\|u\|_{W^{\gamma,p}(B_5)}
\end{equation}
for some $c=c(n,s_0,p)$, which follows from \eqref{norm.loc} with $s$ replaced by $\gamma$. We now note from \eqref{lem.embdiff.a1} in Lemma \ref{lem.embdiff} and \eqref{ineq.simple} that
        \begin{align*}
         \left\|\frac{\delta_{h}^{2}(w\psi)}{|h|^{2\sigma}}\right\|^{p}_{L^{p}(\bbR^{n})} &\leq c(1-\sigma)\left[\frac{\delta_{h}(w\psi)}{|h|^{\sigma}}\right]^{p}_{W^{\sigma,p}({\bbR^n})}\\
          &\leq c(1-\sigma)\left[\frac{\psi_m\delta_{h}w}{|h|^{\sigma}}\right]^{p}_{W^{\sigma,p}({\bbR^n})}+c\left\|\frac{\delta_h w}{|h|^{\sigma}}\right\|^{p}_{L^{p}(\bbR^{n})}\\
            &\quad+c(1-\sigma)|h|^{p-\sigma p}[w]^{p}_{W^{\sigma,p}(\bbR^n)}+c|h|^{p-\sigma p}\|w\|^{p}_{W^{\sigma,p}(\bbR^n)}\nonumber\\
            &\leq c(1-\sigma)\left(\left[\frac{\psi_m\delta_{h}w}{|h|^{\sigma}}\right]^{p}_{W^{\sigma,p}({\bbR^n})}+[w]^{p}_{W^{\sigma,p}(\bbR^{n})}\right)+c\|w\|^{p}_{W^{\sigma,p}(\bbR^n)}\nonumber
        \end{align*}
        for some $c=c(n,s_0,p)$ with denoting $\psi_m=\min\{\psi(x),\psi(y)\}$. 
        
        We next use \eqref{ineq2.vref} with $u$ replaced by $w$ to see that 
        \begin{align*}
            \left[\frac{\psi_m\delta_{h}w}{|h|^{\sigma}}\right]^{p}_{W^{\sigma,p}({\bbR^n})} &\leq c \left(\int_{\bbR^{n}}\int_{\bbR^{n}}|\delta_h V(\delta_{x,y}^{s}w)|^{2}\frac{\psi_m^{2}\,dx\,dy}{|h|^{2\sigma}|x-y|^{n}}\right)^{\frac{p}{2}}[w]^{\frac{p(2-p)}{2}}_{W^{\gamma,p}(\bbR^{n})}.
        \end{align*}
        Combining the above two inequalities together with Lemma \ref{lem.vrefs} with $u$ replaced by $w$ and Young's inequality yields
        \begin{align*}
            \left\|\frac{\delta_{h}^{2}(w\psi)}{|h|^{2\sigma}}\right\|^{p}_{L^{p}(\bbR^{n})}&\leq c(1-\sigma)\left([w]_{W^{\gamma,p}(\bbR^n)}^p+[w]_{W^{\sigma,p}(\bbR^n)}^p+[w]_{W^{s,p}(\bbR^n)}^p\right)\\
            &\quad+c\|w\|_{L^p(\bbR^n)}^p+c\|f+g\|^{p'}_{L^{p'}(B_1)}
        \end{align*}
        for some $c=c(n,s_0,p)$. Using the fact that $w\equiv 0$ on $\bbR^n\setminus B_{7/2}$, for any $t\in(0,\gamma]$,
        \begin{align*}
            [w]_{W^{t,p}(\bbR^n)}^{p}&\leq\int_{B_4}\int_{B_4}\frac{|w(x)-w(y)|^{p}}{|x-y|^{n+tp}}\,dx\,dy+\int_{\bbR^n\setminus B_4}\int_{B_4}\frac{|w(x)|^{p}}{|x-y|^{n+tp}}\,dx\,dy\\
            &\quad+\int_{B_4}\int_{\bbR^n\setminus B_4}\frac{|w(y)|^{p}}{|x-y|^{n+tp}}\,dx\,dy\\
            &\leq c[w]_{W^{\gamma,p}(B_4)}^p+c\int_{\bbR^n\setminus B_4}\int_{B_{7/2}}\frac{|w(x)|^{p}}{|y|^{n+tp}}\,dx\,dy\\
            &\leq \frac{c}{t}\|w\|_{W^{\gamma,p}(\bbR^n)}^p
        \end{align*}
        holds, where $c=c(n,p)$. 
        Using this along with the fact that $s,\sigma\in[s_0,\gamma]$, Young's inequality, \eqref{fnormineqs} and \eqref{wgpnorm}, we arrive at
        \begin{equation*}
            \left\|\frac{\delta_{h}^{2}(w\psi)}{|h|^{2\sigma}}\right\|^{p}_{L^{p}(\bbR^{n})}\leq c (1-\sigma)[u]^{p}_{W^{\gamma,p}(B_{5})}+\|u\|^{p}_{L^{p}(B_{5})}+c\mathrm{Tail}(u;B_{5})^{p}+c\|f\|_{L^{p'}(B_5)}^{p'},
        \end{equation*}
        where $c=c(n,s_0,p)$. 
        % Using Lemma \ref{lem.secquo} together with the above inequality, we get
        % \begin{equation*}
        % \begin{aligned}
        %     [w\psi]^p_{B^{2\sigma,p}_{\infty}(\bbR^{n})}\leq  c(1-\sigma)[u]^{p}_{W^{\gamma,p}(B_{4})}+c\|u\|^p_{L^{p}(B_4)}+c\mathrm{Tail}(u;B_{4})^{p}+c\|f\|_{L^{p'}(B_1)}^{p'}
        % \end{aligned}
        % \end{equation*}
        % for some constant $c=c(n,s_0,p)$.
        By taking $\eta=\xi\psi$, we obtain \eqref{vlast2}.
\end{proof}
Armed with the above two lemmas, we now prove our first main result.

\begin{proof}[\textbf{Proof of Theorem \ref{thm.1}.}] Let us first fix $s_0\in(0,s]$, $\alpha\in(s,\alpha_0)$, where $\alpha_0$ is determined in \eqref{defn.alpha0}. Choose $B_{10\rho}(y_0)\Subset \Omega$. By Lemma \ref{lem.scale}, $v(x)=u(\rho x+y_0)$ is a weak solution to 
\begin{equation}\label{eq.thm1}
    (-\Delta_p)^sv=f_\rho\quad\text{in }B_5,
\end{equation}
where $f_\rho(x)=\rho^{sp}f(\rho x+y_0)$. Then by Lemma \ref{loc.mainthmd} and Lemma \ref{loc.mainthmdbs} along with the fact that $s_0\leq s$, we have
\begin{align*}
         \sup_{0<|h|<\frac{1}{1000}}&\left\|\frac{\delta_h ^2(v\eta)}{|h|^{s_0\overline{p}}}\right\|^p_{L^{p}(\bbR^{n})}\leq\sup_{0<|h|<\frac{1}{1000}}\left\|\frac{\delta_h ^2(v\eta)}{|h|^{s\overline{p}}}\right\|^p_{L^{p}(\bbR^{n})}\\
         &\quad\quad\quad\leq c (1-s)[v]^{p}_{W^{s,p}(B_{5})}+c\|v\|_{L^p(B_5)}^p+c\mathrm{Tail}(v;B_{5})^{p}+c\|f_\rho\|_{L^{p'}(B_5)}^{p'}
    \end{align*}
    for some constant $c=c(n,s_0,p)$ and some function $\eta\in C_{c}^{\infty}(B_{3/4})$ satisfying $\eta\equiv 1$ on $B_{1/2}$ and $|\nabla \eta|\leq c$, where $\overline{p}=\min\{2,p'\}$. With aid of Lemma \ref{lem.secmeb} with standard covering arguments, we now prove the desired result when $p\geq2$.
    
\textbf{In case of $s_0\leq(p-1)/p$ and $p\geq2$: }Employing \eqref{ineq.seq1} in Lemma \ref{lem.secmeb},
\begin{align*}
[v\eta]_{W^{\alpha,p}(\bbR^n)}^p\leq c (1-s)[v]^{p}_{W^{s,p}(B_{5})}+c\|v\|^p_{L^{p}(B_5)}+c\mathrm{Tail}(v;B_{5})^{p}+c\|f_\rho\|_{L^{p'}(B_5)}^{p'}
\end{align*}
holds for some constant $c=c(n,s_0,p,\alpha)$. Changing of variables along with the fact that $\eta\equiv 1$ on $B_{1/2}$ and Lemma \ref{lem.upest}, we get
\begin{align*}
\rho^{p\alpha-n}[u]_{W^{\alpha,p}(B_{\rho/2}(y_0))}^p&\leq c\rho^{sp-n}(1-s)[u]_{W^{s,p}(B_{5\rho}(y_0))}^p+c\rho^{-n}\|u\|_{L^{p}(B_{5\rho}(y_0))}^p\\
&\quad+c\mathrm{Tail}(u;B_{5\rho}(y_0))^{p}+c\rho^{spp'-n}\|f\|^{p'}_{L^{p'}(B_{5\rho}(y_0))}\\
&\leq c\widetilde{E}(u;B_{10\rho}(y_0))^{p}+c\rho^{spp'-n}\|f\|^{p'}_{L^{p'}(B_{10\rho}(y_0))}
\end{align*}
for some constant $c=c(n,s_0,p,\alpha)$. Standard covering arguments now yield
\begin{align}\label{ineq.fspr}
r^{\alpha-\frac{n}{p}} [u]_{W^{\alpha,p}(B_{r}(x_0))}\leq c\widetilde{E}(u;B_{2r}(x_0))+cr^{sp'-\frac{n}{p}}\|f\|^{\frac{1}{p-1}}_{L^{p'}(B_{2r}(x_0))}
\end{align}
for some constant $c=c(n,s_0,p,\alpha)$ whenever $B_{2r}(x_0)\Subset \Omega$.
   
\textbf{In case of $s_0>(p-1)/p$ and $p\geq2$: }Using \eqref{ineq.seq2} in Lemma \ref{lem.secmeb}, we have 
\begin{align*}
[\nabla(v\eta)]_{W^{\alpha-1,p}(\bbR^n)}^p\leq c\left[ (1-s)[v]^{p}_{W^{s,p}(B_{5})}+\|v\|^p_{L^{p}(B_5)}+\mathrm{Tail}(v;B_{5})^{p}+\|f_\rho\|_{L^{p'}(B_5)}^{p'}\right]
\end{align*}
for some constant $c=c(n,s_0,p,\alpha)$. Similarly, as in the proof of \eqref{ineq.fspr}, changing variables, Lemma \ref{lem.upest} and standard covering arguments yield
\begin{align}\label{ineq.fspr2}
r^{\alpha-\frac{n}{p}}[\nabla u]_{W^{\alpha-1,p}(B_{r}(x_0))}\leq c\widetilde{E}(u;B_{2r}(x_0))+cr^{sp'-\frac{n}{p}}\|f\|^{\frac{1}{p-1}}_{L^{p'}(B_{2r}(x_0))}
\end{align} 
for some constant $c=c(n,s_0,p,\alpha)$ whenever $B_{2r}(x_0)\Subset \Omega$.

We now focus on the case $p\leq2$. Its proof is based on a bootstrap argument along with Lemma \ref{lem.scale} and Lemma \ref{loc.mainthmdbs}. 
    
\textbf{In case of $s_0\leq (p-1)/p$ and $p\leq2$: }We now take $\delta=(\alpha_0-\alpha)(p-1)/2$,
\begin{equation*}
\gamma_i=s_0(2-p)^i+(s_0p-\delta)\frac{1-(2-p)^i}{p-1}\quad\text{and}\quad \sigma_i=\gamma_i+(s_0-\gamma_i)p/2
\end{equation*}
to see that $\gamma_1<2s_0$ and $\gamma_{i+1}<2\sigma_i$ for each $i\geq1$.
In addition, since 
\begin{equation*}
\lim_{i\to\infty}\gamma_i=\frac{s_0p-\delta}{p-1},
\end{equation*}
there is a natural number $l=l(n,s_0,p,\alpha)$ such that \begin{equation}\label{cond.l.thm1}
\gamma_l\leq \alpha_0< \gamma_{l+1}.
\end{equation} 
As in the case $p\geq 2$, we first observe from \eqref{ineq.seq1} in Lemma \ref{lem.secmeb} that 
\begin{align*}
[v\eta]_{W^{\gamma_1,p}(\bbR^n)}^p\leq c (1-s)[v]^{p}_{W^{s,p}(B_{5})}+c\|v\|^p_{L^{p}(B_5)}+c\mathrm{Tail}(v;B_{5})^{p}+c\|f_\rho\|_{L^{p'}(B_5)}^{p'}
\end{align*}
for some $c=c(n,s_0,p,\alpha)$. Since $\gamma_1<2s_0$, as in \eqref{ineq.fspr}, we obtain that if $B_{2r}(x_0)\Subset \Omega$, then
\begin{equation}\label{ineq.sin.ind}
r^{\gamma_1-\frac{n}{p}}[u]_{W^{\gamma_1,p}(B_{r}(x_0))}^p\leq c\widetilde{E}(u;B_{2r}(x_0))^{p}+cr^{sp'-\frac{n}{p}}\|f\|^{p'}_{L^{p'}(B_{2r}(x_0))}
\end{equation}
holds for some $c=c(n,s_0,p,\gamma_1)$, where a change of variables, Lemma \ref{lem.upest} and standard covering arguments were used. Since $u\in W^{\gamma_1,p}_{\mathrm{loc}}(\Omega)$, $v(x)=u(\rho x+y_0)\in W^{\gamma_1,p}(B_5)$ is a weak solution to \eqref{eq.thm1}, so that Lemma \ref{loc.mainthmdbs} yields 
\begin{align*}
\sup_{0<|h|<\frac{1}{1000}}\left\|\frac{\delta_h ^2(v\eta)}{|h|^{2\sigma_1}}\right\|^p_{L^{p}(\bbR^{n})}&\leq  c (1-\sigma_1)[v]^{p}_{W^{\gamma_1,p}(B_{5})}+c\|v\|_{L^p(B_5)}^p\\
&\quad+c\mathrm{Tail}(v;B_{5})^{p}+c\|f_\rho\|_{L^{p'}(B_5)}^{p'}
\end{align*}
for some $c=c(n,s_0,p)$. Since $\gamma_2<2\sigma_1$, we now use \eqref{ineq.seq1} in Lemma \ref{lem.secmeb} to see that 
\begin{align*}
[v]^p_{W^{\gamma_2,p}(B_{1/2})}&\leq c[v]^{p}_{W^{\gamma_1,p}(B_{5})}+c\|v\|_{L^{p}(B_5)}^{p}+c\mathrm{Tail}(v;B_{5})^{p}+c\|f_\rho\|_{L^{p'}(B_5)}^{p'},
\end{align*}
where $c=c(n,s_0,p,\alpha)$, as $\sigma_1$ depends only on $n,s_0$ and $p$.
Changing variables and using \eqref{ineq.sin.ind},
\begin{align*}
\rho^{\gamma_2-\frac{n}{p}}[u]_{W^{\gamma_2,p}(B_{\rho/2}(y_0))}\leq c\widetilde{E}(u;B_{10\rho}(y_0))+c\rho^{sp'-\frac{n}{p}}\|f\|_{L^{p'}(B_{10\rho}(y_0))}^{\frac{1}{p-1}}
\end{align*}
holds for $B_{10\rho}(y_0)\subset B_{2r}(x_0)$. Via standard covering arguments, we have 
\begin{equation*}
r^{\gamma_2-\frac{n}{p}}[u]_{W^{\gamma_2,p}(B_{r}(x_0))}\leq c\widetilde{E}(u;B_{2r}(x_0))+cr^{sp'-\frac{n}{p}}\|f\|^{\frac{1}{p-1}}_{L^{p'}(B_{2r}(x_0))}
\end{equation*}
for some constant $c=c(n,s_0,p,\alpha)$ whenever $B_{2r}(x_0)\Subset\Omega$. We now iterate $l-2$ more times to show that 
\begin{equation}\label{ineq.ith}
r^{\gamma_l-\frac{n}{p}}[u]_{W^{\gamma_l,p}(B_{r}(x_0))}\leq c\widetilde{E}(u;B_{2r}(x_0))+cr^{sp'-\frac{n}{p}}\|f\|^{\frac{1}{p-1}}_{L^{p'}(B_{2r}(x_0))}.
\end{equation}
Since $u(\rho x+y_0)\in W^{\gamma_l,p}(B_5)$ is a weak solution to \eqref{eq.thm1}, Lemma \ref{loc.mainthmd}, \eqref{eq:embed}, and covering arguments imply
\begin{equation}\label{ineq.fspr3}
\begin{aligned}
r^{\alpha-\frac{n}{p}}[u]_{W^{\alpha,p}(B_{r}(x_0))}\leq c\widetilde{E}(u;B_{2r}(x_0))+cr^{sp'-\frac{n}{p}}\|f\|^{\frac{1}{p-1}}_{L^{p'}(B_{2r}(x_0))}
\end{aligned}
\end{equation}
for some constant $c=c(n,s_0,p,\alpha)$, where we have used \eqref{cond.l.thm1}. 

\textbf{In case of $s_0>(p-1)/p$ and $p\leq2$: }Since $s_0>(p-1)/p$, there is a positive integer $l$ such that 
\begin{equation*}
\gamma_l\leq 1< \gamma_{l+1}.
\end{equation*}
As in the proof of \eqref{ineq.fspr3}, we get 
\begin{align}\label{ineq.last.sin}
r^{{\gamma}-\frac{n}{p}}[u]_{W^{{\gamma},p}(B_{r}(x_0))}\leq c\widetilde{E}(u;B_{2r}(x_0))+cr^{sp'-\frac{n}{p}}\|f\|^{\frac{1}{p-1}}_{L^{p'}(B_{2r}(x_0))}
\end{align} 
for any $\gamma\in(0,1)$, where $c=c(n,s_0,p,\gamma)$ and $B_{2r}(x_0)\Subset\Omega$. We next select $\widetilde{\gamma}=\frac{1}{2-p}\left(\frac{\alpha_0+\alpha}{2}-s_0p\right)$ to see that 
\begin{equation*}
\widetilde{\gamma}<1\quad\text{and}\quad\frac{\alpha+\alpha_0}{2}=2\widetilde{\gamma}+(s_0-\widetilde{\gamma})p.
\end{equation*}  
Employing Lemma \ref{loc.mainthmdbs} with $\gamma$ replaced by $\widetilde{\gamma}$, we obtain that 
\begin{align*}
\sup_{0<|h|<\frac{1}{1000}}\left\|\frac{\delta_h ^2(v\eta)}{|h|^{\frac{\alpha+\alpha_0}{2}}}\right\|^p_{L^{p}(\bbR^{n})}&\leq  c \|v\|^{p}_{W^{\widetilde{\gamma},p}(B_{5})}+c\mathrm{Tail}(v;B_{5})^{p}+c\|f_\rho\|_{L^{p'}(B_5)}^{p'},
\end{align*}
where $c=c(n,s_0,p,\alpha)$.
Therefore, by \eqref{ineq.seq2} in Lemma \ref{lem.secmeb}, changing variables, \eqref{ineq.last.sin} and Lemma \ref{lem.upest}, we get
\begin{align*}
\rho^{\alpha}[\nabla u]_{W^{\alpha-1,p}(B_{\rho/2}(y_0))}\leq c\widetilde{E}(u;B_{10\rho}(y_0))+c\rho^{sp'-\frac{n}{p}}\|f\|^{\frac{1}{p-1}}_{L^{p'}(B_{10\rho}(y_0))}
\end{align*}
for $B_{10\rho}(y_0)\subset B_{2r}(x_0)$, where $c=c(n,s_0,p,\alpha)$.
Using standard covering argument, we deduce
\begin{align}\label{ineq.fspr4}
r^{\alpha-\frac{n}{p}}[\nabla u]_{W^{\alpha-1,p}(B_{r}(x_0))}\leq c\widetilde{E}(u;B_{2r}(x_0))+cr^{sp'-\frac{n}{p}}\|f\|^{\frac{1}{p-1}}_{L^{p'}(B_{2r}(x_0))}
\end{align}
for some constant $c=c(n,s_0,p,\alpha)$, whenever $B_{2r}(x_0)\Subset\Omega$.

Finally, since $u-(u)_{B_{2r}(x_0)}$ is a also weak solution to \eqref{pt : eq.main}, the desired estimates follow from \eqref{ineq.fspr}, \eqref{ineq.fspr2}, \eqref{ineq.fspr3} and \eqref{ineq.fspr4}. This completes the proof.
\end{proof}
%%%%%%%%%%%%%%%%%

We end this section with the following $V$-function estimates which might be useful in future applications.
    \begin{corollary}
    Let $u\in W^{s,p}(B_{2R}(x_0))\cap L^{p-1}_{sp}(\bbR^n)$ be a weak solution to 
    \begin{align}\label{rmk.eq}
        (-\Delta_p)^su=f\quad\text{in }B_{2R}(x_0).
    \end{align}
    Then for any $|h|<R/1000$, we have 
    \begin{equation}\label{vftnest1}
    \begin{aligned}
        &(1-s)R^{sp}\dashint_{B_{{R}/{2}}(x_0)}\int_{B_{{R}/{2}}(x_0)}|\delta_h V(\delta_{x,y}^su)|^2\frac{\,dx\,dy}{|x-y|^n}\\
        &\leq c|h|^{s\overline{p}}E(u;B_R(x_0))^p+c|h|^{s\overline{p}}R^{spp'-n}\|f\|_{L^{p'}(B_{R}(x_0))}^{p'}
    \end{aligned}
    \end{equation}
    for some constant $c=c(n,s_0,p)$, where $\overline{p}=\min\{2,p'\}$ and the constant $s_0$ is determined in \eqref{s0assum}. In particular, when $p\leq 2$, we have a more refined estimate
     \begin{equation}\label{vftnest2}
    \begin{aligned}
        &(1-s)R^{sp}\dashint_{B_{{R}/{2}}(x_0)}\int_{B_{{R}/{2}}(x_0)}|\delta_h V(\delta_{x,y}^su)|^2\frac{\,dx\,dy}{|x-y|^n}\\
        &\leq c|h|^{2\gamma+(s-\gamma)p}E(u;B_R(x_0))^p+c|h|^{2\sigma}R^{spp'-n}\|f\|_{L^{p'}(B_{R}(x_0))}^{p'}
    \end{aligned}
    \end{equation}
    for any $\gamma\in \left(s,\min\{sp',1\}\right)$, where $c=c(n,s_0,p,\gamma)$.
    \end{corollary}
    \begin{proof}
        In view of a straightforward scaling argument, we may assume that $x_0=0$ and $R=1$. Let us fix $B_{10\rho}(y_0)\Subset B_2$. By Lemma \ref{lem.scale}, 
    $
        u_\rho(x)=u(\rho x+y_0)
    $
    is a weak solution to 
\begin{equation*}
    (-\Delta_p)^su_\rho=f_\rho\quad\text{in }B_{10},
\end{equation*}
where $f_\rho(x)=\rho^{sp}f(\rho x+y_0)$. By Lemma
    \ref{el : lem.loc}, $v=u_\rho\xi$ is a weak solution to 
    \begin{equation*}
    (-\Delta_p)^sv=f_\rho+g\quad\text{in }B_{2},
\end{equation*}
where the function $\xi$ is determined in Lemma \ref{el : lem.loc}, and $g$ satisfies
\begin{equation}\label{rmk.ineqg}
    \|g\|_{L^{p'}(B_{5/2})}^{p'}\leq c(1-s)\|u_\rho\|_{L^{p}(B_3)}^p+c\mathrm{Tail}(u;B_3)^p
\end{equation}
for some constant $c=c(n,s_0,p)$. In light of Lemma \ref{lem.vrefd} and Lemma \ref{lem.vrefs}, we deduce
\begin{align*}
        &(1-s)\int_{\bbR^{n}}\int_{\bbR^n}|\delta_h V(\delta_{x,y}^{s}v)|^{2}\psi(x)\psi(y)\frac{\,dx\,dy}{|x-y|^{n}}\\
        &\leq c|h|^{s\overline{p}}\left((1-s)[v]^{p}_{{W}^{s,p}(\bbR^{n})}+\|v\|^p_{L^p(\bbR^n)}+\|f_{\rho}+g\|^{p'}_{L^{p'}(B_1)}\right)
    \end{align*}
    for some constant $c=c(n,s_0,p)$ and the function $\psi$ is determined in Lemma \ref{lem.vrefd}.
    Using \eqref{norm.loc} and \eqref{rmk.ineqg}, we obtain
    \begin{equation}\label{ineq.vftnsd}
    \begin{aligned}
        &(1-s)\rho^{sp}\dashint_{B_{\rho/2}(y_0)}\int_{B_{\rho/2}(y_0)}|\delta_h V(\delta_{x,y}^{s}u)|^{2}\frac{\,dx\,dy}{|x-y|^{n}}\\
        &\leq c|h|^{s\overline{p}}\left((1-s)\rho^{sp-n}[u]^{p}_{{W}^{s,p}(B_{5\rho}(y_0))}+\rho^{-n}\|u\|^p_{L^p(B_{5\rho}(y_0))}\right)\\
        &\quad+c|h|^{s\overline{p}}\left(\mathrm{Tail}(u;B_{5\rho}(y_0))^p+\rho^{spp'-n}\|f\|^{p'}_{L^{p'}(B_{5\rho}(y_0))}\right)
    \end{aligned}
    \end{equation}
    for some constant $c=c(n,s_0,p)$. Since $u-(u)_{B_{5\rho}(y_0)}$ is a also weak solution to \eqref{rmk.eq} with $x_0=0$ and $R=1$, together with Lemma \ref{lem.poist}, we observe that
    \begin{equation}\label{ineq.vftnsd2}
    \begin{aligned}
        &(1-s)\rho^{sp}\dashint_{B_{\rho/2}(y_0)}\int_{B_{\rho/2}(y_0)}|\delta_h V(\delta_{x,y}^{s}u)|^{2}\frac{\,dx\,dy}{|x-y|^{n}}\\
        &\leq c|h|^{s\overline{p}}\left((1-s)\rho^{sp-n}[u]^{p}_{{W}^{s,p}(B_{5\rho}(y_0))}+\mathrm{Tail}(u-(u)_{B_{5\rho}(y_0)};B_{5\rho}(y_0))^p\right)\\
        &\quad+c|h|^{s\overline{p}}\rho^{spp'-n}\|f\|^{p'}_{L^{p'}(B_{5\rho}(y_0))}
    \end{aligned}
    \end{equation}
    holds, where $c=c(n,s_0,p)$. Then \eqref{vftnest1} follows from Lemma \ref{lem.upest} and a standard covering argument. 

    We are now in the position to prove \eqref{vftnest2}. Let us fix 
    \begin{equation*}
      \gamma\in \left(s,\min\{sp',2+(s-1)p\}\right)
    \end{equation*}
    to see that there is a constant $t_0\in(0,s]$ such that $\alpha_0=\gamma$, where $\alpha_0$ is determined in \eqref{defn.alpha0} with $s_0$ replaced by $t_0$. By Theorem \ref{thm.1}, we observe that $u\in W^{\gamma,p}_{\mathrm{loc}}(\Omega)$ and
    \begin{align}\label{vest.}
        r^{\gamma p-n}[u]^p_{W^{\gamma,p}(B_{r}(z_0))}\leq cE(u;B_{2r}(z_0))^p+cr^{spp'-n}\|f\|^{p'}_{L^{p'}(B_{2r}(z_0))}
    \end{align}
    for some constant $c=c(n,s_0,p,\gamma)$ whenever $B_{2r}(z_0)\Subset\Omega$, as $t_0$ depends only on $n,p$ and $\gamma$. By Lemma \ref{lem.vrefs}, 
   \begin{align*}
        &(1-s)\int_{\bbR^{n}}\int_{\bbR^n}|\delta_h V(\delta_{x,y}^{s}v)|^{2}\psi(x)\psi(y)\frac{\,dx\,dy}{|x-y|^{n}}\\
        &\leq c|h|^{2\gamma+(s-\gamma)p}\left([v]_{W^{\gamma,p}(\bbR^n)}^p+(1-s)[v]^{p}_{{W}^{s,p}(\bbR^{n})}+\|v\|^p_{L^p(\bbR^n)}+\|f\|^{p'}_{L^{p'}(B_1)}\right)
    \end{align*}
    holds for some constant $c=c(n,s_0,p,\gamma)$, where the constant $\sigma$ is determined in Lemma \ref{lem.vrefs}.
    As in \eqref{ineq.vftnsd} and \eqref{ineq.vftnsd2} along with Lemma \ref{lem.upest} and \eqref{vest.}, we deduce 
    \begin{equation*}
    \begin{aligned}
        &(1-s)\rho^{sp}\dashint_{B_{\rho/2}(y_0)}\int_{B_{\rho/2}(y_0)}|\delta_h V(\delta_{x,y}^{s}u)|^{2}\frac{\,dx\,dy}{|x-y|^{n}}\\
        &\leq c|h|^{2\gamma+(s-\gamma)p}\left(E(u;B_{10\rho}(y_0))^p+\rho^{spp'-n}\|f\|^{p'}_{L^{p'}(B_{5\rho}(y_0))}\right)
    \end{aligned}
    \end{equation*}
    for some $c=c(n,s_0,p,\gamma)$. By standard covering arguments, we obtain \eqref{vftnest2}.
\end{proof}

%%%%%%%%%%%%%%%%%%%%%%%%%%%%%

%%%%%%%%%%%%%%%%%%%%%%%%

\section{Differentiable data} \label{sec:ddata}
In this section, we consider the case of a differentiable right-hand side $f\in W^{t,p'}$ for $t\in(0,1)$. We use a different kind of an iteration scheme to utilize the additional regularity imposed on $f$.
\begin{lemma}\label{lem.diffquo}
    Let $u\in W^{s,p}(\bbR^{n})$ be a weak solution to 
    \begin{equation}\label{eq.itesubd}
        (-\Delta_p)^su=f\quad\text{in }B_{2}
    \end{equation}
    with $f\in L^{p'}(B_{5/2})$. For any $|h|<1/1000$, we have 
    \begin{align}\label{eq:cacc}
    \begin{split}
    & (1-s)\int_{\bbR^n}\int_{\bbR^n}\left|\delta_h V\left(\delta_{x,y}^su\right)\right|^{2}\max\{\psi^{\overline{p}}(x),\psi^{\overline{p}}(y)\}\frac{\,dx\,dy}{|x-y|^{n}}\\
    &\leq c(1-s)\int_{\bbR^n}\int_{\bbR^n}\phi_{|\delta_{x,y}^su|}\left({\left(|\delta_{h}u(x)|+|\delta_hu(y)|\right)} \abs{\delta^s_{x,y}\psi}\right)\frac{\,dx\,dy}{|x-y|^{n}}\\
    &\quad+c\int_{B_1}|\delta_ h u||\delta_h f|\,dx
    \end{split}
\end{align}
    for some constant $c=c(n,p)$, where the function $\psi$ is determined in Lemma \ref{lem.combvest} and ${\overline{p}}=\max\{2,p\}$.
\end{lemma}

\begin{proof}
    Fix a cut-off function $\psi=\psi(x)\in C_{c}^{\infty}(B_{3/4})$ with \eqref{test.sing} and fix $0<|h|<\frac{1}{1000}$.
    By testing \eqref{eq.itesubd} with $\delta_{-h}(\psi^{\overline{p}}\delta_{h}u)$, we get
    \begin{align*}
    (1-s)\int_{\bbR^n}\int_{\bbR^n}\delta_h A\left(\delta^s_{x,y}u\right)\delta^s_{x,y}(\psi^{\overline{p}}\delta_{h}u)\dfrac{\,dx\,dy}{\abs{x-y}^n}=\int_{B_1}\delta_h f\,\psi^{\overline{p}}\delta_h u\,dx.
    \end{align*}
    Let us estimate now estimate the quantity
    \begin{equation*}
        J_1\coloneqq\delta_h A(\delta^s_{x,y}u)\delta^s_{x,y}(\psi^{\overline{p}}\delta_{h}u).
    \end{equation*}
We may assume $\delta_{h}u(x)\geq\delta_{h}u(y)$ by interchanging the roles of $\delta_{h}u(x)$ and $\delta_{h}u(y)$.

\textbf{In case of $\psi(x)\geq \psi(y)$:} In this case, by \eqref{equi.shif} we have 
	\begin{equation}\label{ineq.estij1}
	\begin{aligned}
	J_1&\geq \delta_h A(\delta^s_{x,y}u)\psi^{\overline{p}}(x)\delta^s_{x,y}(\delta_{h}u)\\
	&=\delta_h A(\delta^s_{x,y}u)\psi^{\overline{p}}(x)\delta_{h}(\delta^s_{x,y}u)\geq \frac{1}{c}\phi_{|\delta_{x,y}^su|}\left({|\delta_h \delta_{x,y}^su|}\right)\psi^{\overline{p}}(x).
	\end{aligned}
	\end{equation}

\textbf{In case of $\psi(x)\leq \psi(y)$:} We rewrite $J_1$ as 
	\begin{align*}
	J_1=\delta_h A(\delta^s_{x,y}u)\psi^{\overline{p}}(y)\delta^s_{x,y}(\delta_{h}u)+\delta_h A(\delta^s_{x,y}u)\delta_{h}u(x)\delta^s_{x,y}(\psi^{\overline{p}})\eqqcolon J_{1,1}+J_{1,2}.
	\end{align*}
	As in the estimate of \eqref{ineq.estij1}, we have $J_{1,1}\geq \phi_{|\delta_{x,y}^su|}\left({|\delta_h \delta_{x,y}^su|}\right)\psi^{\overline{p}}(y)$. On the other hand, using \eqref{equi.shif} and the fact that $|\psi(x)|\leq|\psi(y)|$, one can see that
	\begin{align*}
	J_{1,2}&\geq -c\phi'_{|\delta_{x,y}^su|}\left(|\delta_h \delta_{x,y}^su|\right)|\delta_{h}u(x)||\psi(y)|^{\overline{p}-1}\abs{\delta^s_{x,y}\psi}
	\end{align*}
	for some constant $c=c(n,p)$.
	Using Young's inequality given in \cite[Lemma 31, Lemma 32, and Lemma 34]{DieEtt08}, we obtain
	\begin{align*}
	J_{1,2}&\geq -\epsilon(\phi_{|\delta_{x,y}^su|})^{*}\left(\psi^{\overline{p}-1}(y){\phi'_{|\delta_{x,y}^su|}\left(|\delta_h \delta_{x,y}^su|\right)}\right)-c_{\epsilon}\phi_{|\delta_{x,y}^su|}\left({|\delta_{h}u(x)|}\abs{\delta^s_{x,y}\psi}\right)\\
	&\geq -\epsilon\psi^{\overline{p}}(y)\phi_{|\delta_{x,y}^su|}\left({|\delta_h \delta_{x,y}^su|}\right)-c_{\epsilon}\phi_{|\delta_{x,y}^su|}\left({|\delta_{h}u(x)|}\abs{\delta^s_{x,y}\psi}\right)
	\end{align*}
	for some constant $c_\epsilon=c(n,p,\epsilon)$. Combining all the estimates along with the choice of $\epsilon=\frac{1}{4c}$, we arrive at the estimate
	\begin{align*}
	J_1&\geq \frac{1}{c}\max\{\psi^{\overline{p}}(x),\psi^{\overline{p}}(y)\}\phi_{|\delta_{x,y}^su|}\left({|\delta_h \delta^s_{x,y}u|}\right)\\
	&\quad-c\phi_{|\delta_{x,y}^su|}\left({\left(|\delta_{h}u(x)|+|\delta_hu(y)|\right)}\delta^s_{x,y}\psi\right)
	\end{align*}
	for some constant $c=c(n,p)$.
Thus, by the above estimate for $J_1$ along with \eqref{equi.shif} and Lemma \ref{el : lem.eleineq}, we conclude that \eqref{eq:cacc} holds with respect to some constant $c=c(n,p)$.
\end{proof}

\subsection{The degenerate case}
To use a bootstrap argument, we further estimate the right-hand side of the inequality given in Lemma \ref{lem.diffquo}. In order to accomplish this, we use a different approach in each of the cases $p\in(1,2]$ and $p \in [2,\infty)$.
\begin{lemma}\label{lem.degind}
    Let $p\geq 2$. Let $u\in W^{s,p}(\bbR^{n})$ be a weak solution to 
    \begin{equation*}
        (-\Delta_p)^su=f\quad\text{in }B_{2}
    \end{equation*}
    with $f\in L^{p'}(B_2)$. For any $|h|<1/1000$, we have
    \begin{align*}
        (1-s)[\delta_h(u\psi)]_{W^{s,p}(\bbR^n)}^p&\leq c\left((1-s)^{\frac1p}[u]_{W^{s,p}(\bbR^{n})}\right)^{p-2}\left\|{\delta_h u}\right\|_{L^{p}(\bbR^{n})}^{2}+c\left\|{\delta_h u}\right\|_{L^p(\bbR^n)}^p\nonumber\\
        &\quad+c|h|^p\left((1-s)[u]^p_{{W}^{s,p}(\bbR^n)}+\|u\|_{L^p(\bbR^n)}^p\right)\\
        &\quad+c\left\|{\delta_h u}\right\|_{L^{p}(B_1)}\left\|{\delta_h f}\right\|_{L^{p'}(B_{1})}
    \end{align*}
    for some constant $c=c(n,s_0,p)$, where the function $\psi$ is determined in Lemma \ref{lem.combvest} and the constant $s_0$ is determined in \eqref{s0assum}.
\end{lemma}
\begin{proof}
Fix a cut-off function $\psi\in C_{c}^\infty(B_{3/4})$ with $\psi\equiv 1$ in $B_{1/2}$ and \eqref{test.sing}. Using \eqref{ineq.simple0}, \eqref{test.sing}, \eqref{ineq.test.grhe} and Lemma \ref{el : lem.eleineq}, we observe that
\begin{align*}
   &\int_{\bbR^n}\int_{\bbR^n}\abs{\delta^s_{x,y}\delta_h(u\psi)}^p\dfrac{\,dx\,dy}{\abs{x-y}^n}\\
   &\leq\int_{\bbR^n}\int_{\bbR^n}\left|\delta_h V\left(\delta^s_{x,y}u\right)\right|^{2}\max\{\psi^{p}(x),\psi^p(y)\}\frac{\,dx\,dy}{|x-y|^{n}}\\
   &\quad+ \frac{c}{1-s}\|\delta_hu\|_{L^p(\bbR^n)}^p+c|h|^p\left([u]_{W^{s,p}(\bbR^n)}^p+\frac{c}{1-s}\|u\|_{L^p(\bbR^n)}^p\right)
\end{align*}
for some $c=c(n,s_0,p)$.
We next note that from \eqref{ineq1.shif} and H\"older's inequality,
\begin{equation}\label{ineq.vpn}
\begin{aligned}
    &\int_{\bbR^n}\int_{\bbR^n}\phi_{|\delta_{x,y}^su|}\left(\left(|\delta_{h}u(x)|+|\delta_hu(y)|\right) \abs{\delta^s_{x,y}\psi}\right)\frac{\,dx\,dy}{|x-y|^{n}}\\
    &\leq c\int_{\bbR^n}\int_{\bbR^n}{|\delta_{x,y}^su|^{p-2}}{(|\delta_{h}u(x)|+|\delta_{h}u(x)|)^2}\abs{\delta^s_{x,y}\psi}^2\frac{\,dx\,dy}{|x-y|^{n}}\\
    &\quad+c\int_{\bbR^n}\int_{\bbR^n}{\left(|\delta_{h}u(x)|^p+|\delta_{h}u(y)|^p\right)}\abs{\delta^s_{x,y}\psi}^p\frac{\,dx\,dy}{|x-y|^{n}}\\
     &\leq \frac{c}{(1-s)^{\frac{2}{p}}}[u]_{W^{s,p}(\bbR^n)}^{p-2}\|\delta_h u\|_{L^p(\bbR^n)}^2+\frac{c}{1-s}\|\delta_h u\|_{L^p(\bbR^n)}^p
\end{aligned}
\end{equation}
holds for some constant $c=c(n,s_0,p)$.
Thus, combining the above two inequalities along with Lemma \ref{lem.diffquo} yields the desired estimate.
\end{proof}

\begin{lemma}\label{lem.deg.reg.ind} Let $u\in W^{\gamma,p}(B_{5R}(x_0))\cap L^{p-1}_{sp}(\bbR^n)$ be a weak solution to 
\begin{equation*}
    (-\Delta_p)^su=f\quad\text{in }B_{5R}(x_0),
\end{equation*}
    where $f\in W^{t,p'}(B_{5R}(x_0))$ and $\gamma\in[s,\min\{1,\alpha_1\})$ with the constant $\alpha_1$ determined in \eqref{defn.alpha1}. If $s_0+\frac{\gamma+\min\{\gamma,t\}}{p}\leq 1$, then we have 
    \begin{align*}
        R^{\sigma-\frac{n}{p}}[u]_{W^{\sigma,p}(B_{R/2}(x_0))}&\leq c(1-s)^{\frac{1}{p}}R^{\gamma-\frac{n}{p}}[u]_{W^{\gamma,p}(B_{5R}(x_0))}\\
        &\quad+c\mathrm{Tail}(u-(u)_{B_{5R}(x_0)};B_{5R}(x_0))\\
        &\quad+cR^{\frac{sp+t}{p-1}-\frac{n}{p}}[f]_{W^{t,p'}(B_{5R}(x_0))}^{\frac{1}{p-1}}
    \end{align*}
    for some constant $c=c(n,s_0,p,\gamma,\sigma,t)$, where $\sigma<s_0+\frac{\gamma+\min\{\gamma,t\}}{p}$. On the other hand, if $s_0+\frac{\gamma+\min\{\gamma,t\}}{p}>1 $, then we have 
    \begin{align*}
        R^{\sigma-\frac{n}{p}}[\nabla u]_{W^{\sigma-1,p}(B_{R/2}(x_0))}^p&\leq c(1-s)^{\frac{1}{p}}R^{\gamma-\frac{n}{p}}[u]_{W^{\gamma,p}(B_{5R}(x_0))}\\
        &\quad+c\mathrm{Tail}(u-(u)_{B_{5R}(x_0)};B_{5R}(x_0))\\
        &\quad+cR^{\frac{sp+t}{p-1}-\frac{n}{p}}[f]_{W^{t,p'}(B_{5R}(x_0))}^{\frac{1}{p-1}}
    \end{align*}
    for some constant $c=c(n,s_0,p,\gamma,\sigma,t)$, where $1<\sigma<s_0+\frac{\gamma+\min\{\gamma,t\}}{p}$.
\end{lemma}
\begin{proof}
By Lemma \ref{lem.scale}, we may assume $x_0=0$ and $R=1$.
    Let us fix $\sigma<s_0+\frac{\gamma+\min\{\gamma,t\}}{p}$ and choose 
    \begin{equation*}
        \sigma_0=\frac{1}{2}\left(\sigma+s_0+\frac{\gamma+\min\{\gamma,t\}}{p}\right).
    \end{equation*}
    Then there is a constant $t_0<\min\{t,\gamma\}$ such that 
    \begin{equation*}
        \sigma_0=s_0+\frac{\gamma+t_0}{p}.
    \end{equation*}
    We choose a cutoff function $\xi\in C_c^{\infty}(B_{7/2})$ such that $|\nabla \xi|\leq4$. By Corollary \ref{el : lem.loc.deg}, $w=u\xi$ is a weak solution to
    \begin{equation*}
        (-\Delta_p)^sw=f+g\quad\text{in }B_2,
    \end{equation*}
    where $g\in W^{t_0,p'}(B_{5/2})$ with the estimate
    \begin{equation}\label{ineq1.g}
    \begin{aligned}
        [g]^{p'}_{W^{t_0,p'}(B_{5/2})}&\leq c(1-s)\|u\|_{W^{\gamma,p}(B_{3})}^p+c\mathrm{Tail}(u;B_{3})^p
    \end{aligned}
    \end{equation}
    for some constant $c=c(n,s_0,p,\gamma,t)$.

    Next note from Lemma \ref{lem.degind} that 
    \begin{align}\label{ineq0.itscd}
        (1-s)[\delta(w\psi)]^p_{W^{s,p}(\bbR^n)}&\leq c\left((1-s)^{\frac1p}[w]_{W^{s,p}(\bbR^n)}\right)^{p-2}\|\delta_h w\|_{L^p(\bbR^n)}^{2}+c\|\delta_h w\|^p_{L^p(\bbR^n)}\nonumber\\
        &\quad+c|h|^p\left((1-s)[w]_{W^{s,p}(\bbR^n)}+\|w\|^p_{L^p(\bbR^n)}\right) \nonumber\\
        &\quad+c\|\delta_h w\|_{L^p(B_1)}\|\delta_h F\|_{L^{p'}(B_1)},
    \end{align}
    where $F=f+g$. Then by (b) in Lemma \ref{lem.embdiff} and Lemma \ref{lem.poist}, we obtain
\begin{equation}\label{ineq1.F}
    \begin{aligned}
        \sup_{0<|h|<\frac{1}{1000}}\left\|\frac{\delta_h F}{|h|^{t_0}}\right\|_{L^{p'}(B_1)}&\leq  \sup_{0<|h|<\frac{1}{1000}}\left\|\frac{\delta_h (f-(f)_{B_2})}{|h|^{t}}\right\|_{L^{p'}(B_1)}\\
        &\quad+\sup_{0<|h|<\frac{1}{1000}}\left\|\frac{\delta_h (g-(g)_{B_2})}{|h|^{t_0}}\right\|_{L^{p'}(B_1)}\\
        &\leq c(1-t)^{\frac{1}{p'}}[f]_{W^{t,p'}(B_2)}+c(1-t_0)^{\frac{1}{p'}}[g]_{W^{t_0,p'}(B_2)}
    \end{aligned}
    \end{equation}
    for some constant $c=c(n,s_0,p,\gamma,t,\sigma)$, as $t_0$ depends only on $s_0,p,\gamma,t$ and $\sigma$. We now divide both sides of \eqref{ineq0.itscd} by $|h|^{\gamma+t_0}$ to see that 
    \begin{align*}
       (1-s)\left[\frac{\delta_h(w\psi)}{|h|^{\frac{\gamma+t_0}{p}}}\right]^p_{W^{s,p}(\bbR^n)}&\leq c\left((1-s)^{\frac1p}[w]_{W^{s,p}(\bbR^n)}\right)^{p-2}\left\|\frac{\delta_h w}{|h|^{\gamma}}\right\|_{L^p(\bbR^n)}^{2}\\
       &\quad+c\left\|\frac{\delta_h w}{|h|^{\gamma}}\right\|^p_{L^p(\bbR^n)}+c\left((1-s)[w]_{W^{s,p}(\bbR^n)}+\|w\|^p_{L^p(\bbR^n)}\right)\\
       &\quad+c\left\|\frac{\delta_h w}{|h|^{\gamma}}\right\|_{L^p(B_1)}\left\|\frac{\delta_h F}{|h|^{t_0}}\right\|_{L^{p'}(B_1)},
    \end{align*}
    where the fact that $|h|^{\gamma}\leq c|h|^{t_0}$ and $|h|\leq c|h|^{\frac{\gamma+t_0}{p}}$ for some constant $c$ is used. 
    
    By \eqref{lem.embdiff.a1} and \eqref{lem.embdiff.a2} in Lemma \ref{lem.embdiff}, Young's inequality, \eqref{ineq1.F} and the above inequality, we deduce 
    \begin{align*}
        &\sup_{0<|h|<\frac{1}{1000}}\left\|\frac{\delta_h^2 (w\psi)}{|h|^{s+\frac{\gamma+t_0}{p}}}\right\|_{L^p(\bbR^n)}^p\\
        &\leq c(1-s)\sup_{0<|h|<\frac{1}{1000}}\left[\frac{\delta_h(w\psi)}{|h|^{\frac{\gamma+t_0}{p}}}\right]^p_{W^{s,p}(\bbR^n)}+c\sup_{0<|h|<\frac{1}{1000}}\left\|\frac{\delta_h(w\psi)}{|h|^{\gamma}}\right\|_{L^p(\bbR^n)}^p\\
        &\leq c(1-s)[w]^{p}_{W^{s,p}(\bbR^n)}+(1-\gamma)[w]^{p}_{W^{\gamma,p}(\bbR^n)}+c\|w\|_{L^p(\bbR^n)}^p\\
        &\quad+ c(1-t)[f]_{W^{t,p'}(B_2)}^{p'}+c(1-t_0)[g]_{W^{t_0,p'}(B_2)}^{p'}
    \end{align*}
    for some constant $c=c(n,s_0,p)$. Using \eqref{norm.loc}, \eqref{ineq1.g} and \eqref{ineq0.itscd}, we obtain 
    \begin{align*}
        \sup_{0<|h|<\frac{1}{1000}}\left\|\frac{\delta_h^2 (w\psi)}{|h|^{s+\frac{\gamma+t_0}{p}}}\right\|_{L^p(\bbR^n)}^p&\leq (1-s)^{\frac1p}[u]^{p}_{W^{\gamma,p}(B_{5})}+\|u\|_{L^p(B_{5})}^p+\mathrm{Tail}(u;B_5)^{p}\\
        &\quad+c[f]_{W^{t,p'}(B_5)}^{p'},
    \end{align*}
    where $c=c(n,s_0,p,\gamma,\sigma,t)$. Suppose $s_0+\frac{\gamma+\min\{\gamma,t\}}{p}\leq 1$. Then \eqref{ineq.seq1} in Lemma \ref{lem.secmeb} yields 
    \begin{align*}
        [u]_{W^{\sigma,p}(B_{1/2})}^p&\leq [w\psi]^p_{W^{\sigma,p}(\bbR^n)}\\
        &\leq (1-s)^{\frac1p}[u]^p_{W^{\gamma,p}(B_{5})}+\|u\|_{L^p(B_{5})}^p+\mathrm{Tail}(u;B_5)^{p}+c[f]_{W^{t,p'}(B_5)}^{p'}
    \end{align*}
    for some constant $c=c(n,s_0,p,\gamma,\sigma,t)$. By the fact that $u-(u)_{B_5}$ is a also weak solution to our equation, together with Lemma \ref{lem.poist}, we get 
    \begin{align*}
        [u]_{W^{\sigma,p}(B_{1/2})}^p&\leq (1-s)^{\frac1p}[u]^p_{W^{\gamma,p}(B_{5})}+\mathrm{Tail}(u-(u)_{B_5};B_5)^{p}+c[f]_{W^{t,p'}(B_5)}^{p'},
    \end{align*}
    where $c=c(n,s_0,p,\gamma,\sigma,t)$. We now suppose $s_0+\frac{\gamma+\min\{\gamma,t\}}{p}> 1$. Following the same lines as in the above procedure together with \eqref{ineq.seq2} in Lemma \ref{lem.secmeb} yields the desired estimate.
\end{proof}

We now use a bootstrap argument to obtain our next main result. 
\begin{proof}[Proof of Theorem \ref{thm.regd}]
For $p\geq 2$, let us fix $s_0\in(0,s]$, and $t\in\left(0,\min\left\{\frac{s_0p}{p-2},1\right\}\right)$. Choose $\alpha\in(0,\alpha_1)$ for $\alpha_1\leq 1$ and $\alpha\in(1,\alpha_1)$ for $\alpha_1>1$, where the constant $\alpha_1$ is defined in \eqref{defn.alpha1}. The proof goes on two steps.

\textbf{Step 1: }We prove
\begin{equation}\label{ineq.fir.thm7}
u\in W^{t,p}_{\mathrm{loc}}(\Omega)
\end{equation}
with the estimate 
\begin{equation}\label{est.fir.thm7}
    R^{t-\frac{n}p}[u]_{W^{t,p}(B_r(x_0))}\leq cE(u;B_{2r}(x_0))+cR^{\frac{sp+t}{p-1}-\frac{n}p}\|f\|^{\frac{1}{p-1}}_{\widetilde{W}^{t,p'}(B_{2r}(x_0))}
\end{equation}
for some constant $c=c(n,s_0,p,t)$, where the notation $\|\cdot\|_{\widetilde{W}^{t,p'}}$ is given by \eqref{not.norm} whenever $B_{2r}(x_0)\Subset \Omega$.
To do this, define the number
\begin{equation}\label{cond.epis}
    \epsilon=\epsilon(s_0,p,t):=\min\left\{\frac{s_0}{2},\frac{s_0p-t(p-2)}{2p}\right\}>0
\end{equation}
as well as the sequence
\begin{equation}\label{seq.gamma.d}
   \gamma_{i}=(s_0-\epsilon)\sum_{k=0}^{i}\left(\frac2p\right)^k
\end{equation}
for any $i\geq0$. We summarize some facts for $\gamma_i$.
\begin{itemize}
\item By $\epsilon\leq \frac{s_0}{2}$ from \eqref{cond.epis}, we observe that $\gamma_i>0$.
\item Also, for any $i\geq0$,
\begin{equation}\label{seq.gamma.do}
\gamma_{i+1}=\tfrac{2\gamma_i}{p}+s_0-\epsilon.
\end{equation}
\item Since 
\begin{align*}
\eqref{cond.epis}\,\implies\,\epsilon\leq \tfrac{s_0p-t(p-2)}{2p}\,\implies\, t<\lim_{i\to\infty}\gamma_i=\begin{cases}
\frac{(s_0-\epsilon)p}{p-2}&\quad\text{if }p>2,\\
\infty&\quad\text{if }p=2,
\end{cases}
\end{align*}
there exists ${i}_0$ such that $i_0=\min\{i\in\setN \cup \{0\}:t<\gamma_{i}\}$.
\end{itemize}
Then the proof of \eqref{ineq.fir.thm7} goes as follows.
\begin{enumerate}
	\item[(1)] If $i_0=0$, then $t<\gamma_0<s_0$. Hence by \eqref{ineq2.upest} together with \eqref{eq:embed}, we have \eqref{ineq.fir.thm7} with \eqref{est.fir.thm7}.
	\item[(2)] If $i_0=1$, we have $\gamma_0\leq t<\gamma_1$. Then by Lemma \ref{lem.deg.reg.ind} with the choice of $\gamma=\gamma_0$ and $\sigma=t<\min\{\gamma_1,1\}=\min\{\frac{2\gamma_0}{p}+s_0-\epsilon,1\}$ from \eqref{seq.gamma.do} and by the standard covering argument, \eqref{ineq.fir.thm7}  with \eqref{est.fir.thm7} holds.
	\item[(3)] If $i_0\geq2$, then $\gamma_0<\cdots<\gamma_{i_0-1}\leq t<\gamma_{i_0}$ holds. Thus Lemma \ref{lem.deg.reg.ind} along with standard covering arguments gives $u\in W^{\gamma_1,p}_{\mathrm{loc}}(\Omega)$ with the estimate 
 \begin{equation}\label{ineq.deg.ga1}
    R^{\gamma_1-\frac{n}p}[u]_{W^{\gamma_1,p}(B_r(x_0))}\leq cE(u;B_{2r}(x_0))+cR^{\frac{sp+t}{p-1}-\frac{n}p}\|f\|^{\frac{1}{p-1}}_{\widetilde{W}^{t,p'}(B_{2r}(x_0))}
\end{equation}
whenever $B_{2r}(x_0)\Subset \Omega$, where $c=c(n,s_0,p,t)$.
  By iterating this procedure $i_0-2$ times along with \eqref{seq.gamma.do}, we obtain $u\in W^{\gamma_{i_0-1},p}_{\mathrm{loc}}(\Omega)$ and \eqref{ineq.deg.ga1} with $\gamma_1$ replaced by $\gamma_{i_0}$. Again Lemma \ref{lem.deg.reg.ind} for $t<\gamma_{i_0}$ yields
	\eqref{ineq.fir.thm7} with \eqref{est.fir.thm7}.
\end{enumerate}
Thus by \eqref{ineq.fir.thm7}, we can use Lemma \ref{lem.deg.reg.ind} with $\gamma$ replaced by $t$ for the next step. Moreover, by \eqref{eq:embed}, we have $u\in W^{\overline{t},p}_{\mathrm{loc}}(\Omega)$ for any $\overline{t}\leq t$.

\textbf{Step 2: }We define the number
\begin{equation}\label{choi.delta.thm7}
\delta=\delta(s_0,p,t,\alpha):=\min\left\{\frac{p-1}{2}\left(\frac{s_0p+t}{p-1}-\alpha\right),\frac{1}{2}\left(s_0-\frac{p-2}{p}t\right)\right\}>0
\end{equation}
and the sequence
\begin{equation*}
\widetilde{\gamma}_0=t\quad\text{and}\quad \widetilde{\gamma}_j=t\left(\frac1p\right)^{j}+\left(s_0-\delta+\frac{t}{p}\right)\sum_{k=0}^{j-1}\left(\frac1p\right)^{k}
\end{equation*}
for any $j\geq 1$. We record some facts for $\widetilde{\gamma}_j$.
\begin{itemize}
\item By $\delta\leq \frac{1}{2}(s_0-\frac{p-2}{p}t)$ from \eqref{choi.delta.thm7}, we observe $t\leq\widetilde{\gamma}_{j-1}<\widetilde{\gamma}_{j}$.
\item Also, for any $j\geq1$,
\begin{equation}\label{eq:widegamma}
\widetilde{\gamma}_j=\frac{\widetilde{\gamma}_{j-1}+t}{p}+s_0-\delta.
\end{equation}
\item Since 
\begin{align*}
\eqref{choi.delta.thm7}\,\,\implies\,\, \delta\leq\frac{p-1}{2}\left(\frac{s_0p+t}{p-1}-\alpha\right) \,\,\implies\,\, \alpha<\lim_{j\to\infty}\widetilde{\gamma}_j=\frac{s_0p+t-2\delta}{p-1},
\end{align*}
there exists ${j}_0$ such that $j_0=\min\{j\in\setN \cup \{0\}:\alpha<\widetilde{\gamma}_{j}\}$.
\end{itemize}

We now consider two cases.

\textbf{In case of $\frac{t+s_0p}{p-1}\leq1$: } Note that in this case $\alpha<1$, since $\alpha_0=\frac{t+s_0p}{p-1}\leq 1$ so that $\alpha\in(0,\alpha_0)$ holds. The proof is similar to (1)--(3).
\begin{itemize}
	\item If $j_0=0$, then $\alpha<\widetilde{\gamma}_0$. Thus by \textbf{Step 1}, we have $u\in W^{\alpha,p}_{\loc}(\Omega)$ with the estimate \eqref{est.sec.thm7}.

	\item If $j_0=1$, we have $\widetilde{\gamma}_0\leq\alpha<\widetilde{\gamma}_1$. Then by Lemma \ref{lem.deg.reg.ind} with the choice of $\gamma=t(=\widetilde{\gamma}_0)$ and $\sigma=\alpha<\min\{\widetilde{\gamma}_1,1\}=\min\{\frac{\widetilde{\gamma}_{0}+t}{p}+s_0-\delta,1\}$ from \eqref{eq:widegamma}, $u\in W^{\alpha,p}_{\loc}(\Omega)$ holds. In addition, covering arguments yields \eqref{est.sec.thm7}.
	\item If $j_0\geq 2$, then $\widetilde{\gamma}_0<\cdots<\widetilde{\gamma}_{i_0-1}\leq \alpha<\widetilde{\gamma}_{i_0}$ holds. Thus Lemma \ref{lem.deg.reg.ind} gives $u\in W^{\widetilde{\gamma}_1,p}_{\mathrm{loc}}(\Omega)$. By iterating this $j_0-2$ times along with \eqref{eq:widegamma}, we obtain $u\in W^{\widetilde{\gamma}_{j_0-1},p}_{\mathrm{loc}}(\Omega)$. Again Lemma \ref{lem.deg.reg.ind} for $\alpha<\widetilde{\gamma}_{j_0}$ yields $u\in W^{\alpha,p}_{\loc}(\Omega)$ with the estimate \eqref{est.sec.thm7}.
\end{itemize}

\textbf{In case of $\frac{t+s_0p}{p-1}>1$: } Note that in this case $\alpha>1$, since $\alpha_0=s_0+\frac{1+t}{p}> 1$ so that $\alpha\in(1,\alpha_0)$ holds. Since $\widetilde{\gamma}_0=t<1$, $\{\widetilde{\gamma}_j\}_{j\in\setN}$ is an increasing sequence, $\lim_{j\to\infty}\widetilde{\gamma}_j>\alpha>1$, there is a positive integer $j_0$ such that 
	\begin{align*}
	\widetilde{\gamma}_{j_0-1} \leq 1<\widetilde{\gamma}_{j_0}.
	\end{align*}
	By following the same lines as in the case when $\frac{t+s_0p}{p-1}\leq1$, we have $u\in W^{\gamma,p}_{\mathrm{loc}}(\Omega)$ for any $\gamma\in(0,1)$ as $1<\widetilde{\gamma}_{j_0}$. Moreover, we get
\begin{equation}\label{estg.sec.thm7}
    R^{\gamma-\frac{n}p}[u]_{W^{\gamma,p}(B_r(x_0))}\leq cE(u;B_{2r}(x_0))+cR^{\frac{sp+t}{p-1}-\frac{n}p}\|f\|^{\frac{1}{p-1}}_{\widetilde{W}^{t,p'}(B_{2r}(x_0))}
\end{equation}
for some constant $c=c(n,s_0,p,t,\gamma)$, whenever $B_{2r}(x_0)\Subset \Omega$.
	We now take 
	\begin{align*}
	\widetilde{\gamma}=\max\left\{{p\left(\frac{\alpha_1+\alpha}{2}-s_0\right)-t},t\right\}
	\end{align*}
	to see that 
	\begin{equation*}
	s_0+\frac{\widetilde{\gamma}+t}{p}\geq\frac{\alpha+\alpha_1}{2}.
	\end{equation*}
	Since $\widetilde{\gamma}\in [t,1)$ and $\alpha<\frac{\alpha+\alpha_1}{2}$, by Lemma \ref{lem.deg.reg.ind} with $\gamma=\tilde{\gamma}$ and with $\sigma=\alpha$ and by \eqref{estg.sec.thm7}, we have $\nabla u\in W^{\alpha-1,p}_{\mathrm{loc}}(\Omega)$ with \eqref{est.secg.thm7}.
The proof is completed.
\end{proof}

%We now prove Corollary \ref{cor.deg}.
%\begin{proof}[Proof of Corollary \ref{cor.deg}.]
%Let us fix $s_0\in(0,s]$ and $\alpha\in\left(0,\min\left\{\frac{s_0 p}{p-2},s_0+\frac{2}{p}\right\}\right)$. Now we divide the proof into two cases.

%\textbf{In case of $\frac{s_0 p}{p-2}\leq 1$: } We first observe 
%\begin{align*}
%    \frac{\frac{s_0p}{p-2}+s_0p}{p-1}=\frac{s_0p}{p-2}.
%\end{align*}
%Using this and \eqref{est.sec.thm7} in Theorem \ref{thm.regd}, we have
%\begin{align*}
%r^{\alpha-\frac{n}{p}}[u]_{W^{\alpha,p}(B_r(x_0))}\leq cE(u;B_{2r}(x_0))
%\end{align*}
%for some constant $c=c(n,s_0,p,\alpha)$, whenever $B_{2r}(x_0)\Subset\Omega$.

%\textbf{In case of $\frac{s_0 p}{p-2}> 1$: }By Theorem \ref{thm.regd} with $t=1$ and $f=0$, we have 
%\begin{align*}
%r^{\alpha-\frac{n}{p}} [\nabla u]_{W^{\alpha-1,p}(B_r(x_0))}\leq cE(u;B_{2r}(x_0))
%\end{align*}
%for some constant $c=c(n,s_0,p,\alpha)$, whenever $B_{2r}(x_0)\Subset\Omega$. This completes the proof.
%\end{proof}

We end this subsection with the following estimates in terms of the $V$-function.
\begin{corollary}
For any given $p\geq2$ and $s_0\in(0,s]$, let $u\in W^{s,p}(B_{2R}(x_0))\cap L^{p-1}_{sp}(\bbR^n)$ be a weak solution to 
\begin{equation}\label{eq.cor64}
(-\Delta_p)^su=0\quad\text{in }B_{2R}(x_0).
\end{equation}
If $s_0>\frac{p-2}{p}$, then we have
\begin{align*}
    (1-s)R^{sp}\dashint_{B_{R/2}(x_0)}\int_{B_{R/2}(x_0)}|\delta_hV(\delta_{x,y}^su)|^2\frac{\,dx\,dy}{|x-y|^n}\leq c|h|^2E(u;B_{R}(x_0))^p
\end{align*}
for some constant $c=c(n,s_0,p)$.
\end{corollary}
\begin{proof}
    By Theorem \ref{thm.regd}, we have
    \begin{align}\label{gineq}
        r\left(\dashint_{B_r(y_0)}|\nabla u|^p\,dx\right)^{\frac{1}{p}}\leq cE(u;B_{2r}(y_0))
    \end{align}
    for any $B_{2r}(y_0)\Subset B_{2R}(x_0)$, where $c=c(n,s_0,p)$. Let us fix $B_{10\rho}(y_0)\Subset B_{2R}(x_0)$ and define
    \begin{equation*}
        u_\rho (x)=u(\rho x+y_0)
    \end{equation*}
    to see that $ u_\rho\in W^{1,p}(B_{10})\cap L^{p-1}_{sp}(\bbR^n)$ is a weak solution to
    \begin{align*}
        (-\Delta_p)^su_\rho =0\quad\text{in }B_{5}.
    \end{align*}
    Then by Corollary \ref{el : lem.loc.deg}, $w(x)=(u_\rho\xi)(x)\in W^{1,p}(\bbR^n)$ is a weak solution to 
    \begin{align*}
        (-\Delta_p)^sw=g\quad\text{in }B_{2},
    \end{align*}
    where $\xi$ is the function defined in Corollary \ref{el : lem.loc.deg} with $R=1$ and $x_0=0$, and $g\in W^{1,p'}(B_{5/2})$ with the estimate 
    \begin{align}\label{ineq1.v11}
        \|\nabla g\|^{p'}_{L^{p'}(B_{5/2})}\leq (1-s)\|u_\rho\|^p_{W^{1,p}(B_3)}+c\mathrm{Tail}(u_{\rho};B_3)^p
    \end{align}
    for some constant $c=c(n,s_0,p)$. In addition, recalling $w(x)=(u_\rho\xi)(x)$, we observe from \eqref{norm.loc} that
    \begin{align}\label{ineq1.v12}
        [w]_{W^{s,p}(\bbR^n)}^p\leq c\|u_\rho\|_{W^{s,p}(B_5)}^p.
    \end{align}
    Also, 
    \begin{align}\label{ineq1.v13}
        \|\nabla w\|_{L^p(\bbR^n)}\leq c\|u_\rho\|_{W^{1,p}(B_5)}^p
    \end{align}
    for some constant $c=c(n,s_0,p)$.

    By Lemma \ref{lem.diffquo} and \eqref{ineq.vpn}, we have 
    \begin{align*}
        &(1-s)\int_{B_{1/2}}\int_{B_{1/2}}|\delta_hV(\delta_{x,y}^sw)|^2\frac{\,dx\,dy}{|x-y|^n}\\
        &\leq c\left(\left((1-s)[w]_{W^{s,p}(\bbR^n)}\right)^{p-2}\|\delta_h w\|_{L^p(\bbR^n)}^2+\|\delta_h w\|_{L^p(\bbR^n)}^p\right)\\
        &\quad+\|\delta_h w\|_{L^p(B_2)}\|\delta_h g\|_{L^{p'}(B_2)}\eqqcolon I
    \end{align*}
    for some constant $c=c(n,s_0,p)$. Since 
    \begin{align*}
        \left\|\frac{\delta_h w}{|h|}\right\|_{L^p(\bbR^n)}^2\leq \|\nabla w\|_{L^p(\bbR^n)},
    \end{align*}
    we further estimate $I$ as
    \begin{align*}
       &(1-s)\int_{B_{1/2}}\int_{B_{1/2}}|\delta_hV(\delta_{x,y}^sw)|^2\frac{\,dx\,dy}{|x-y|^n}\\
       &\leq |h|^2\left[c(1-s)[w]^{p}_{W^{s,p}(\bbR^n)}+c\|\nabla w\|^p_{L^p(\bbR^n)}+c\|\nabla g\|^p_{L^{p'}(B_2)}\right]
    \end{align*}
    for some constant $c=c(n,s_0,p)$, where we have used Young's inequality. Using \eqref{ineq1.v11}--\eqref{ineq1.v13} and the change of variables yield
    \begin{align*}
        &(1-s)\rho^{sp}\dashint_{B_{\rho/2}(y_0)}\int_{B_{\rho/2}(y_0)}\frac{|\delta_hV(\delta_{x,y}^su)|^2}{|h|^2}\frac{\,dx\,dy}{|x-y|^n}\\
        &\leq c\rho^{sp-n}(1-s)[u]_{W^{s,p}(B_{5\rho}(y_0))}^p+c\rho^{p-n}\|\nabla u\|^p_{L^p(B_{5\rho}(y_0))}\\
        &\quad+c\rho^{-n}\| u\|^p_{L^p(B_{5\rho}(y_0))}+c\mathrm{Tail}(u;B_5)^p
    \end{align*}
    for some constant $c=c(n,s_0,p)$. By \eqref{gineq}, Lemma \ref{lem.upest} and the fact that $u-(u)_{B_{10\rho}(y_0)}$ is a weak solution to \eqref{eq.cor64}, we get
    \begin{align*}
        &(1-s)\rho^{sp}\dashint_{B_{\rho/2}(y_0)}\int_{B_{\rho/2}(y_0)}|\delta_hV(\delta_{x,y}^su)|^2\frac{\,dx\,dy}{|x-y|^n}\leq c|h|^2E(u;B_{10\rho}(y_0))^p,
    \end{align*}
    where $c=c(n,s_0,p)$. Standard covering arguments yield the desired estimate.
\end{proof}

\subsection{The singular case}
We now prove an analogous version of Lemma \ref{lem.degind} when $p\leq2$. 
\begin{lemma}\label{lem.sinind}
    Let $p\leq 2$ and $\gamma\in[s,1)$. Let $u\in W^{\gamma,p}(\bbR^{n})$ be a weak solution to 
    \begin{equation}
    \label{eq.itesub}
        (-\Delta_p)^su=f\quad\text{in }B_{2}
    \end{equation}
    with $f\in L^{p'}(B_2)$. For any $|h|<1/1000$, we have 
    \begin{equation}\label{ineq.lem.sinind}
    \begin{aligned}
        &(1-\varsigma)\left[{\delta_h (u\psi)}\right]^{p}_{W^{\varsigma,p}(\bbR^n)}\\
        &\leq  c[(1-\varsigma)^{\frac1p}u]_{W^{\gamma,p}(\bbR^n)}^{\frac{p(2-p)}{2}}\left(\left\|{\delta_h u}\right\|_{L^{p}(\bbR^{n})}^{\frac{p^{2}}{2}}+\left\|{\delta_h u}\right\|^{\frac{p}{2}}_{L^{p}(B_1)}\left\|{\delta_h f}\right\|^{\frac{p}{2}}_{L^{p'}(B_{1})}\right)\\
        &\quad+c\left\|{\delta_h u}\right\|_{L^{p}(\bbR^{n})}^{p}+c|h|^p\|u\|_{L^{p}(\bbR^n)}^p+c|h|^p(1-\varsigma)[u]_{W^{\varsigma,p}(\bbR^n)}^p
    \end{aligned}
    \end{equation}
    for some constant $c=c(n,s_0,p)$, where $\varsigma=\frac{sp}{2}+\frac{(2-p)\gamma}{2}$, the function $\psi$ is determined in Lemma \ref{lem.combvest} and the constant $s_0$ is determined in \eqref{s0assum}.
\end{lemma}
\begin{proof}
We first note from \eqref{ineq.simple0} that 
\begin{equation}\label{ineq.sub110}
\begin{aligned}
    \int_{\bbR^n}\int_{\bbR^n}\frac{|\delta^{\varsigma}_{x,y}\delta_h (u\psi)|^{p}}{|x-y|^{n}}\,dx\,dy&\leq \int_{\bbR^n}\int_{\bbR^n}\frac{|\delta^{\varsigma}_{x,y}\delta_h u|^{p}\psi_M^p}{|x-y|^{n}}\,dx\,dy+\frac{c}{1-\varsigma}\|\delta_hu\|_{L^p(\bbR^n)}^p\\
    &\quad +c|h|^p[u]^p_{{W}^{\varsigma,p}(\bbR^n)}+\frac{c}{1-\varsigma}|h|^p\|u\|^p_{L^{p}(\bbR^n)}
\end{aligned}
\end{equation}
for some constant $c=c(n,s_0,p)$, where we denote 
\begin{equation*}
   \psi_M= \max\{\psi(x),\psi(y)\}.
\end{equation*}
Then H\"older's inequality and Lemma \ref{el : lem.eleineq} yields
\begin{equation}\label{ineq.sub111}
\begin{aligned}
    &\int_{\bbR^n}\int_{\bbR^n}\frac{|\delta^{\varsigma}_{x,y}\delta_h u|^{p} \psi_M^p}{|x-y|^{n}}\,dx\,dy\\
    &=\int_{\bbR^n}\int_{\bbR^n}\frac{|\delta^s_{x,y}\delta_h u|^{p}(|\delta^s_{x,y}u|+|\delta^s_{x,y}u_h|)^{\frac{p(p-2)}{2}}(|\delta^\gamma_{x,y}u|+|\delta^\gamma_{x,y}u_h|)^{\frac{p(2-p)}{2}}\psi_M^p}{|x-y|^{n}}\,dx\,dy\\
    &\leq \Bigg(\int_{\bbR^n}\int_{\bbR^n}(|\delta^s_{x,y}u|+|\delta^s_{x,y}u_h|)^{p-2}{|\delta^s_{x,y}\delta_h u|^{2}\psi_M^2}\frac{\,dx\,dy}{|x-y|^{n}}\Bigg)^{\frac{p}{2}}[u]_{W^{\gamma,p}(B_{1})}^{\frac{p(2-p)
    }{2}}\\
    &\leq  \Bigg(\int_{\bbR^n}\int_{\bbR^n}\left|\delta_h V\left(\delta^s_{x,y}u\right)\right|^{2}\psi_M^2\frac{\,dx\,dy}{|x-y|^{n}}\Bigg)^{\frac{p}{2}}[u]_{W^{\gamma,p}(B_{1})}^{\frac{p(2-p)
    }{2}}.
\end{aligned}
\end{equation}
We next note from \eqref{ineq2.shif} and H\"older's inequality that
\begin{equation}\label{ineq.sub112}
\begin{aligned}
    &\int_{\bbR^n}\int_{\bbR^n}\phi_{|\delta_{x,y}^su|}\left({\left(|\delta_{h}u(x)|+|\delta_hu(y)|\right)} \frac{|\psi(x)-\psi(y)|}{|x-y|^{s}}\right)\frac{\,dx\,dy}{|x-y|^{n}}\\
    &\leq c\int_{\bbR^n}\int_{\bbR^n}\left({|\delta_{h}u(x)|^p+|\delta_{h}u(y)|^p}\right) \frac{|\psi(x)-\psi(y)|^p}{|x-y|^{sp}}\frac{\,dx\,dy}{|x-y|^{n}}\\
     &\leq \frac{c}{1-s}\|\delta_h u\|_{L^p(\bbR^n)}^p
\end{aligned}
\end{equation}
for some constant $c=c(n,s_0,p)$.
Therefore, employing the inequality
\begin{align*}
\int_{B_1}\abs{\delta_h u}\abs{\delta_h f}\,dx\leq\norm{\delta_h u}_{L^p(B_1)}\norm{\delta_h f}_{L^{p'}(B_1)},
\end{align*}
and combining all the estimates \eqref{ineq.sub110}--\eqref{ineq.sub112} along with Lemma \ref{lem.diffquo} and the fact that $\varsigma\geq s$, \eqref{ineq.lem.sinind} is obtained.
\end{proof}
Using Lemma \ref{lem.sinind}, we obtain an analogous version of Lemma \ref{lem.deg.reg.ind} when $p\leq2$.

\begin{lemma}\label{lem.sin.reg.ind}
Let $p\leq2$ and $\gamma\in[s,1)$. Suppose that $u\in W^{\gamma,p}(B_{5R}(x_0))\cap L^{p-1}_{sp}(\bbR^n)$ be a weak solution to 
\begin{equation*}
    (-\Delta_p)^su=f\quad\text{in }B_{5R}(x_0),
\end{equation*}
where $f\in W^{t,p'}(B_{5R}(x_0))$ for some constant $t\in(0,p-1)$. Let us denote 
\begin{equation*}
    \varsigma_0=\frac{(2-p)\gamma}{2}+\frac{s_0p}{2}.
\end{equation*}
Then there exists $c=c(n,s_0,p,t,\sigma,\gamma)\geq 1$ such that the following holds.
\begin{itemize}
	\item If $\frac{\gamma+\min\{(p-1)\gamma,t\}}{2}+\varsigma_0\leq 1$, then for any $\sigma<\frac{\gamma+\min\{(p-1)\gamma,t\}}{2}+\varsigma_0$, we have 
	\begin{align*}
	R^{\sigma-\frac{n}{p}}&[u]_{W^{\sigma,p}(B_{R/2}(x_0))}\\
	&\leq c(1-s)^{\frac1p}R^{\gamma-\frac{n}{p}}[u]_{W^{\gamma,p}(B_{5R}(x_0))}+c\mathrm{Tail}(u-(u)_{B_{5R}(x_0)};B_{5R}(x_0))\\
	&\quad+cR^{\frac{sp+t}{p-1}-\frac{n}{p}}[f]_{W^{t,p'}(B_{5R}(x_0))}^{\frac{1}{p-1}}=:\mathcal{R}.
	\end{align*}
	\item If $\frac{\gamma+\min\{(p-1)\gamma,t\}}{2}+\varsigma_0>1$, then for any $\sigma\in\left(1,\frac{\gamma+\min\{(p-1)\gamma,t\}}{2}+\varsigma_0\right)$, we have 
	\begin{align*}
	R^{\sigma-\frac{n}{p}}[\nabla u]_{W^{\sigma-1,p}(B_{R/2}(x_0))}^p\leq\mathcal{R}.
	\end{align*}
\end{itemize}
\end{lemma}
\begin{proof}
    By Lemma \ref{lem.scale}, we may assume $x_0=0$ and $R=1$. Let us fix $\sigma<\varsigma_0+\frac{\gamma+\min\{(p-1)\gamma,t\}}{2}$ and choose
    \begin{equation*}
        \sigma_0=\frac12\left(\sigma+\varsigma_0+\frac{\gamma+\min\{(p-1)\gamma,t\}}{2}\right).
    \end{equation*}
    Then there is a constant $t_0<\min\{(p-1)\gamma,t\}$ such that 
    \begin{equation*}
        \varsigma_0+\frac{\gamma+t_0}{2}=\sigma_0.
    \end{equation*}
    We next choose a cutoff function $\xi\in C_c^{\infty}(B_{7/2})$ such that $|\nabla \xi|\leq4$. By Lemma \ref{el : lem.loc.sin}, $w=u\xi$ is a weak solution to
    \begin{equation*}
        (-\Delta_p)^sw=f+g\quad\text{in }B_2,
    \end{equation*}
    where $g\in W^{\gamma(p-1),p'}(B_{5/2})$ with the estimate
    \begin{align}\label{ineq.gts}
        [g]^{p'}_{W^{\gamma(p-1),p'}(B_{5/2})}&\leq c(1-s)^{p'}\|u\|_{W^{\gamma,p}(B_{3})}^p+c\mathrm{Tail}(u;B_{3})^p
    \end{align}
    for some constant $c=c(n,s_0,p,\gamma)$.

    Using Lemma \ref{lem.sinind} and the fact that for any $|h|<1/1000$, $|h|^{p\gamma}\leq c|h|^{\frac{p\gamma}{2}}\leq c|h|^{\frac{\gamma+t_0}{2}}$ and $|h|\leq c|h|^{\frac{\gamma+t_0}{2}}$ for some constant $c$, we deduce 
    \begin{equation}\label{ineq1.itscs}
    \begin{aligned}
    &(1-\varsigma_0)\left[\frac{\delta_h (w\psi)}{|h|^{\frac{\gamma+t_0}{2}}}\right]_{W^{\varsigma_0,p}(\bbR^n)}^{p}\\
    &\leq  c[(1-\varsigma_0)^{\frac1p}w]_{W^{\gamma,p}(\bbR^n)}^{\frac{p(2-p)}{2}}\left(\left\|\frac{\delta_h w}{|h|^{\gamma}}\right\|_{L^{p}(\bbR^{n})}^{\frac{p^{2}}{2}}+\left\|\frac{\delta_h w}{|h|^{\gamma}}\right\|^{\frac{p}{2}}_{L^{p}(B_1)}\left\|\frac{\delta_h F}{|h|^{t_0}}\right\|^{\frac{p}{2}}_{L^{p'}(B_{1})}\right)\\
    &\quad+c\left\|\frac{\delta_h w}{|h|^{\gamma}}\right\|_{L^{p}(\bbR^{n})}^{p}+c\|w\|_{L^{p}(\bbR^n)}^p+c(1-\varsigma_0)[w]_{W^{\varsigma_0,p}(\bbR^n)}^p,
\end{aligned}
\end{equation}
where the function $\psi$ is determined in Lemma \ref{lem.sinind} and $F=f+g$. Let us note from \eqref{lem.embdiff.a1} in Lemma \ref{lem.embdiff} and Lemma \ref{lem.poist},
    \begin{align}\label{ineq0.itscs}
        \sup_{0<|h|<\frac{1}{1000}}\left\|\frac{\delta_h F}{|h|^{t_0}}\right\|_{L^{p'}(B_1)}^{p'}&\leq  c\sup_{0<|h|<\frac{1}{1000}}\left\|\frac{\delta_h (f-(f)_{B_2})}{|h|^{t}}\right\|_{L^{p'}(B_1)}\nonumber\\
        &\quad+c\sup_{0<|h|<\frac{1}{1000}}\left\|\frac{\delta_h (g-(g)_{B_2})}{|h|^{(p-1)\gamma}}\right\|_{L^{p'}(B_2)}^{p'}\nonumber\\
        &\leq c(1-t)[f]^{p'}_{W^{t,p'}(B_2)}+(1-(p-1)\gamma)[g]_{W^{(p-1)\gamma,p'}(B_2)}^{p'}
    \end{align}
    for some constant $c=c(n,s_0,p,\gamma,t)$.
    We now use \eqref{lem.embdiff.a2} and \eqref{lem.embdiff.a1} in Lemma \ref{lem.embdiff} to see that
    \begin{align*}
        \sup_{0<|h|<\frac{1}{1000}}\left\|\frac{\delta^2_h (w\psi)}{|h|^{\varsigma_0+\frac{\gamma+t_0}{2}}}\right\|_{L^p(\bbR^n)}^p&\leq c(1-\varsigma_0)\left[\frac{\delta_h (w\psi)}{|h|^{\frac{\gamma+t_0}{2}}}\right]_{W^{\varsigma_0,p}(\bbR^n)}^{p}+c\left\|\frac{\delta_h (w\psi)}{|h|^{\gamma}}\right\|_{L^{p}(\bbR^n)}^{p}\\
        &\leq c(1-\varsigma_0)\left[\frac{\delta_h (w\psi)}{|h|^{\frac{\gamma+t_0}{2}}}\right]_{W^{\varsigma_0,p}(\bbR^n)}^{p}+c(1-\gamma)[w\psi]_{W^{\gamma,p}(\bbR^n)}^p
    \end{align*}
    for some constant $c=c(n,s_0,p,\gamma,\sigma)$.
    Thus, combining this, \eqref{ineq1.itscs} and \eqref{ineq0.itscs} yields
    \begin{align*}
        \sup_{0<|h|<\frac{1}{1000}}\left\|\frac{\delta^2_h (w\psi)}{|h|^{\varsigma_0+\frac{\gamma+t_0}{2}}}\right\|_{L^p(\bbR^n)}^p&\leq c(1-\varsigma_0)([w]_{W^{\gamma,p}(\bbR^n)}^p+[w]_{W^{\sigma,p}(\bbR^n)}^p)\\
        &\quad +c\|w\|_{L^p(\bbR^n)}^p+c[f]_{W^{t,p'}(B_2)}^{p'}+c[g]_{W^{(p-1)\gamma,p'}(B_2)}^{p'}.
    \end{align*}
    We now use the fact that $w=u\xi$, \eqref{ineq.gts} and \eqref{norm.loc} to further estimate the right-hand side of the above display as
    \begin{align*}
        \sup_{0<|h|<\frac{1}{1000}}\left\|\frac{\delta^2_h (w\psi)}{|h|^{\varsigma_0+\frac{\gamma+t_0}{2}}}\right\|_{L^p(\bbR^n)}^p&\leq c(1-s)[u]_{W^{\gamma,p}(B_5)}^p+c\|u\|_{L^p(B_5)}^p\\
        &\quad+c\mathrm{Tail}(u;B_5)^p+c[f]_{W^{t,p'}(B_2)}^{p'}
    \end{align*}
    for some constant $c=c(n,s_0,p,\gamma,\sigma)$.
    In light of \eqref{ineq.seq1} in Lemma \ref{lem.secmeb} along with the fact that $w=u\xi$, we get
    \begin{align*}
        [u]_{W^{\sigma,p}(B_{1/2})}^p&\leq c(1-s)[u]_{W^{\gamma,p}(B_5)}^p+c\|u\|_{L^p(B_5)}^p+c\mathrm{Tail}(u;B_5)^p+c[f]_{W^{t,p'}(B_2)}^{p'},
    \end{align*}
    whenever $\varsigma_0+\frac{\gamma+t_0}{2}\leq1$. Similarly, if $\varsigma_0+\frac{\gamma+t_0}{2}>1$, then we have 
    \begin{align*}
        [\nabla u]_{W^{\sigma-1,p}(B_{1/2})}^p&\leq c(1-s)[u]_{W^{\gamma,p}(B_5)}^p+c\|u\|_{L^p(B_5)}^p+c\mathrm{Tail}(u;B_5)^p+c[f]_{W^{t,p'}(B_2)}^{p'}
    \end{align*}
    for some constant $c=c(n,s_0,p,\gamma,\sigma)$. Using the fact that $u-(u)_{B_5}$ is a also weak solution to 
    \begin{equation*}
        (-\Delta_p)^su =f\quad\text{in }B_5
    \end{equation*} and using Lemma \ref{lem.poist}, we get the desired estimate.
\end{proof}
We now prove Theorem \ref{thm.regs}.
\begin{proof}[Proof of Theorem \ref{thm.regs}.]
For $p\leq 2$, let us fix $s_0\in(0,s]$ and $t\in\left(0,p-1\right)$.

We take a sequence
    \begin{align*}
        \gamma_i=\frac{s_0p}{4}i\quad\text{for each }i\geq0.
    \end{align*}
    Then there is a nonnegative integer $i_0$ such that 
    \begin{align*}
    \begin{cases}
    t<(p-1)\gamma_0&\quad\text{if }i_0=0,\\
    (p-1)\gamma_{i_0}\leq t<(p-1)\gamma_{i_0+1}&\quad\text{if }i_0>0.
  	\end{cases}   
    \end{align*}
    Since $\gamma_0<s$, we observe $u\in W^{\gamma_0,p}_{\mathrm{loc}}(\Omega)$. Moreover, using \eqref{eq:embed} and Lemma \ref{lem.upest}, we observe that the estimate 
\begin{equation}\label{ineq.fir.thm9}
    R^{\gamma_0-\frac{n}p}[u]_{W^{\gamma_0,p}(B_r(x_0))}\leq cE(u;B_{2r}(x_0))+cR^{\frac{sp+t}{p-1}-\frac{n}p}\|f\|^{\frac{1}{p-1}}_{\widetilde{W}^{t,p'}(B_{2r}(x_0))}
\end{equation}
holds for some constant $c=c(n,s_0,p,t)$, whenever $B_{2r}(x_0)\Subset \Omega$. Using Lemma \ref{lem.sin.reg.ind} and standard covering arguments, we inductively prove $u\in W^{\gamma_{i_0},p}_{\mathrm{loc}}(\Omega)$ and obtain the estimate \eqref{ineq.fir.thm9} with $\gamma_0$ replaced by $\gamma_{i_0}$. Since 
    \begin{equation*}
        \frac{t}{p-1}<\gamma_{i_{0}+1}\leq\frac{\gamma_{i_0}+\min\{(p-1)\gamma_{i_0},t\}}{2}+\frac{(2-p)\gamma_{i_0}}{2}+\frac{s_0 p}{2},
    \end{equation*}
    applying Lemma \ref{lem.sin.reg.ind} yields $u\in W^{\frac{t}{p-1},p}_{\mathrm{loc}}(\Omega)$.
    Let us define 
    \begin{equation}\label{choi.sin.ep}
        \epsilon=\min\left\{\frac{(p-1)}{4}\left(\frac{\alpha+\alpha_2}{2}-\alpha\right),\frac{s_0p}{4}\right\}.
    \end{equation}
    We next consider the sequence 
    \begin{align*}
        \widetilde{\gamma}_0=\frac{t}{p-1}\quad\text{and}\quad\widetilde{\gamma}_i=\widetilde{\gamma}_0\left(\frac{3-p}{2}\right)^i+\left(\frac{s_0p+t}{2}-\epsilon\right)\sum_{k=0}^{i-1}\left(\frac{3-p}{2}\right)^k
    \end{align*}
    and observe that $\widetilde{\gamma}_{i}\leq \widetilde{\gamma}_{i+1}$ and $\widetilde{\gamma}_i\geq \frac{t}{p-1}$ for any $i\geq0$, as well as
    \begin{align}\label{lim.sin.widegam}
        \lim_{i\to\infty}\widetilde{\gamma}_i=\frac{s_0p+t}{p-1}-\frac{2\epsilon}{p-1}>\frac{\alpha_2+\alpha}{2}.
    \end{align}
    \begin{enumerate}
        \item $\frac{s_0p+t}{p-1}\leq1$. By \eqref{lim.sin.widegam} along with \eqref{choi.sin.ep}, there is a positive integer $i_0$ such that 
        \begin{equation*}
            \widetilde{\gamma}_{i_0}\leq \alpha<\widetilde{\gamma}_{i_0+1}.
        \end{equation*}
        By iterating $i_0$ times Lemma \ref{lem.sin.reg.ind} together with covering arguments, we obtain $u\in W^{\widetilde{\gamma}_{i_0},p}_{\mathrm{loc}}(\Omega)$ with the estimate 
\begin{equation}\label{ineq2.fir.thm9}
    R^{\widetilde{\gamma}_{i_0}-\frac{n}p}[u]_{W^{\widetilde{\gamma}_{i_0},p}(B_r(x_0))}\leq cE(u;B_{2r}(x_0))+cR^{\frac{sp+t}{p-1}-\frac{n}p}\|f\|^{\frac{1}{p-1}}_{\widetilde{W}^{t,p'}(B_{2r}(x_0))}
\end{equation}
for some constant $c=c(n,s_0,p,t,\alpha)$, whenever $B_{2r}(x_0)\Subset \Omega$. Thus by applying Lemma \ref{lem.sin.reg.ind} and standard covering arguments along with \eqref{ineq2.fir.thm9}, $u\in W^{\alpha,p}_{\mathrm{loc}}(\Omega)$ holds with \eqref{est.sec.thm9}.
        \item $\frac{s_0p+t}{p-1}>1$. Since $\frac{\alpha_2+\alpha}{2}>1$, there is a positive integer $j_0$ such that 
        \begin{align*}
            \widetilde{\gamma}_{j_0}\leq 1<\widetilde{\gamma}_{j_0+1}.
        \end{align*}
        As in the case when $\frac{s_0p+t}{p-1}\leq1$, we get $u\in W^{\widetilde{\gamma},p}_{\mathrm{loc}}(\Omega)$ with 
        \begin{equation*}
            \widetilde{\gamma}=\max\left\{\frac{1}{3-p}\left(\frac{s_0p+t}{2}-(1+\alpha)\right),\frac{t}{p-1}\right\}<1.
        \end{equation*} In addition, we have \eqref{ineq2.fir.thm9} with $\widetilde{\gamma}_{i_0}$ replaced by $\widetilde{\gamma}$.
        Since 
        \begin{equation*}
            \frac{s_0p}{2}+\frac{(3-p)\widetilde{\gamma}}{2}+\frac{t}{2}>\alpha,
        \end{equation*}
        by Lemma \ref{lem.sin.reg.ind} we arrive at $\nabla u\in W^{\alpha-1,p}_{\mathrm{loc}}(\Omega)$ with \eqref{est.secg.thm9}.
    \end{enumerate}
Then the proof is completed.
\end{proof}
Finally, we prove Corollary \ref{cor.regs}, that is, our main result in the homogeneous subquadratic case.
\begin{proof}[Proof of Corollary \ref{cor.regs}.]
    Let us fix $s_0\in(0,s]$ and 
    \begin{equation*}
    \alpha\in \left(1,\max\{1+\frac{s_0 p}{2},2+(s_0-1)p\}\right).
    \end{equation*}
     Suppose {$1+\frac{s_0p}{2}\leq 2+(s_0-1)p$ so that}
    $\max\{1+\frac{s_0 p}{2},2+(s_0-1)p\}=2+(s_0-1)p$. In this case, the estimate \eqref{est.sin.grad}
    follows from Theorem \ref{thm.1} with $f=0$. On the other hand, if {$2+(s_0-1)p\leq 1+\frac{s_0p}{2}$ so that} $\max\{1+\frac{s_0 p}{2},2+(s_0-1)p\}=1+\frac{s_0 p}{2}$, Theorem \ref{thm.regs} with $f=0$ yields \eqref{est.sin.grad}.
    This completes the proof.
\end{proof}

\appendix
\section{Power functions}\label{app.power}
In this appendix, we prove some properties of power functions that are useful to prove the sharpness of our results. Let us assume $\sigma \in(0,1)$ and $p\in[1,\infty)$.
By \cite[page 44, Example 1]{RunSic96}, we observe
\begin{align}\label{goal1.app}
    u(x)=|x|^{\sigma-\frac{n}{p}}\in W^{s,p}_{\mathrm{loc}}(\bbR^n)\quad\text{if }s<\sigma
\end{align}
and
\begin{align}\label{goal2.app}
    u(x)=|x|^{\sigma-\frac{n}{p}}\notin W^{s,p}(B_1)\quad\text{if }s\geq\sigma.
\end{align}

We next prove that the function 
\begin{equation}\label{defn.u}
    u(x)=|x|^{sp'+\epsilon-\frac{n}{p}}\in W^{sp'+\epsilon/2,p}_{\mathrm{loc}}(\bbR^n)\cap L^{p-1}_{sp}(\bbR^n)
\end{equation}is a weak solution to 
\begin{align}\label{eq11}
    (-\Delta_p)^su=f\quad\text{in }B_1,
\end{align}
where $f(x)=c|x|^{(p-1)\left(\epsilon-\frac{n}{p}\right)}$ for some constant $c$. 
By \cite[Lemma 2.3]{IanMosSqu16}, observe that 
\begin{align*}
    \int_{\bbR^n}f\phi\,dx=\int_{\bbR^n}(-\Delta_p)^su(x)\phi(x)\,dx<\infty
\end{align*}
for any $\phi\in C_{c}^{\infty}(B_1)$, which implies that $(-\Delta_p)^su(x)$ is well defined a.e. In addition, there is a point $x_0\in B_1\setminus\{0\}$ such that $(-\Delta_p)^su(x_0)<\infty$.
We next observe that
\begin{align*}
    (-\Delta_p)^su(rx)=r^{(p-1)\left(\epsilon-\frac{n}{p}\right)}(-\Delta_p)^su(x)
\end{align*}
for any $x\in \bbR^n$, $r>0$ and $ (-\Delta_p)^su(x)$ is a radially symmetric function.
 From this, we deduce
\begin{align*}
    (-\Delta_p)^su(x)=c|x|^{(p-1)\left(\epsilon-\frac{n}{p}\right)}\in L^{p'}_{\mathrm{loc}}(\bbR)
\end{align*}
for some constant $c$. Therefore, $u$ is a weak solution to \eqref{eq11}.

%\begin{coi} \normalfont
%	The authors declare that there is no conflict of interest.
%\end{coi}
%\begin{das} \normalfont
%	Data sharing is not applicable to this article as no datasets were generated or analysed during
%	the current study.
%\end{das}

%%%%%%%%%%%%%%

\printbibliography

\end{document}